\documentclass[11pt]{article}
\usepackage{fullpage}
\usepackage{amsmath}
\usepackage{amsthm}
\usepackage{amssymb}

\usepackage{color,graphics}
\usepackage{enumerate}

\theoremstyle{plain}                    

\newtheorem{theorem}{Theorem}[section]
\newtheorem{lemma}[theorem]{Lemma}

\newtheorem{corollary}[theorem]{Corollary}

\theoremstyle{definition}

\newtheorem{definition}[theorem]{Definition}

\theoremstyle{remark}
\newtheorem{remark}[theorem]{Remark}

\newcommand{\pr}{\ensuremath{\mathbb{P}}}
\newcommand{\R}{\ensuremath{\mathcal{R}}}
\newcommand{\B}{\ensuremath{\mathcal{B}}}

\newcommand{\torus}[2]{\ensuremath{\mathbb{T} (#1,#2)}}
\newcommand{\red}{\textit{red}}
\newcommand{\blue}{\textit{blue}}

\numberwithin{equation}{section}

\title{Competing first passage percolation on random regular graphs
}

\author{Ton\'{c}i Antunovi\'{c} \\ U.C. Berkeley and UCLA \\ {\tt tantunovic@math.ucla.edu}
\and Yael Dekel \\ The Hebrew University \\ {\tt yaelvin@cs.huji.ac.il}
\and Elchanan Mossel \\ U.C. Berkeley and Weizmann institute \\ {\tt mossel@stat.berkeley.edu}
\and Yuval Peres \\ Microsoft Research \\ {\tt peres@microsoft.com}}

\date{}

\begin{document}
\maketitle

\begin{abstract}
We consider two competing first passage percolation processes started
from uniformly chosen subsets of a random regular graph on $N$
vertices. The processes are allowed to spread with different rates,
start from vertex subsets of different sizes or at different times. We
obtain tight results regarding the sizes of the vertex sets occupied
by each process, showing that in the generic situation one process
will occupy $\Theta(1)N^\alpha$ vertices, for some $0 < \alpha <
1$. The value of $\alpha$ is calculated in terms of the relative rates
of the processes, as well as the sizes of the initial vertex sets and
the possible time advantage of one process.

The motivation for this work comes from the study of viral marketing
on social networks. The described processes can be viewed as two
competing products spreading through a social network (random regular
graph). Considering the processes which grow at different rates
(corresponding to different attraction levels of the two products) or
starting at different times (the first to market advantage) allows to
model aspects of real competition. The results obtained can be
interpreted as one of the two products taking the lion share of the
market. We compare these results to the same process run on $d$
dimensional grids where we show that in the generic situation the two
products will have a linear fraction of the market each.
\end{abstract}

\section{Introduction}

First passage percolation is one of the most studied discrete models in probability theory. 
It can be realized as a random graph metric when edges have independent identically distributed weights. 
Often the distribution is assumed to be exponential and then the ball of a radius $t$ (from a fixed vertex) is a Markov set process $\mathcal{R}$, in which new vertices are occupied at a rate proportional to the number of their neighbors already in $\mathcal{R}(t)$.
Apart from the classical shape problem on infinite transitive graphs (see \cite{CoxDurrett81}), recently there was substantial interest in estimating diameter, typical distance, flooding times and related quantities for the 
  process on large finite (and possibly random) graphs \cite{Janson97, vdHHvM02, Bhamidi08, BvdHH10a, BvdHH10b, BvdHH}.

In a related two type Richardson model, introduced in \cite{HaggstromPemantle98}, one considers two first passage percolation processes, a blue and a red one, with possibly different rates, spreading through the graph and capturing non-colored vertices. Each non-colored vertex becomes colored with color $c$ at the rate proportional to the number of $c$ colored neighbors (this can also be viewed as the Voronoi tessellation with respect to two independent first passage percolation metric). A significant amount of work on this model has been devoted to identifying the cases in which both colors grow indefinitely \cite{HaggstromPemantle00, DeijfenHaggstrom06A, DeijfenHaggstrom06B, DeijfenHaggstrom07, GaretMarchand08, Hoffman05, Hoffman08}. 
%
In the current paper we are studying a version of this model on large random regular graphs (which are objects of independent interest \cite{Bollobas01, JLR00}). We are interested in the sizes of each colored component, while allowing the processes to start at different times, from sets of different sizes and spread with different rates.

From an applied point of view, this model can be viewed to simulate spreading of two products (or viruses) through a social network. In recent years, diffusion processes on social networks have been the
focus of intense study in a variety of areas. Traditionally these
processes have been of major interest in epidemiology where they model
the spread of diseases and
immunization~\cite{Morris:04,Liggett:85,Liggett:01,Durrett88,BeBoChSa:05,DuJu:u}.
Much of the recent interest has resulted from applications in
sociology, economics, and engineering~\cite{BrownReinegen:87,
  AsRoLeVe:01, GoLiMu:01a,GoLiMu:01b,
  DomingosRichardson:01,DomingosRichardson:02, KeKlTa:03,KeKlTa:05}.

The interpretations of the diffusion process in terms of product
marketing and in terms of virus spread lead to some natural questions
we address in this paper.
What is the advantage that the first product (the first virus) has in terms of the initial time it can spread with no competition? What is the effect of one of them starting with larger initial size (initial seed sets) than the other one or having a larger rate (higher quality of a product)? What is the effect of the structure of the social network on the outcome of the competition between the two products?
%
%
%
%
%
%
 To answer the last question we compare the results for the model on large random regular graphs to the same model on
 large $d$ dimensional tori.
 The first family of graphs model some (but
 not all) features of current social networks (small diameter,
 expansion etc.) while the second family models traditional spatial
 graph processes that are traditionally studies in epidemiology,
 ecology and statistical physics. We proceed with a formal definition
 of the process and a statement of our main results.

\subsection{Description of the process and the results}
Let $G=(V,E)$ be a graph with $|V|$ vertices and $|E|$ edges,
and let $\mathcal{B}_0$ and $\mathcal{R}_0$ be  disjoint sets of vertices (we think of $\B_0$ as a set of \blue\ vertices and of $\R_0$ as a set of \red\ vertices). Denote by $N(v)$ the set of neighbors of $v$.
Competing first passage percolation (CFPP) considered in this paper is a Markov process, whose state space is the family of subsets of $V$, which evolves by coloring an uncolored vertex blue (red) at the rate equal to $\beta$ ($\rho$) times the number of neighbors of $v$ which are already blue (red).
That is, at any time $t$, each vertex $v \notin \mathcal{B}_t \cup \mathcal{R}_t$ becomes an element of $\mathcal{B}_t$ at the rate equal to $\beta |N(v) \cap \mathcal{B}_t|$, and an element of $\mathcal{R}_t$ at the rate equal to $\rho |N(v) \cap \mathcal{B}_t|$. Here $\beta$ and $\rho$ are parameters fixed throughout, called \emph{rates} of $\mathcal{B}$ and $\mathcal{R}$ respectively. Sets $\mathcal{B}_t$ and $\mathcal{R}_t$ are increasing in $t$, that is once a vertex gets colored with a certain color it does not change its state again.


The discretized version  $(\tilde{\mathcal{B}}_n , \tilde{\mathcal{R}}_n)$ of this process, which records its state only at times when a change happens, has very simple jump rules. 
At each integer $n$, choose an edge connecting a vertex $u$ in $\tilde{\mathcal{B}}_n \cup \tilde{\mathcal{R}}_n$ to a vertex $v$  in the complement $\big(\tilde{\mathcal{B}}_n \cup \tilde{\mathcal{R}}_n\big)^c$. The edges incident to a vertex in $\tilde{\mathcal{B}}_n$ are chosen with probability proportional to $\beta$ and those incident to a vertex in $\tilde{\mathcal{R}}_n$ with probability proportional to $\rho$. If $u \in \tilde{\mathcal{B}}_n$ then set $\tilde{\mathcal{B}}_{n+1} = \tilde{\mathcal{B}}_{n} \cup \{v\}$ and $\tilde{\mathcal{R}}_{n+1} = \tilde{\mathcal{R}}_n$. If $u \in \tilde{\mathcal{R}}_n$ then set $\tilde{\mathcal{R}}_{n+1} = \tilde{\mathcal{R}}_{n} \cup \{v\}$ and $\tilde{\mathcal{B}}_{n+1} = \tilde{\mathcal{B}}_n$. 

By $\mathcal{B}_{\text{fin}}$ and $\mathcal{R}_{\text{fin}}$ denote the final set of blue and red vertices when the whole graph is exhausted, and their sizes by $B_{\text{fin}}$ and $R_{\text{fin}}$ respectively. We are interested in the asymptotic behavior of $B_{\text{fin}}$ and $R_{\text{fin}}$ as the size of the graph tends to infinity, and how it depends on the choice of initial sets $\mathcal{B}_0$ and $\mathcal{R}_0$ and rates $\beta$ and $\rho$. Observe that time parametrization is irrelevant for the sets $\mathcal{B}_{\text{fin}}$ and $\mathcal{R}_{\text{fin}}$. In particular, we will be mainly studying the process through its discretized version $(\tilde{\mathcal{B}}_n, \tilde{\mathcal{R}}_n)$, which will be denoted by $(\mathcal{B}_n, \mathcal{R}_n)$ (as opposed to $(\mathcal{B}_t, \mathcal{R}_t)$ for the continuous process).

Consider the (finite) set of all simple $d$-regular vertex-labeled graphs with the vertex set $\{1,\dots,N\}$. The random $d$-regular graph on $N$ vertices is a random graph chosen uniformly from this set (here we assume that $dN$ is even, as otherwise such graphs do not exist). We will study the above process on the random $d$-regular graphs. 
Sets $\mathcal{B}_0$ and $\mathcal{R}_0$ will be chosen random as well.
This all means that we will first choose a $d$-regular graph graph on $N$ vertices from the uniform distribution,  conditioned on its realization we will sample sets $\mathcal{B}_0$ and $\mathcal{R}_0$ using a certain rule, and conditioned on the realization of this coupling we will run the competing first passage percolation process (CFPP) described above. Note that we will always assume that $d \geq 3$. The reason for this assumption is that $2$-regular graphs are just disjoint unions of cycles. As these graphs (except in the case of one cycle) are not connected, the process can not spread throughout the whole graph.

Here we state one of our results which is just a special case of Theorem \ref{thm:main_both_large}, but which nicely describes the type of results we obtain. It refers to the  case when the sets $\mathcal{B}_0$ and $\mathcal{R}_0$ are chosen uniformly of large prescribed size (in Section \ref{sec: fpp statements} we will allow more general rules for choosing the initial sets to model certain aspects of competitive behavior in real-world networks).

As our theorems give the asymptotics of $B_{\text{fin}}$ and $R_{\text{fin}}$ as the graph sizes  $N \to \infty$, values such as $B_0$, $R_0$, $B_{\text{fin}}$ and $R_{\text{fin}}$ will in general depend on $N$. However, to keep the formulas more readable, we will not always emphasize this dependence explicitly.

\begin{theorem}\label{thm:fpp_main_intro}
For $d \geq 3$ and a random $d$-regular graph on $N$ vertices,
assume that the sets $\mathcal{B}_0=\mathcal{B}_0(N)$ and $\mathcal{R}_0=\mathcal{R}_0(N)$ are chosen uniformly at random among all disjoint vertex subsets of sizes $B_0=B_0(N)$ and $R_0=R_0(N)$ respectively. Assume that there are constants $c_1, C_1$ such that $c_1 N^{\alpha_b}\leq B_0 \leq C_1 N^{\alpha_b}$ and $c_1 N^{\alpha_r}\leq R_0 \leq C_1 N^{\alpha_r}$. Then there exist constants $c_2, C_2$ such that with probability converging to $1$ as $N\to \infty$
\begin{itemize}
\item[i)]  $c_2 N^{\alpha_b + (1-\alpha_r)\beta/\rho} \leq B_{\text{fin}} \leq C_2 N^{\alpha_b + (1-\alpha_r)\beta/\rho}$, in the case $(1-\alpha_r)/\rho \leq  (1-\alpha_b)/\beta$,
\item[i)]  $c_2 N^{\alpha_r + (1-\alpha_b)\rho/\beta} \leq R_{\text{fin}} \leq C_2 N^{\alpha_r + (1-\alpha_b)\rho/\beta}$, in the case $ (1-\alpha_b)/\beta \leq  (1-\alpha_r)/\rho$.
\end{itemize}
\end{theorem}

This result shows that typically one process occupies only $o(N)$ vertices, and the other one everything else. From the applied perspective, the results of this type can be interpreted as one of the two products taking the lion share of the market. This result stands in striking contrast with the ones that we obtained in the case when the underlying graph is a large torus. Our results (see Theorem \ref{TorusThm}) show that even if we start one of the processes earlier than the other and we give it a much higher rate, the other process will still occupy a linear fraction of vertices with high probability.

\subsection{Acknowledgments}
E.M. is supported by MSF DMS
1106999 award, by DOD ONR grant N000141110140, by ISF
grant 1300/08 and by a Minerva Grant.

\section{Statements of results}\label{sec: fpp statements}


Recall that we want from our model to handle situations which arise in cases when one of the processes starts earlier than the other. Therefore we start by discussing how we choose the initial random sets $\mathcal{B}_0$ and $\mathcal{R}_0$. 

\begin{definition}\label{def: uniform initial sets}
For a graph $G=(V,E)$ we say that the pair $(\mathcal{B}_0, \mathcal{R}_0)$ of subsets of $V$, is  \emph {uniform of size}  $(B_0,R_0)$ if it is chosen uniformly at random among all pairs of disjoint subsets of $V$ of the sizes $B_0$ and $R_0$ respectively.
\end{definition}

For the case when one process (say $\mathcal{B}$) starts earlier than the other, the idea is to let $\mathcal{B}$ evolve from a uniformly chosen subset of some size, until it reaches a certain prescribed size. Then we define $\mathcal{B}_0$ to be the occupied set and take $\mathcal{R}_0$ to be a uniform subset of $\mathcal{B}_0^c$. The first phase in which only $\mathcal{B}$ grows is simply the CFPP process in which $\mathcal{R}$ starts from the empty set of vertices.
This leads to the following definition.

\begin{definition}\label{def: early initial sets}
For a graph $G=(V,E)$ we say that the random pair $(\mathcal{B}_0, \mathcal{R}_0)$ of disjoint subsets of $V$, is  \emph {uniform of size} $(B_0,R_0)$ \emph {with} $\mathcal{B}_0$ \emph {center of size} $k_0$ if
\begin{itemize}
\item[i)] $\mathcal{B}_0^0$ is a uniformly chosen subset of $V$ of the size $k_0$,
\item[ii)] $\mathcal{B}_0 = \mathcal{B}_T^0$, where $(\mathcal{B}^0,\mathcal{R}^0)$ is the CFPP process ran from $(\mathcal{B}_0^0, \emptyset)$, and $T$ is the first time $k$ that $|\mathcal{B}_k^0| = B_0$,
\item[iii)] $\mathcal{R}_0$ is the uniformly chosen subset of $\mathcal{B}_0^c$ of size $R_0$.
\end{itemize}
For a sequence of graphs $G_N=(V_N,E_N)$ we say for a sequence $(\mathcal{B}_{0}(N), \mathcal{R}_{0}(N))$ of disjoint pairs of random  subsets of $V_N$, that $\mathcal{B}_{0}$ \emph{has a small center} if 
\begin{itemize}
\item[i)] for every $N$, $(\mathcal{B}_{0}(N), \mathcal{R}_{0}(N))$ is  uniform of some size $(B_0(N),R_0(N))$ with $\mathcal{B}_0(N)$ center of size $k_0(N)$,
\item[ii)]  $\lim_{N \to \infty}k_0(N)/B_0(N) = 0$.
\end{itemize}
\end{definition}

For our results we need to either choose the pair of initial sets uniformly of prescribed size, or always allow one of the processes to have a significant advantage. This is captured by the following definition.

\begin{definition}\label{def: admissible sets}
For a sequence of graphs $G_N=(V_N,E_N)$ we say  that a sequence $(\mathcal{B}_{0}(N), \mathcal{R}_{0}(N))$ of disjoint pairs of random subsets of $V_N$ is \emph{admissible} of size $(B_0(N),R_0(N))$ if
\begin{itemize}
\item[i)] $|\mathcal{B}_0(N)|=B_0(N)$ and $|\mathcal{R}_0(N)|=R_0(N)$,
\item[ii)] either $\mathcal{B}_{0}$ or $\mathcal{R}_0$ has a small center, or for every $N$ the pair $(\mathcal{B}_{0}(N), \mathcal{R}_{0}(N))$ is uniform of size $(B_0(N),R_0(N))$.
\end{itemize}
\end{definition}

We now state our main results. Recall that Beta distribution  $Beta(\rho_1,\rho_2)$ is a distribution supported on $[0,1]$ with the density proportional to $t^{\rho_1-1}(1-t)^{\rho_2-1}$, and that Gamma distribution $\Gamma(\alpha,\beta)$ is supported on $\mathbb{R}^+$ with the density proportional to $t^{\alpha-1}e^{-\beta t}$. An elementary property of Gamma distribution that we will use to simplify the formulas in the statements is that if $Q$ has  $\Gamma(\alpha,1)$ distribution, then $\beta Q$ has $\Gamma(\alpha,1/\beta)$ distribution.

It is useful at this point to introduce some notation that will repeatedly appear in the statements of the main theorem. First define
\begin{equation}\label{eq: main thm starting type}
(\delta_B,\delta_R) = 
\left\{
\begin{array}{ll}
(1,1),& \text{ if }(\mathcal{B}_0,\mathcal{R}_0) \text{ are chosen uniformly,}\\
(1-2/d,1),& \text{ if }\mathcal{B}_0 \text{ has a small center,}\\
(1,1-2/d),& \text{ if } \mathcal{R}_0 \text{ has a small center.}
\end{array}
\right.
\end{equation}
Next define 
\begin{align}\label{eq: main thm multiplicative constant}
\sigma_B & = \frac{\delta_B\beta}{\delta_R^{\beta/\rho}\rho}\int_0^1 (t^{1/d}-t^{1-1/d})^{\beta/\rho -1}~dt, \\
\sigma_R & = \frac{\delta_R\rho}{\delta_B^{\rho/\beta}\beta}\int_0^1 (t^{1/d}-t^{1-1/d})^{\rho/\beta -1}~dt \nonumber,
\end{align}
which will appear as multiplicative constants in the limiting distribution for the component sizes, and
\begin{align}\label{eq: main thm multiplicative constant boundary}
\sigma'_B & = d\frac{\delta_B(\beta+\rho)}{\delta_R^{\beta/\rho}\rho}\int_0^1 (t-t^{d-1})^{\beta/\rho}~dt, \\
\sigma'_R & = d\frac{\delta_R(\beta+\rho)}{\delta_B^{\rho/\beta}\beta}\int_0^1 (t-t^{d-1})^{\rho/\beta}~dt \nonumber,
\end{align}
which will appear as multiplicative constants in the limiting distribution for the boundary.

In the statements of Theorems \ref{thm:main_both_large} to \ref{thm:boundary} we assume that $G$ is a random $d$-regular graph on $N$ vertices, and $(\mathcal{B},\mathcal{R})$ a competing first passage percolation process on $G$ with fixed parameters $(\beta,\rho)$. 

The first theorem covers the case when both processes start from a large size, and shows that the properly rescaled final sizes of colored components converge in probability.

\begin{theorem}\label{thm:main_both_large}
Assume that sequences $B_0=B_0(N)$ and $R_0=R_0(N)$ satisfy $\lim_{N}B_0 = \lim_{N}R_0 = \infty$ and $\lim_N(B_0+R_0)/N =0$.
If $(\mathcal{B},\mathcal{R})$ are started from admissible pairs of sizes $(B_0,R_0)$ then
 there exists deterministic sequences $\overline{B}= \overline{B}(N)$ and $\overline{R}= \overline{R}(N)$, such that  for every $\varepsilon > 0$ the final sizes $B_{\text{fin}}= B_{\text{fin}}(N)$ and $R_{\text{fin}}=R_{\text{fin}}(N)$ satisfy
\[
\lim_{N \to \infty}\mathbb{P}(|B_{\text{fin}} - \overline{B}| > \varepsilon \overline{B}) = \lim_{N \to \infty}\mathbb{P}(|R_{\text{fin}} - \overline{R}| > \varepsilon \overline{R}) = 0.
\]
Moreover, there are positive constants $c$ and $C$ depending only on $\rho/\beta$ and $d$ such that
\begin{align*}
c \min\left(R_0(N/B_0)^{\rho/\beta}, N\right)\leq \overline{R} \leq \min\left(CR_0(N/B_0)^{\rho/\beta},N\right),
\end{align*}
and
\begin{align*}
c \min\left(B_0(N/R_0)^{\beta/\rho}, N\right)\leq \overline{B} \leq \min\left(CB_0(N/R_0)^{\beta/\rho},N\right).
\end{align*}
For $\beta = \rho$ we have 
\[
\overline{B} = \frac{\delta_BB_0}{\delta_BB_0+\delta_RR_0}N \ \text{ and } \ \overline{R} = \frac{\delta_RR_0}{\delta_BB_0+\delta_RR_0}N, \ \text{ if } \ \beta = \rho,
\]
and for $\beta \neq \rho$
\begin{align*}
\overline{R}& = \sigma_R R_0(N/B_0)^{\rho/\beta}, \ \text{if }\ \lim_{N \to \infty} \frac{R_0(N/B_0)^{\rho/\beta}}{N} = 0,\\
\overline{B}& = \sigma_B B_0(N/R_0)^{\beta/\rho}, \ \text{if } \  \lim_{N \to \infty} \frac{B_0(N/R_0)^{\beta/\rho}}{N} = 0 \ \Big(\Leftrightarrow \lim_{N \to \infty} \frac{R_0(N/B_0)^{\rho/\beta}}{N} = \infty\Big).
\end{align*}
\end{theorem}

The following theorem covers the case when both processes start from fixed sizes, that is both sequences $B_0=B_0(N)$ and $R_0=R_0(N)$ are constant. Here of course, we don't need to worry about the possibility of one process starting earlier - such a version wouldn't be admissible.

\begin{theorem}\label{thm:small_sizes}
Assume that $B_0$, $R_0$ and $d \geq 3$ are fixed positive integers, and $\beta > 0$, $\rho > 0$ fixed rates. If $\beta = \rho$ then $R_{\text{fin}}/N$ and $B_{\text{fin}}/N$ converge in distribution, as $N \to \infty$ to 
\[
\frac{R_{\text{fin}}}{N}\to W 
, \ \ \frac{B_{\text{fin}}}{N}\to1-W, 
\]
where $W$ is distributed as $Beta(\frac{dR_0}{d-2}, \frac{dB_0}{d-2})$.
If $\beta \neq \rho$ then we have the following convergence in distribution
\begin{align*}
\frac{R_{\text{fin}}}{N^{\rho/\beta}} & \rightarrow  \sigma_R\frac{U}{V^{\rho/\beta}}\ \text{ if } \ \rho < \beta,\\ 
\frac{B_{\text{fin}}}{N^{\beta/\rho}} & \rightarrow  \sigma_B\frac{V}{U^{\beta/\rho}} \ \text{ if } \ \beta < \rho,
\end{align*}
where
$U$ and $V$ are independent random variables with Gamma distributions $\Gamma(\frac{d}{d-2}B_0,\frac{d}{d-2})$ and $\Gamma(\frac{d}{d-2}R_0,\frac{d}{d-2})$, respectively.
\end{theorem}


The next theorem covers the last case, when the initial set of one process is small and the initial set for the other one is large.

\begin{theorem}\label{thm:mixed_sizes}
Let $B_0$ be fixed, $\lim_{N \to \infty}R_0=\infty$ and $\lim_{N \to \infty}R_0/N=0$. If either $\beta \leq \rho$ or both $\beta > \rho$ and $\lim_N \frac{(N/R_0)^{\beta/\rho}}{N} = 0$ hold then we have the  convergence in distribution
\[
\frac{B_{\text{fin}}}{(N/R_0)^{\beta/\rho}} \rightarrow \sigma_B \Gamma\Big(\frac{dB_0}{d-2},\frac{d}{d-2}\Big).
\]
If both $\beta > \rho$ and $\lim_N \frac{R_0N^{\rho/\beta}}{N}=0$ $\big(\Leftrightarrow \lim_N \frac{(N/R_0)^{\beta/\rho}}{N} = \infty\big)$ hold then  we have the  convergence in distribution
\[
\frac{R_{\text{fin}}}{R_0N^{\rho/\beta}}   \rightarrow \sigma_R \Gamma\Big(\frac{dB_0}{d-2},\frac{d}{d-2}\Big)^{-\rho/\beta}.
\]
\end{theorem}

Note that, for $\beta \neq \rho$ in Theorems \ref{thm:main_both_large} to \ref{thm:mixed_sizes} the precise asymptotics is always given for the process which occupies $o(N)$ vertices. In other words, if one process ``looses the game'' (that is occupies only $o(N)$ vertices) we can give a very precise estimate on the distribution of the number of vertices it has occupied.

Finally, we describe the size of the boundary between the colored components. Since random regular graphs are expanders with high probability the above results on component sizes give us information on the size of the boundary up to the multiplicative factor. The following theorem however, provides information about the multiplicative constant as well.

\begin{theorem}\label{thm:boundary}
Let $D_{\text{fin}}$ denote the number of edges connecting the sets $\mathcal{B}_{\text{fin}}$ and $\mathcal{R}_{\text{fin}}$. Then
\begin{itemize}
\item[i)] 
If $B_0=B_0(N)$ and $R_0=R_0(N)$ satisfy $\lim_N B_0 = \lim_N R_0 = \infty$ and $\lim_N(B_0+R_0)/N = 0$ then
then there is a sequence $\overline{D}$ such that for any $\varepsilon > 0$ we have $\lim_{N \to \infty}\mathbb{P}(|D_{\text{fin}} - \overline{D}| > \varepsilon \overline{D}) = 0$. Moreover
\begin{align*}
\overline{D}& = \frac{\delta_B \delta_R B_0 R_0}{(\delta_B B_0+ \delta_R R_0)^2}(d-2)N, \ \text{ if } \ \beta = \rho,\\
\overline{D}& = \sigma'_R R_0(N/B_0)^{\rho/\beta}, \text{ if } \beta \neq \rho, \text{ and } \lim_{N \to \infty} \frac{R_0(N/B_0)^{\rho/\beta}}{N} = 0,\\
\overline{D}& = \sigma'_B B_0(N/R_0)^{\beta/\rho}, \text{ if } \beta \neq \rho, \text{ and } \lim_{N \to \infty} \frac{B_0(N/R_0)^{\beta/\rho}}{N} = 0 \ \Big(\Leftrightarrow \lim_{N \to \infty} \frac{R_0(N/B_0)^{\rho/\beta}}{N} = \infty\Big).
\end{align*}

\item[ii)] Assume that $B_0=B_0(N)$ and $R_0=R_0(N)$ are constant.  If $\beta = \rho$ then $D_{\text{fin}}/N$ converges in distribution as $N \to \infty$ 
\[
\frac{D_{\text{fin}}}{N} \to (d-2)W(1-W)
\]
where $W$ is distributed as $Beta(\frac{dR_0}{ d-2}, \frac{dB_0}{ d-2})$.
 If $\beta \neq \rho$ then we have the following convergence in distribution
\begin{align*}
\frac{D_{\text{fin}}}{N^{\rho/\beta}}& \to \sigma'_R \frac{U}{V^{\rho / \beta}},\ \text{ if }  \rho < \beta \\
\frac{D_{\text{fin}}}{N^{\beta/\rho}}& \to \sigma'_B \frac{V}{U^{\beta / \rho}} ,\ \text{ if }  \beta < \rho,
\end{align*}
where $U$ and $V$ are independent and distributed as $\Gamma\Big(\frac{ dR_0}{ d-2},\frac{d}{d-2}\Big)$ and $\Gamma\Big(\frac{dB_0}{d-2},\frac{d}{d-2}\Big)$ respectively.

\item[iii)] 
Let $B_0$ be fixed and $\lim_{N \to \infty}R_0=\infty$ and $\lim_{N \to \infty}R_0/N=0$. If either $\beta \leq \rho$ or both $\beta > \rho$ and $\lim_N \frac{(N/R_0)^{\beta/\rho}}{N} = 0$ hold then we have the  convergence in distribution
\[
\frac{D_{\text{fin}}}{(N/R_0)^{\beta/\rho}} \rightarrow \sigma'_B \Gamma\Big(\frac{dB_0}{d-2},\frac{d}{d-2}\Big).
\]
If both $\beta > \rho$ and $\lim_N \frac{R_0N^{\rho/\beta}}{N}=0$ $\big(\Leftrightarrow \lim_N \frac{(N/R_0)^{\beta/\rho}}{N} = \infty\big)$ hold then  we have the  convergence in distribution
\[
\frac{D_{\text{fin}}}{R_0N^{\rho/\beta}}   \rightarrow \sigma'_R \Gamma\Big(\frac{dB_0}{d-2},\frac{d}{d-2}\Big)^{-\rho/\beta }.
\]


\end{itemize}
\end{theorem}

The following theorem about the behavior of the processes on the torus stands in contrast with the results on the component sizes in the case of random regular graph.
Here we assume that $G= \torus{N}d$ is a $d$-dimensional torus with $N$ vertices, and $(\mathcal{B},\mathcal{R})$ a CFPP process on $\torus{N}d$ with parameters $(\beta,\rho)$. 

\begin{theorem}\label{TorusThm}
Let $\torus{N}d = \big( \mathbb Z / n \mathbb Z \big)^d$ for $n$ such that $N = n^d$, be the $d$-dimensional torus with $N$ vertices, and fix the rates $\beta > 0$, $\rho > 0$.  
Then for any $\varepsilon > 0$ there exists $\delta>0$ such that 
\[
\liminf_{n \to \infty} \mathbb{P}(B_{\text{fin}} > \delta N, R_{\text{fin}} > \delta N) > 1-\varepsilon,
\]
if either of the following two conditions hold
\begin{itemize}
\item[i)] $B_0$ and $R_0$ are fixed positive integers, and $(\mathcal{B}_0,\mathcal{R}_0)$ are chosen uniformly of size $(B_0,R_0)$,
\item[ii)] $R_0$ and $k_0$ are fixed positive integers, sequence $B_0$ converges to $\infty$ and it satisfies $\lim_N B_0/N = 0$ and $(\mathcal{B}_0,\mathcal{R}_0)$ are chosen uniformly of size $(B_0,R_0)$, with $\mathcal{B}_0$ center of size $k_0$.
\end{itemize}
\end{theorem}



\subsection{Remarks and follow up work}
We note that the results of all the theorems above cannot hold if the
sets $\B_0$ and $\R_0$ are arbitrary. Consider for example the case
where $\B_0$ is the ball of radius $r$ in the graph around a vertex
$v$ and $\R_0$ consists of all vertices at distance exactly $r+1$ from
$v$. While the set $\R_0$ is not much bigger than $\B_0$ - clearly the
remaining vertices will all become red.

The fact that the results do not hold for arbitrary sets raise various
game theoretic questions.  For example, consider a game where player
$B$ has to choose the set $\B_0$ and player $R$ has to choose the set
$\R_0$.  Suppose player $B$ can choose up to $N^{\alpha_1}$ initial
vertices and player $R$ can choose up to $N^{\alpha_2}$ initial
vertices.  What are the Nash Equilibrea of this game? Are the payoffs
in the Nash Equilibrea close to the payoffs obtained if the two players
place the initial sets at random? Similar game theoretic questions may
be asked if players alternate in placing the elements of $\B_0$ and
$\R_0$.

As far as we know this game was first defined by Bharathi, Kempe and
Salek in \cite{BharathiKempeSalek07}. Their paper provides an
approximation algorithm for the best response and shows that the
social price of competition is at most $2$ but does not analyze the
utilities of each of the players in a Nash Equilibrea.  A different
direction of future study is extending the result in the current work
to more realistic models of social networks and marketing.  In
particular it would be interesting to study the same question on
preferential attachment random graphs and other more realistic models of
social networks. We expect that for such graphs, game theoretic
consideration can play an important role due to the different degrees
and connectivity of different vertices.

\subsection{Related work}
As mentioned earlier, diffusion and growth processes have been studied
intensely in the past few years in relation to many areas such as
sociology, economics and engineering. Among the models studied are
\textit{stochastic cellular automata} (see, for example
\cite{Wolfram86}, \cite{GoLiMu:01a}, \cite{GoLiMu:01b}), \textit{the
  voter model} which was first introduced by Clifford and Sudbury in
\cite{CliffordSudbury73} and has been much studied since in, for
example, \cite{HolleyLiggett75}, 
\cite{DonnellyWelsh83}, \textit{the contact process} (see, for
example, \cite{Griffeath81}), \textit{the stochastic Ising model} (see
\cite{Glauber63}, \cite{Bremaud01}), and \textit{the influence model}
(see \cite{AsRoLeVe:01}).

Recently, a strong motivation for analyzing diffusion processes has
emanated from the study of viral marketing strategies in data mining
(see, for example, \cite{DomingosRichardson:01},
\cite{DomingosRichardson:02}, \cite{KeKlTa:03}, \cite{KeKlTa:05}). In
this model one takes into account the ``network value'' of potential
customers, that is, it seeks to target a set of individuals whose
influence on the social network through word-of-mouth effects is
high. For a given diffusion process, we define the influence
maximization problem. For each initial set of active nodes $S$, we
define $\sigma(S)$ to be the expected size of the set of active nodes
at the end of the process. In the influence maximization problem, we
aim to find a set $S$ of fixed size that maximizes $\sigma(S)$. In
attempts to find a set of influential individuals, heuristic
approaches such as picking individuals of high degree or picking
individuals with short average distance to the rest of the network
have been commonly used, typically with no theoretic guarantees (see
\cite{WassermanFaust94}). In \cite{KeKlTa:03} it was shown that the
influence maximization problem is NP-hard to approximate within a
factor of $1 - e^{-1} + \varepsilon$ for all $\varepsilon
> 0$. On the other hand, in \cite{KeKlTa:05} it was shown that under
the assumption that the function $\sigma$ is \textit{submodular}, for
every $\varepsilon > 0$ it is possible to find a set $S$ of fixed size
that is a $(1 - e^{-1} - \varepsilon)$-approximation of
the maximum in random polynomial time. In \cite{MosselRoch10} it was
proven that the function $\sigma$ is indeed submodular.

As mentioned earlier the paper \cite{BharathiKempeSalek07} defines the
competitive influence maximization problem on general graphs. We
believe that an interesting research direction is to show that for
random $d$-regular graphs, the payoffs of the two players at each Nash
Equilibrea are essentially the same as the payoff obtained by playing
according to random strategies.


\section{Coupling with the configuration model}

The configuration model (CM), introduced by Bollob\'{a}s in \cite{Bollobas80}, is a randomized algorithm used to construct a uniform random $d$-regular labeled graph on $N$ vertices (we always assume that $dN$ is even, as otherwise there is no such graph).
In this
model we view each vertex $i \in [N]=\{1,2,\dots,N\}$ of the graph as a set $H(i)$ of $d$
\textit{half-edges}. We pick a uniform perfect matching on the set $\cup_{i \in [N]} H(i)$ of all
$dN$ half-edges (recall that $dN$ is even), and contract each $d$-tuple of half-edges $H(i)$ back to a
single vertex. This yields a $d$-regular graph with the vertex set $[N]$, and in which every coupled pair of a half-edge in $H(i)$ and $H(j)$ gives an edge connecting the vertices $i$ and $j$.
It is shown in \cite{Bollobas80} that with probability
that tends to $e^{\frac{1-d^2}4}$ as $N \rightarrow \infty$,
this is a simple $d$-regular graph. Moreover, conditioning
on the event that the graph is simple, it is uniformly distributed
among all simple $d$-regular labeled graphs on $N$ vertices. 

We will couple the configuration model (CM) and the competing process (CFPP), and will prove the statements in Theorems \ref{thm:main_both_large} to \ref{thm:boundary} for the coupled process. This will be sufficient, as these statements hold asymptotically almost surely and the probability of generating a simple $d$-regular graph is bounded from below.

To make the coupling easier we will slightly modify the competing first passage percolation model we study, by also coloring the edges.
In the modified process (MCFPP) we start like before with two disjoint subsets  $\mathcal{B}_0$ and  $\mathcal{R}_0$ of the vertex set colored blue and red respectively, and initially we set all the edges uncolored. At the $n$-th step we choose a pair $(u,e)$ of a vertex $u$ in $\mathcal{B}_n \cup \mathcal{R}_n$ and an incident uncolored edge $e$, using the same probabilities as before: every pair for which $u \in \mathcal{B}_n$ is chosen with the probability proportional to $\beta$ and every pair for which $u \in \mathcal{R}_n$ is chosen with the probability proportional to $\rho$.
Then we color the edge $e$ in the color of $u$. If the other end of $e$ is uncolored, we also color it into the color of $u$. In this modification we can have steps that do not yield to coloring of new vertices, but it is easy to see that conditioned that the set $\mathcal{B}_n \cup \mathcal{R}_n$ does increase, the transition probabilities are the same as in the original model. Thus the distribution of  $(\mathcal{B}_{\text{fin}},\mathcal{R}_{\text{fin}})$ is unchanged.

We now describe the coupling of CM and MCFPP which will be refer to as CP.
We first focus on the  case when $(\mathcal{B}_0,\mathcal{R}_0)$ is chosen uniformly of size $(B_0,R_0)$, but will allow for one of $\mathcal{B}_0$ or $\mathcal{R}_0$ to be empty, (this is important as it will correspond to the evolution of the process when one set is given an advantage). In the coupling $\mathcal{X}_n$, $\mathcal{Y}_n$, $\mathcal{Z}_n$ and $\mathcal{W}_n$ will denote the set of active blue, active red, uncolored and explored half-edges, and these four sets will be disjoint for every $n$.
The coupling CP goes as follows.
\begin{itemize}
\item[i)]
Start with the random pair of subsets $(\mathcal{B}_0, \mathcal{R}_0)$ which is chosen uniformly at random from all disjoint pairs of subsets of $[N]$ of sizes $(B_0,R_0)$. 
Denote by $\mathcal{X}_0$ and $\mathcal{Y}_0$ the half-edges corresponding to vertices in $\mathcal{B}_0$ and $\mathcal{R}_0$ respectively, that is $\mathcal{X}_0 = \cup_{i \in \mathcal{B}_0}H(i)$ and $\mathcal{Y}_0 = \cup_{i \in \mathcal{R}_0}H(i)$. 
Set the uncolored and explored half-edges initially as $\mathcal{Z}_0 = \big(\mathcal{X}_0\cup  \mathcal{Y}_0\big)^c$ and  $\mathcal{W}_0=\emptyset$.
\item[ii)]
At every time step $n \geq 0$ choose a half edge $x$ in $\mathcal{X}_n \cup \mathcal{Y}_n$, the ones in $\mathcal{X}_n$ with probability proportional to $\beta$, and the ones in $\mathcal{Y}_n$ with probability proportional to $\rho$. Match the chosen half edge to a uniformly chosen half edge $y$ in $\mathcal{X}_n\cup \mathcal{Y}_n\cup \mathcal{Z}_n\backslash \{x\}$. Make $x$ and $y$ explored.
Let $i,j \in [N]$ be such that $x \in H(i)$ and $y \in H(j)$.
 If $y\in \mathcal{Z}_n$ then color all the half-edges in $H(j)$ with the color of $x$.  More precisely, if $x \in \mathcal{X}_n$ set  
\[
\mathcal{X}_{n+1} = \mathcal{X}_n \cup H(j) \backslash \{x,y\}, \  \mathcal{Y}_{n+1} = \mathcal{Y}_n, \ \mathcal{Z}_{n+1} = \mathcal{Z}_n \backslash H(j), \ \text{ if } y \in \mathcal{Z}_n,
\] 
\[
\mathcal{X}_{n+1} = \mathcal{X}_n  \backslash \{x,y\}, \  \mathcal{Y}_{n+1} = \mathcal{Y}_n, \ \mathcal{Z}_{n+1} = \mathcal{Z}_n , \ \text{ if } y \in \mathcal{X}_n,
\]
\[
\mathcal{X}_{n+1} = \mathcal{X}_n \backslash \{x\}, \  \mathcal{Y}_{n+1} = \mathcal{Y}_n\backslash \{y\}, \ \mathcal{Z}_{n+1} = \mathcal{Z}_n,  \ \text{ if } y \in \mathcal{Y}_n,
\]
and $\mathcal{W}_{n+1} = \mathcal{W}_n \cup \{x,y\}$ in every case. 
If $x \in \mathcal{Y}_n$ proceed analogously. 
Connect vertices $i$ and $j$ with an edge and color this edge in the color of $i$ (same as the color of $x$). If $j$ is uncolored, color it into the color of $i$ as well. More precisely set 
\[
\mathcal{B}_{n+1} = \mathcal{B}_n, \ \mathcal{R}_{n+1} = \mathcal{R}_n, \ \text{ if } j \in \mathcal{B}_n \cup \mathcal{R}_n \text{ (} \Leftrightarrow y \in \mathcal{X}_n \cup \mathcal{Y}_n \text{)},
\]
\[
\mathcal{B}_{n+1} = \mathcal{B}_n \cup \{j\}, \ \mathcal{R}_{n+1} = \mathcal{R}_n,\ \text{ if } j \notin \mathcal{B}_n \cup \mathcal{R}_n \text{ (} \Leftrightarrow y \notin \mathcal{X}_n \cup \mathcal{Y}_n \text{)},\text{ and } i \in \mathcal{B}_n \text{ (} \Leftrightarrow x \in \mathcal{X}_n \text{)},
\]
\[
\mathcal{B}_{n+1} = \mathcal{B}_n, \ \mathcal{R}_{n+1} = \mathcal{R}_n \cup \{j\}, \ \text{ if } j \notin \mathcal{B}_n \cup \mathcal{R}_n \text{ (} \Leftrightarrow y \notin \mathcal{X}_n \cup \mathcal{Y}_n \text{)},\text{ and } i \in \mathcal{R}_n \text{ (} \Leftrightarrow x \in \mathcal{Y}_n \text{)}.
\]

\item[iii)] Stop the algorithm when $\mathcal{X}_n = \mathcal{Y}_n= \emptyset$. 
\end{itemize}
Note the algorithm can fail to color the whole graph only if, for some $n$, we have $\mathcal{X}_n = \mathcal{Y}_n = \emptyset$ and $\mathcal{Z}_n \neq \emptyset$. 
This is consistent with the fact that random $d$-regular graph may fail to be connected, and if all the initial colored vertices are in a single connected component, the CFPP process will not color all the vertices.
However, for any $d \geq 3$ the probability of this event converges to $0$, as $N \to \infty$, see \cite{Bollobas01, JLR00}. As all our results hold asymptotically as $N\to \infty$ this will not be an issue.


By $X_n$, $Y_n$ and $Z_n$ denote the sizes of $\mathcal{X}_n$, $\mathcal{Y}_n$ and  $\mathcal{Z}_n$ respectively. Denoting $M= X_0 + Y_0 + Z_0$,  and observing that at  each time two half-edges become inactive we have either $X_n + Y_n +Z_n = M -2n$ or $X_n + Y_n=0$. The process $(X_n, Y_n, Z_n)$ is a Markov chain with the following transition probabilities (given $X_n + Y_n >0$)
\begin{align*}
\mathbb{P}(X_{n+1} = X_n + d-2, Y_{n+1} = Y_n, Z_{n+1} = Z_n-d) & = \frac{\beta X_n}{\beta X_n + \rho Y_n}\frac{Z_n}{X_n +Y_n +Z_n -1} \\
\mathbb{P}(X_{n+1} = X_n , Y_{n+1} = Y_n + d-2, Z_{n+1} = Z_n-d) & = \frac{\rho Y_n}{\beta X_n + \rho Y_n}\frac{Z_n}{X_n + Y_n +Z_n -1} \\
\mathbb{P}(X_{n+1} = X_n-2 , Y_{n+1} = Y_n, Z_{n+1} = Z_n) & = \frac{\beta X_n}{\beta X_n + \rho Y_n}\frac{X_n-1}{X_n + Y_n +Z_n -1} \\
\mathbb{P}(X_{n+1} = X_n , Y_{n+1} = Y_n-2, Z_{n+1} = Z_n) & = \frac{\rho Y_n}{\beta X_n + \rho Y_n}\frac{Y_n-1}{X_n + Y_n +Z_n -1} \\
\mathbb{P}(X_{n+1} = X_n-1 , Y_{n+1} = Y_n-1, Z_{n+1} = Z_n) & = \frac{(\rho + \beta) X_nY_n}{(\beta X_n + \rho Y_n)(X_n + Y_n  +Z_n -1)}.
\end{align*}
If at some point $X_{n_0} +Y_{n_0}=0$ and $Z_{n_0}>0$ set $X_n =X_{n_0}$, $Y_n =Y_{n_0}$ and $Z_n =Z_{n_0}$ for all $n \geq n_0$.
The advantage of this coupling is that the process $(B_n,R_n)$ can be studied through the process $(X_n,Y_n,Z_n)$. This process in turn, is completely described by the above transition probabilities. Indeed, most of the technical work in this chapter is devoted to establishing maximal inequalities for the process $(X_n,Y_n,Z_n)$.

Next we justify the coupling, that is explain why the CP process produces the graph with the set of blue and red vertices, which is equal in distribution to $(G,\mathcal{B}_{\text{fin}}, \mathcal{R}_{\text{fin}})$ produced by first generating a random regular graph $G$ and then running MCFPP on it (denote this by CM$\times$MCFPP). Denote by $\mathbf{A}_n$ the cluster formed by the colored edges and colored vertices at time $n$ in CM$\times$MCFPP and by $\mathbf{A}_n'$ the cluster formed by  the colored edges and colored vertices at time $n$ in CP (both $\mathbf{A}_n$ and $\mathbf{A}_n'$  contain the information about the color of each edge and vertex). It suffices to show that $\mathbf{A}_n$ and  $\mathbf{A}_n'$ have the same distribution for every $n$. We show this inductively. For $n=0$ the claim is obvious, as both $\mathbf{A}_0$ and $\mathbf{A}_0'$ consist of uniform disjoint subsets of $[N]$ of prescribed size colored blue and red respectively, and all the edges are uncolored. Condition on some realization $\mathbf{A}_n'=A$. Observe that the probability that $\mathbf{A}_{n+1}'$ is formed by connecting a vertex $i\in A$ to an uncolored vertex 
is given by
\[
\frac{\tau E_n(i)}{\beta X_n + \rho Y_n}\frac{d}{M-2n-1},
\]
where $\tau=\beta$ if $i$ is blue and $\tau=\rho$ if $i$ red, and
$E_n(i)$ is the number half-edges incident to $i$ and not present in $A$ (that is $|(\mathcal{X}_n \cup \mathcal{Y}_n) \cap H(i)|$). 

To study the same conditional probability for CM$\times$MCFPP observe that the event $\{\mathbf{A}_n=A\}$ happens if and only if the graph generated by CM supports the cluster $A$, and MCFPP on this graph generates $A$ in the $n$-th step. The event that $\mathbf{A}_{n+1}$ is formed by joining a vertex $i\in A$ with an uncolored vertex $j$ happens if and only if CM produces a graph in which there is $N \geq 1$ edges connecting $i$ and $j$ (as $j$ is uncolored, such edges can not be a part of the cluster $A$), and in the next step MCFPP spreads along one of these edges.
For any of the $E_n(i)$ half-edges, conditioned on $\{\mathbf{A}_n=A\}$, the probability it gets matched to a half-edge in $H(j)$ is simply $d/(M-n-1)$. Thus $\mathbb{E}(N|\mathbf{A}_n=A) = dE_n(i)/(M-2n-1)$. On the other hand, since $X_n$ and $Y_n$ are simply numbers of active half-edges emanating from blue and red colored vertices in $A$ respectively, given $N=k$ the conditional probability that MCFPP spreads along one of these edges is simply $\tau k/(\beta X_n+\rho Y_n)$. Thus the probability of the observed event in CM$\times$MCFPP is
\[
\sum_{k=1}^{E_n(i)} \frac{\tau k\mathbb{P}(N=k)}{\beta X_n + \rho Y_n} = \frac{\tau \mathbb{E}(N)}{\beta X_n + \rho Y_n} = \frac{\tau E_n(i)}{\beta X_n + \rho Y_n}\frac{d}{M-2n-1}.
\]
Similar calculations show that the  conditional probabilities agree for the events that $j$ is already chosen red or blue.


The situation when one process starts earlier than the other is handled in almost exactly the same way. 
The initial dynamics, when only one process grows, can be described throught stage 0 of the CP process defined below. 
In this stage all the variables will be denoted with the superscript $0$.
First observe that the above coupling makes sense even if $\mathcal{R}_0 = \emptyset$. To generate random subsets $(\mathcal{B}_0,\mathcal{R}_0)$ as in Definition \ref{def: early initial sets} simply run CP with $\mathcal{B}_0^0$ as a uniformly chosen subset of $[N]$ of size $k_0$, $\mathcal{R}_0^0 = \emptyset$ until $\mathcal{B}_k^0$ grows to the prescribed size $B_0$. More precisely, we define the stage 0 of CP as follows.
\begin{itemize}
\item[i)] Take $\mathcal{B}_0^0$ as a uniformly chosen subset of $[N]$ of size $k_0$,  $\mathcal{R}_0^0 = \emptyset$, $\mathcal{X}_0^0 = \cup_{i \in \mathcal{B}_0^0}H(i)$, $\mathcal{Y}_0^0 = \emptyset$, $\mathcal{Z}_0^0= (\mathcal{X}_0^0)^c$, $\mathcal{W}_0^0 = \emptyset$.
\item[ii)] Run CP with the above initial conditions.
\item[iii)] Stop CP at $T$ when $|\mathcal{B}_T^0| = B_0$. Set $\mathcal{B}_0 = \mathcal{B}_T^0$, $\mathcal{X}_0 = \mathcal{X}_T^0$, $\mathcal{W}_0 = \mathcal{W}_T^0$.
\end{itemize}
By the strong Markov property, CP for the initial conditions when $(\mathcal{B}_0,\mathcal{R}_0)$ is chosen as a uniform subset of size $(B_0,R_0)$, with $\mathcal{B}_0$ center of size $k_0$ goes as follows
\begin{itemize}
\item[i)] Run stage 0 of CP.
\item[ii)] Take $\mathcal{B}_0$, $\mathcal{X}_0$ and $\mathcal{W}_0$ as produced in stage 0, $\mathcal{R}_0$ as a uniform subset of $\mathcal{B}_0^c$ of size $R_0$, $\mathcal{Y}_0 = \cup_{i \in \mathcal{R}_0}H(i)$ and $\mathcal{Z}_0 = \big(\mathcal{X}_0 \cup \mathcal{Y}_0 \cup \mathcal{W}_0 \big)^c$.
\item[iii)] Run CP with the initial conditions from ii).
\end{itemize}
As in the usual CP we denote the sizes of corresponding sets by the same letter in the normal font and set $M= X_ 0+ Y_0 + Z_0 = dN - W_0$.

The following   result shows how one can estimate the final sizes simply from $X_0=|\mathcal{X}_0|$ and $Y_0=|\mathcal{Y}_0|$. From this theorem we will derive all our results about the competing process on the random regular graph. 
Note that the notation $x= (1 \pm \varepsilon) y$ means that $(1-\varepsilon)y \leq x \leq (1+\varepsilon)y$.
As usual a property holds asymptotically almost surely if it's probability converges to 1 as $N \to \infty$. 
A more careful of the analysis of the proofs of the following two theorems shows that the rate at which the probabilities converge to zero does not depend on the initial values of $X_0$, $Y_0$, except through the sequence $L_M$.

\begin{theorem}\label{thm:(1+o(1))_concentration_diff_rates_0}
Let $(L_N)$ be a sequence converging to $\infty$.  Assume that in a CP process, with the admissible initial conditions, sequences $X_0=X_0(N)$, $Y_0=Y_0(N)$ and $Z_0=Z_0(N)$ satisfy $\min(X_0,
 Y_0) \geq L_N$ and $(X_0+
 Y_0)/M \to 0$. When $\beta = \rho$, for any $\varepsilon
>0$ with probability converging to $1$ (as $N \to \infty$) we have
\begin{equation}\label{eq:1+o(1)_concentration_beta_equal_1}
R_{\text{fin}}-R_0 = (1 \pm \varepsilon)
\frac{Y_0M}{d(X_0+Y_0)}.
\end{equation}
When $\beta \neq \rho$, for any $\varepsilon
>0$ with probability converging to $1$ (as $N \to \infty$) we have
\begin{equation}\label{eq:1+o(1)_concentration_beta_no_1}
R_{\text{fin}} - R_0 = (1 \pm \varepsilon) \frac{M}{d}\int_0^1
\phi_{\beta,\rho}^{-1}
\left(MX_0^{\frac{\rho}{\beta-\rho}}Y_0^{\frac{\beta}{\rho-\beta}}
\Big(t^{1/d} -
t^{1-1/d}\Big)\right)dt.
\end{equation}
Here $\phi_{\beta,\rho}^{-1} \colon \mathbb{R}^+ \to (0,1)$ is an inverse of a one-to-one function defined as
\begin{equation}\label{eq: phi 2 parameter function}
\phi_{\beta,\rho} (s)= \Big(\frac{\beta s}{\rho (1-s)}\Big)^{\frac{\rho}{\rho-\beta}} +  \Big(\frac{\rho (1-s)}{\beta s}\Big)^{\frac{\beta}{\beta-\rho}}.
\end{equation}
\end{theorem}
 
Under the same assumptions we have the analogous result for the boundary. Note that $D_0$ denotes the size of the boundary in possibly partially constructed initial graph.
\begin{theorem}\label{thm: 1+o(1) concentration for boundary}
Let the assumptions in Theorem \ref{thm:(1+o(1))_concentration_diff_rates_0} hold. For $\beta = \rho$ and every $\varepsilon > 0$ with probability converging to $1$ (as $N \to \infty$) we have
\begin{equation}\label{eq:(1+o(1))_concentration_boundary_eq_rates}
D_{\text{fin}} - D_0= (1\pm\varepsilon)M\frac{(d-2)X_0Y_0}{d(X_0+Y_0)^2}
\end{equation}
For  $\beta \neq \rho$ and every $\varepsilon > 0$ with probability converging to $1$ (as $N \to \infty$) we have
\begin{equation}\label{eq:(1+o(1))_concentration_boundary_diff_rates}
D_{\text{fin}} -D_0 = (1\pm\varepsilon) \frac{M}{2}\int_0^1\Big(1-t^{d/2-1}\Big) \kappa_{\beta,\rho} \circ \phi^{-1}_{\beta,\rho} 
\left(MX_0^{\frac{\rho}{\beta-\rho}}Y_0^{\frac{\beta}{\rho-\beta}}\Big(t^{1/2} -
t^{(d-1)/2}\Big)\right)dt,
\end{equation}
where $\kappa_{\beta,\rho}(t) = \frac{(\beta + \rho)t(1-t)}{\beta t + \rho(1-t)}$ and $\phi_{\beta,\rho}$ is  as in \eqref{eq: phi 2 parameter function}.
\end{theorem}

It doesn't seem to be possible to simplify the expressions in \eqref{eq:1+o(1)_concentration_beta_no_1} and \eqref{eq:(1+o(1))_concentration_boundary_diff_rates}.
Of course, one could replace the expressions under the inverse function $\phi_{\beta,\rho}$ and redefine $\phi_{\beta,\rho}$ to obtain equivalent formulas. However, this formulation seems to be very convenient as the expression $X_0^{\rho/(\beta-\rho)}Y_0^{\beta/(\rho-\beta)}$ remains unchanged if we switch the roles of the processes $(B,R)$ and change the  parameters $(\beta,\rho)$ to $(\rho,\beta)$. On the other hand, just as one expects $\phi_{\beta,\rho}(s) = \phi_{\rho,\beta}(1-s)$ which gives
 $\phi_{\beta,\rho}^{-1}(t) = 1 - \phi_{\rho,\beta}^{-1}(t)$ and $\kappa_{\beta,\rho}(t) = \kappa_{\rho,\beta}(1-t)$.

To prove the main theorems from Theorem \ref{thm:(1+o(1))_concentration_diff_rates_0} one needs to relate $X_0$ and $Y_0$ to $B_0$ and $R_0$. In the setting of Theorem \ref{thm:main_both_large} when $(\mathcal{B}_0,\mathcal{R}_0)$ are chosen uniformly, this is trivial as $X_0 = dB_0$ and $Y_0=dR_0$. However, when the $(1+o(1))$ concentration fails this can become more complicated. This is the content of the following section which, deduces Theorems \ref{thm:main_both_large} to \ref{thm:boundary} from Theorems \ref{thm:(1+o(1))_concentration_diff_rates_0} and \ref{thm: 1+o(1) concentration for boundary}. Having done enough preparation these proofs will be very short.
Note, that the proofs of parts of Theorem \ref{thm:boundary} follow verbatim the short proofs of the corresponding statements of boundary sizes, so they will be omitted.

\section{Deducing the main theorems}

First we start with simple technical remarks that will be of use in our estimates.
\begin{remark}\label{rem:nesting}
Let $(P_n)_{n \geq 0}$ be a Markov process adapted to the filtration $(\mathcal{F}_n)$, such that $P_0 = 0$, $P_{n+1}-P_n \in \{0,1\}$, and let $\tau$ be a stopping time adapted to the same filtration. Let $\sigma$ be the first index $k$ such that $P_{k}-P_{k-1} =0$. If on the event $\{\tau>k\}\cap \{\sigma > k\}$ we have $\mathbb{P}(P_{k+1}-P_k =1|\mathcal{F}_k) \geq 1-\delta_k$, then a simple induction in $n$ shows that
\begin{equation}\label{eq:concentration_prob_very_likely}
\mathbb{P}(P_n = n) \geq \prod_{k=1}^n (1-\delta_k) - \mathbb{P}(\tau \leq n) \geq 1-\sum_{k=1}^n \delta_k - \mathbb{P}(\tau \leq n).
\end{equation}
If on the other hand, on the event $\{\tau>k\}$ we have $\mathbb{P}(P_{k+1}-P_k =1|\mathcal{F}_k) =(1\pm\varepsilon) p_k$, then $P_n$ can be stochastically dominated from below by $(Y_1 + \dots +Y_n)1_{\{\tau > n\}}$, and from above by  $Z_1 + \dots +Z_n + n1_{\{\tau \leq n\}}$, where $Y_k$ are iid Bernoulli with parameters $(1-\varepsilon)p_k$, and $Z_k$ are iid Bernoulli with parameters $(1+\varepsilon)p_k$, both independent of $\tau$. In particular, by Chebyshev bound 
\begin{equation}\label{eq:concentration_prob_very_unlikely}
\mathbb{P}(P_n= (1\pm 2\varepsilon)\sum_{k=1}^np_k) \geq 1 - \frac{1+\varepsilon}{\varepsilon^2 (1-\varepsilon)\sum_{k=1}^n p_k} - \mathbb{P}(\tau \leq n).
\end{equation}
Finally observe $\mathbb{E}(P_{\tau\wedge n}) \leq (1+\varepsilon)\sum_{k=1}^n p_k$.
\end{remark}

The probability that the process $(X_n,Y_n,Z_n)$ gets ``stuck'' at some point due to running out of active half-edges has probability converging to zero, as $N\to\infty$. 
This is due to the fact that the random graph generated by the configuration model is connected with probability converging to 1 (see also Theorem \ref{thm: urn applied to model}).
Since the convergence in our results is either in distribution or in probability, we can assume in the proofs that the process does not get stuck.
This is justified by the above remark, because the time when the active half-edges run out is a stopping time. 

\begin{remark}\label{rem:confg_model_random_regular_graphs}
As mentioned, conditioned on the event that the random graph generated by the configuration model is simple, it's law is exactly that of a random regular graph. Since the probability that the generated graph is simple converges to $e^{(1-d^2)/4}$, we don't worry about this conditioning in the proofs of claims which hold with probability converging to 1 (as $N \to \infty$). However, we have to be somewhat more careful when proving convergence in distribution (as in Theorems \ref{thm:small_sizes} to \ref{thm:boundary}). To prove these results we will run the process up to the (possibly random) time $T$ which is large compared to the sizes of initial components $X_0$ and $Y_0$, but still small compared to the total size of the graph $M$, that is, as $M \to \infty$ we have $X_T, Y_T \to \infty$, $(X_0+Y_0)/T\to 0$ and $T/M\to 0$, where all convergences are uniform. Then we prove that at this time $T$, the coefficient $MX_T^{\rho/(\beta-\rho)}Y_T^{\beta/(\rho-\beta)}$ appearing in \eqref{eq:1+o(1)_concentration_beta_no_1} and \eqref{eq:(1+o(1))_concentration_boundary_diff_rates} converges in distribution. To finish the proof, we will invoke the concentration results from Theorems \ref{thm:(1+o(1))_concentration_diff_rates_0} and \ref{thm: 1+o(1) concentration for boundary}. 
Indeed it is a simple observation that this implies the convergence in distribution for the random graph generated by the configuration model.
The fact that conditioning on the event that this random graph is simple does not change the limiting distribution, follows from the following two facts: with probability converging to 1 there are no loops or double edges generated up to time $T$; conditioned on the realization of the configuration model at time $T$, the probability of generating a simple graph converges to $e^{(1-d^2)/4}$. 
The first claim is straightforward from the fact that at each step the probability of creating a loop or a double edge is bounded from above by $C/M$ (where the constant $C$ depends only on $d$) and that $T/M \to 0$. 
To justify the second claim, observe that at time $T$ we have $(1-o(1))N$ of vertices with all $d$ half-edges unmatched, and $o(1)N$ vertices which have at least one (and no more than $d-1$) half-edges unmatched. Then, we only need to argue that a probability that a configuration model which starts with $(1-o(1))N$ vertices of degree $d$ and $o(1)N$ vertices of degree less than $d$, generates a simple graph with probability converging to $e^{(1-d^2)/4}$. Indeed, under rather general assumptions the probability that a uniformly sampled multigraph with $n$ vertices and vertex degrees $(d_i)_{1 \leq i \leq n}$ is asymptotically equal to
\[
\exp\left(\frac{1}{4} - \frac{1}{4}\left(\frac{\sum_{i=1}^n d_i^2}{\sum_{i=1}^nd_i}\right)^2\right),
\]
see Theorem II.16 in \cite{Bollobas01}, which implies the second claim.
\end{remark}

If one of the initial sets, say $\mathcal{B}_0$, is assumed to have a small center then we still have $Y_0 = dR_0$. However, to estimate $X_0$ one needs to understand the evolution of the number of active edges, in the stage 0 of CP when only the blue set evolves. The small center assumption yields $X_0 = (1+o(1))(d-2)B_0$ with high probability, see Lemma \ref{lemma: first stage analysis} below. 


First we introduce the notion of an urn model that will be used throughout the chapter.
\begin{definition}\label{def:urn_model}
We say that a process $(S_n,Z_n)_n$ adapted to a filtration $(\mathcal{F}_n)$ is a P\'{o}lya urn process with a replacement matrix $A=\left(\begin{array}{ll}a_{11} & a_{12} \\ a_{21} & a_{22}\end{array}\right)$  if conditioned on $\mathcal{F}_n$ and on the event $\{S_n \geq 0, Z_n \geq 0, S_n +Z_n>0\}$, with probability $S_n /(S_n+ Z_n)$ we have $(S_{n+1},Z_{n+1}) = (S_n,Z_n)+(a_{11},a_{12})$ and otherwise $(S_{n+1},Z_{n+1}) = (S_n,Z_n)+(a_{21},a_{22})$.
\end{definition}

\begin{remark}\label{rem:dff_rates_and_urn_model}
The process $(X_n,Y_n)$ such that $(X_{n+1},Y_{n+1}) = (X_n+a_1,Y_n)$,
with the probability $\alpha_1 X_n/(\alpha_1 X_n + \alpha_2 Y_n)$ and
$(X_{n+1},Y_{n+1}) = (X_n,Y_n+a_2)$ otherwise, can be thought of as an
urn process in which we draw balls with different weights. It is easy
to observe that the process $(S_n,Z_n) = (\alpha_1 X_n,\alpha_2 Y_n)$
is indeed an urn process with $S_0 = \alpha_1 X_0$, $Z_0 = \alpha_2
Y_0$ and the replacement matrix $\left(\begin{array}{ll} \alpha_1 a_1
  & 0 \\ 0 & \alpha_2 a_2\end{array}\right)$.
\end{remark}

From the jump probabilities of $(X_n,Y_n,Z_n)$  we immediately obtain the following lemma, which is true even in the stage 0 of CP, that is for $Y_0=0$.

\begin{lemma}\label{lemma:urn_in_configuration}
The process $(Z_n,X_n+Y_n-1)$ is an urn model with the replacement  matrix 
\[
A = \left(\begin{array}{rr} -d & d-2 \\ 0 &
  -2 \end{array} \right)~.
\]
\end{lemma}
The following result will be used extensively throughout the paper. Having in mind Lemma \ref{lemma:urn_in_configuration}, and the fact that $X_n , Y_n , Z_n$ pertain their values after the process is stopped due $X_n +Y_n=0$, 
it follows directly from Theorems 1.1 and 1.4 in \cite{antunovic}. 

\begin{theorem}\label{thm: urn applied to model}
For any $\varepsilon > 0$ and any sequence $(L_M)$ converging to $\infty$ there exists a sequence $\delta_M$ (which doesn't depend on $X_0$ and $Y_0$), such that in the CP model started with $X_0$ blue and $Y_0$ red half-edges
and $Z_0=M-X_0-Y_0$ uncolored half-edges, for any $\varepsilon > 0$
the events
\[
\left\{ Z_n = (1\pm \varepsilon) Z_0(1-2n/M)^{d/2} ,
\text{ for all } 0 \leq n \leq \frac{M}{2}\Big(1 - \frac{L_M}{Z_0^{2/d}}\Big) \right\},
\]
\[
\left\{ Z_n =0,
\text{ for all } n \geq \frac{M}{2}\Big(1 - \frac{1}{Z_0^{2/d}L_M}\Big) \right\},
\]
and
\[
\left\{
X_n+Y_n = (1\pm \varepsilon)\Big((M-2n)- Z_0(1-2n/M)^{d/2}\Big), \text{ for all } 0 \leq n < M/2 \right\}
\]
have probabilities at least $1-\delta_M$. 
\end{theorem}

The following lemma relates the number of final number number of colored half-edges and colored vertices in the stage 0 of CP. This is relevant since the natural initial parameter $B_0$ appearing in the statements of main theorems, has to be replaced by  $X_0$ for the application of theorems \ref{thm:(1+o(1))_concentration_diff_rates_0} and \ref{thm: 1+o(1) concentration for boundary}. The lemma is a trivial consequence that in stage 0 of CP, in most steps we are coloring new vertices and thus increase the value of $X_n^0$ by $d-2$ and the value of $B_n^0$ by 1. 

\begin{lemma}\label{lemma: first stage analysis}
Assume that $(\mathcal{B}_0, \mathcal{R}_0)=(\mathcal{B}_0(N), \mathcal{R}_0(N))$ is a sequence of uniform subsets of size $(B_0,R_0)$ with $\mathcal{B}_0$ of small center (as in Definition \ref{def: early initial sets}) and $\lim_N B_0/N=0$. 
Then, for any $\varepsilon > 0$ the probability that the stage 0 of CP that generates $\mathcal{B}_0$ ends with
\[
X_0 = (1\pm \varepsilon) (d-2)B_0,
\]
converges to 1, as $N \to \infty$.
\end{lemma}

\begin{proof}
Recall the notation from the definition of stage 0 of CP.
If at the $n$-th step of stage 0 CP we spread to a new vertex then we have $B_{n+1}^0 = B_n^0 +1$ and $X_{n+1}^0 = X_n^0 + d-2$. Otherwise, $B_{n+1}^0 = B_n^0$ and $X_{n+1}^0 = X_n^0-2$. This leads to a simple deterministic relation 
\begin{equation}\label{eq:early_start_deterministic}
X_n^0 - X_0^0 = d(B_n^0 - B_0^0) - 2n.  
\end{equation}
By the definition of stage 0 of CP we have $B_0/X_0^0 \to \infty$.
Next we will show that for every $\varepsilon > 0$ there exists $\delta > 0$ such that for any sequence $L_N \to \infty$ 
\begin{equation}\label{eq:early_start_main_estimate}
X_n^0 - X_0^0 = (1\pm \varepsilon)(d-2)n, \text{ for all } X_0^0 L_N \leq n \leq \delta N,
\end{equation}
has probability which converges to $1$, as $N \to \infty$. Actually we don't need the lower bound $X_0^0L_N$ on $n$ here, (see the proof of Theorem \ref{thm:small_sizes}) but this will suffice for our needs in this lemma.
By \eqref{eq:early_start_deterministic} the event in \eqref{eq:early_start_main_estimate} implies
\begin{equation}\label{eq:early_start_2}
 B_n^0 - B_0^0 = (1\pm \varepsilon)n , \text{ for all } X_0^0L_N \leq n \leq \delta N.
\end{equation}
The lower bound above can reach the value of $(1-\varepsilon)\delta N$ and since $\lim_N B_0/N = 0$, for $N$ large enough we have $B_0 < (1-\varepsilon)\delta N$. Also taking $L_n$ so that $X_0^0L_N/B_0 \to 0$, the inequality in \eqref{eq:early_start_2} is also true for $n=T$ and since $B_0 = B_T^0$ 
\[
\lim_{N \to \infty}\mathbb{P}\left(\frac{B_0 -B_0^0}{1+\varepsilon} \leq T \leq \frac{B_0 -B_0^0}{1-\varepsilon}\right) =1.
\]
Then $n=T$ in \eqref{eq:early_start_deterministic} gives that with the probability converging to $1$
\[
 (B_0-B_0^0) \Big(d-\frac{2}{1-\varepsilon}\Big) \leq X_0 - X_0^0\leq (B_0-B_0^0) \Big(d-\frac{2}{1+\varepsilon}\Big) .
\]
Since $\varepsilon >0$ is arbitrary, $X_0^0 = dB_0^0$ and $B_0^0/B_0 \to 0$ the claim follows. 

Thus we are left to prove the claim in \eqref{eq:early_start_main_estimate}.
Since $Y_0^0=0$, we have $Y_n^0=0$ for all $n$, and Theorem \ref{thm: urn applied to model} implies that for every $\varepsilon > 0$
\[
X_n^0 = (1\pm \varepsilon)\Big((M^0-2n)- Z_0^0(1-2n/M^0)^{d/2}\Big),  \text{ for all }0 \leq n < M^0/2,
\]
with the probability converging to 1.
Recalling that $Z_0^0 = M^0-X_0^0$ the above event yields that $X_n^0 -X_0^0$ is at least
\[
(1-\varepsilon)(M^0-2n)\left(1-\Big(1-\frac{2n}{M^0}\Big)^{d/2-1}\right) - X_0^0\left(1-\Big(1-\frac{2n}{M^0}\Big)^{d/2}\right) - \varepsilon X_0^0,
\]
and at most
\[
(1+\varepsilon)(M^0-2n)\left(1-\Big(1-\frac{2n}{M^0}\Big)^{d/2-1}\right) + \varepsilon X_0^0. 
\]
As $M^0/N \to d$, we can choose $\delta$ small enough so that for all $0 \leq n \leq \delta N$ we have
\[
(1-\varepsilon/2)\frac{\lambda n}{M^0} \leq 1- \Big(1-\frac{2n}{M^0}\Big)^{\lambda/2} \leq (1+\varepsilon/2)\frac{\lambda n}{M^0},
\]
for $\lambda = d-2$ and $\lambda =d$. Using these inequalities, recalling $X_0^0$ is constant and reducing $\delta$ if necessary 
 yields \eqref{eq:early_start_main_estimate}.
\end{proof}

To deduce Theorems \ref{thm:main_both_large} to \ref{thm:boundary} we also need the following estimates on the function $\phi_{\beta,\rho}^{-1}$.

\begin{lemma}\label{lemma:inverse_estimates}
Let $\beta \neq \rho$ and the function $\phi_{\beta,\rho} \colon (0,1) \to
\mathbb{R}^+$ as in \eqref{eq: phi 2 parameter function}.
Then 
$\phi_{\beta,\rho}$ is one-to-one and onto and there are constants $c_1 <
c_2$ such that the inverse function $\phi_{\beta,\rho}^{-1} \colon
\mathbb{R}^+ \to (0,1)$ satisfies
\[
c_1 \Big(s^{(\rho-\beta)/\beta} \wedge 1\Big) \leq \phi_{\beta,\rho}^{-1}(s) \leq
\Big(c_2 s^{(\rho-\beta)/\beta}\Big) \wedge 1.
\]
Moreover,
\[
\frac{\phi_{\beta,\rho}^{-1}(s)}{s^{(\rho-\beta)/\beta}} \to \frac{\rho}{\beta}, \left\{\begin{array}{ll}\text{ as } s \to \infty, & \text{ for } \beta > \rho,  \\ \text{ as } s \to 0, & \text{ for } \rho > \beta. \end{array}\right.
\]
\end{lemma}

\begin{proof}
For $\beta > \rho$ the function $\phi_{\beta,\rho}$ is decreasing and
$\phi_{\beta,\rho}(1)=0$ and $\lim_{t \downarrow 0}
\frac{\phi_{\beta,\rho}(t)}{t^{\beta/(\rho-\beta)}}=(\rho/\beta)^{\beta/(\beta-\rho)}$. Therefore
$\phi_{\beta,\rho}^{-1}$ is decreasing with $\phi_{\beta,\rho}^{-1}(0)=1$ and
$\lim_{s \to \infty} \frac{\phi_{\beta,\rho}^{-1}(s)}{s^{(\rho-\beta)/\beta}}
=\frac{\rho}{\beta}$. This proves the claim for $\beta > \rho$. For $\beta < \rho$ the
function $\phi_{\beta,\rho}$ is increasing with $\phi_{\beta, \rho}(0)=0$, $\lim_{t
  \downarrow 0}\frac{\phi_{\beta,\rho}(t)}{t^{\beta/(\rho-\beta)}} =(\rho/\beta)^{\beta/(\beta-\rho)}$ and
$\lim_{t \uparrow 1} \phi_{\beta,\rho}(t) = \infty$. This implies that
$\phi_{\beta,\rho}^{-1}$ is also increasing and $\phi_{\beta,\rho}^{-1}(0) = 0$,
$\lim_{s\downarrow 0}\frac{\phi_{\beta,\rho}^{-1}(s)}{s^{(\rho-\beta)/\beta}}=\frac{\rho}{\beta}$ and
$\lim_{s \to \infty}\phi_{\beta,\rho}^{-1}(s) = 1$, which is enough to
deduce the claim in the case $\beta < \rho$.
\end{proof}

\begin{lemma}\label{lemma:unif_continuity_of_log}
For any $\varepsilon>0$ there exist a $\delta > 0$ such that $t = (1\pm \delta)s$ implies both 
\[
\phi_{\beta,\rho}^{-1}(t) = (1\pm \epsilon)\phi_{\beta,\rho}^{-1}(s)
\ \text{ and } \ 
1 -\phi_{\beta,\rho}^{-1}(t) = (1\pm \epsilon)(1- \phi_{\beta,\rho}^{-1}(s)).
\]
In other words, $\log \phi_{\beta,\rho}^{-1}$ and $\log (1-\phi_{\beta,\rho}^{-1})$ are both uniformly continuous on $\mathbb{R}^+$.
\end{lemma}

\begin{proof}
For the first inequality
we need to show that for any $\varepsilon > 0$ there is a $\delta > 0$ such that $x / y > 1+\delta$ implies that $\phi_{\beta,\rho}(x)/\phi_{\beta,\rho}(y)$ is either larger than $1+\varepsilon$ or smaller than $1-\varepsilon$. This holds since this is true when $\phi_{\beta,\rho}$ is replaced by either of the two summands on the right hand side of \eqref{eq: phi 2 parameter function}. For the second claim
use the symmetry $\phi_{\rho,\beta}^{-1}(t) = 1-\phi_{\beta,\rho}^{-1}(t)$.
\end{proof}


The following corollary clarifies the asymptotic behavior of $R_{\text{fin}}-R_0$. Note that $s \wedge t = \min(s,t)$.

\begin{corollary}\label{cor:nicer_main_estimates}
Let the assumptions in Theorem \ref{thm:(1+o(1))_concentration_diff_rates_0} hold.
Then there are constants $c$ and $C$ such that with probability converging to 1 (as
$N \to \infty$)
\begin{equation}\label{eq:1+o(1)_concentration_up_t_const_factor}
c \frac{Y_0/M}{(X_0/M)^{\rho/\beta}}\wedge c \leq
\frac{R_{\text{fin}} - R_0}{M/d} \leq
C\frac{Y_0/M}{(X_0/M)^{\rho/\beta}} \wedge 1.
\end{equation}
Moreover, if $\beta \neq \rho$ and $\frac{(Y_0/M)^\beta}{(X_0/M)^\rho}$ converges to zero, then with probability converging to 1 
\begin{equation}\label{eq:limiting_concentration}
\frac{R_{\text{fin}} - R_0}{M/d} = (1\pm\varepsilon)\frac{\rho Y_0/M}{\beta(X_0/M)^{\rho/\beta}}\int_0^1 \Big(t^{1/d}-t^{(d-1)/d}\Big)^{\rho/\beta -1}~dt.
\end{equation}
\end{corollary}


\begin{proof}
First we prove \eqref{eq:1+o(1)_concentration_up_t_const_factor}.
For
$\beta=\rho$ the inequalities in
\eqref{eq:1+o(1)_concentration_up_t_const_factor} with $C=1$ and $c=1/2$ are easy to check from \eqref{eq:1+o(1)_concentration_beta_equal_1},
so we focus on the case $\beta \neq \rho$.  By \eqref{eq:1+o(1)_concentration_beta_no_1} and Lemma
\ref{lemma:inverse_estimates} we have
\begin{multline*}
c_1 \int_0^1 \frac{M^{\rho/\beta
    -1} Y_0}{X_0^{\rho/\beta}}\Big(t^{1/d} -
t^{(d-1)/d}\Big)^{\rho/\beta-1} \wedge 1 \ dt \leq
\frac{R_{\text{fin}} - R_0}{M/d} \\ \leq \int_0^1
c_2\frac{M^{\rho/\beta -1}Y_0}{X_0^{\rho/\beta}}\Big(t^{1/d} -
t^{(d-1)/d}\Big)^{\rho/\beta-1} \wedge 1 \ dt.
\end{multline*}
The trivial inequality $R_{\text{fin}} - R_0 \leq M/d$ yields $1$
for the upper bound in
\eqref{eq:1+o(1)_concentration_up_t_const_factor}, and the constant lower bound in \eqref{eq:1+o(1)_concentration_up_t_const_factor} is obvious. 
For other  bounds in
\eqref{eq:1+o(1)_concentration_up_t_const_factor}, it suffices to show that 
\begin{equation}\label{eq:1+o(1)_concentration_integral_help}
\int_{0}^1\Big( t^{1/d} -
t^{(d-1)/d}\Big)^{\rho/\beta-1 } \ dt
\end{equation}
is bounded from above and below by strictly positive constants depending on $\rho/\beta$ and $d$ only.
This follows from the fact that the function under the integral is integrable on $(0,1)$ and bounded away from zero on the interval
 $[1/2,3/4]$.

To prove \eqref{eq:limiting_concentration} observe that the condition implies that  $MX_0^{\rho/(\beta-\rho)}Y_0^{\beta/(\rho-\beta)} \to \infty$, for $\beta >\rho$ and $MX_0^{\rho/(\beta-\rho)}Y_0^{\beta/(\rho-\beta)} \to 0$, for $\beta <\rho$, and apply Lemma \ref{lemma:inverse_estimates} and dominated convergence theorem.
\end{proof}

In the same way we get the analogous estimates on the boundary. Bounds for the size of the boundary analogous to those in \eqref{eq:1+o(1)_concentration_up_t_const_factor} are  straightforward from \eqref{eq:1+o(1)_concentration_up_t_const_factor}, as random regular graphs are expanders asymptotically almost surely, so we present the analogues of \eqref{eq:limiting_concentration}.

\begin{corollary}\label{cor:nicer_boundary_estimates}
Let the assumptions in Theorem \ref{thm:(1+o(1))_concentration_diff_rates_0} hold.
Then with probability converging to 1
\begin{equation}\label{eq:limiting_concentration_boundary}
D_{\text{fin}} = (1\pm\varepsilon)\frac{\beta+ \rho}{\beta}\frac{Y_0}{(X_0/M)^{\rho/\beta}}\int_0^1 \Big(t-t^{d-1}\Big)^{\rho/\beta}~dt.
\end{equation}
when $\beta \neq \rho$ and $\frac{(Y_0/M)^\beta}{(X_0/M)^\rho}$ converges to zero. 
\end{corollary}

\begin{proof}
Applying Dominated convergence theorem as in the proof of the second part of Corollary \ref{cor:nicer_main_estimates} and the fact that $\lim_{t \to 0}\kappa_{\beta,\rho}(t)/t = (\beta + \rho)/\rho$ we get
\[
D_{\text{fin}} = (1\pm\varepsilon)\frac{\beta+ \rho}{2\beta}\frac{Y_0}{(X_0/M)^{\rho/\beta}}\int_0^1 \Big(1-t^{d/2-1}\Big)\Big(t^{1/2}- t^{(d-1)/2}\Big)^{\rho/\beta -1}~dt.
\]
Formula \eqref{eq:limiting_concentration_boundary} is obtained by substituting the integration variable with $t^{1/2}$.
\end{proof}

Now we are ready to prove Theorem \ref{thm:main_both_large}. Proof of Theorem \ref{thm:boundary} part i) is identical and will be omitted.

\begin{proof}[Proof of Theorem \ref{thm:main_both_large}]
In the setting of Theorem \ref{thm:main_both_large} we have $(X_0,Y_0)=(d\delta_B B_0,d\delta_R R_0)$, $M/N \to d$ and $Z_0/M\to 1$. The existence of sequences $(\overline{B},\overline{R})$  is straightforward from Theorem \ref{thm:(1+o(1))_concentration_diff_rates_0}, and their estimates from  \eqref{eq:1+o(1)_concentration_up_t_const_factor}. The formulas for $\overline{B}$ and $\overline{R}$ in the case $\beta = \rho$ follow from \eqref{eq:1+o(1)_concentration_beta_equal_1} 
and the fact that $(B_0+R_0)/N  \to 0$.  The formulas in the case $\beta \neq \rho$ follow from \eqref{eq:limiting_concentration}. 
\end{proof}

To prove Theorem \ref{thm:small_sizes} we compare the process $(X_n,Y_n)$ from the coupled process CP to an urn model.
For the case of equal rates $\beta = \rho$, recall that in the P\'{o}lya urn process $(S_n,Z_n)$ with the replacement matrix $\left(\begin{array}{ll} \alpha & 0 \\ 0 &
  \alpha \end{array}\right)$ random variable $S_n/(\alpha n)$ converges in distribution to $Beta(S_0/\alpha,Z_0/\alpha)$.
If the rates are different,
we use the following
result by Svante Janson (part of Theorem 1.4 in \cite{Janson06}). 

\begin{theorem}[Janson]\label{thm:svante_theorem}
Consider the P\'{o}lya urn process $(S_n,Z_n)$ with the replacement
matrix $\left(\begin{array}{ll} \alpha & 0 \\ 0 &
  \delta \end{array}\right)$, where $\alpha>0$, $\delta > 0$, $S_0 >
0$ and $Z_0 > 0$. Let $U \sim \Gamma(S_0/\alpha,1)$ and $V \sim
\Gamma(Z_0/\delta,1)$ be two independent random variables with Gamma
distribution and parameters $S_0/\alpha$ and $Z_0/\delta$
respectively.
If $\alpha < \delta$ then in distribution
\[
\frac{S_n}{n^{\alpha/\delta}} \rightarrow \alpha \frac{U}{V^{\alpha/\delta}}.
\]
\end{theorem}

Now we prove Theorem \ref{thm:small_sizes}. Again the proof of Theorem \ref{thm:boundary} ii) is analogous and will be omitted.
\begin{proof}[Proof of Theorem \ref{thm:small_sizes}]
Consider the process $(X_{n,1},Y_{n,1})$ such that $(X_{0,1},Y_{0,1})
= (dB_0,dR_0)$ and that $(X_{n+1,1},Y_{n+1,1}) =(X_{n,1} + d-2,Y_{n,1})$ happens
with probability $\beta X_{n,1}/(\beta X_{n,1} + \rho Y_{n,1})$, and
$(X_{n+1,1},Y_{n+1,1}) =(X_{n,1},Y_{n,1}+d-2)$ otherwise. 
First we show that we can couple the processes $(X_n,Y_n)$ and
$(X_{n,1},Y_{n,1})$ so that with probability converging to $1$, as $M
\to \infty$, we have $(X_n,Y_n) = (X_{n,1},Y_{n,1})$ for all $0 \leq n
\leq M^{1/4}$. 
The construction of the
coupling is simple: we set $(X_{n+1,1},Y_{n+1,1}) =(X_{n,1} +
d-2,Y_{n,1})$ if in the step ii) of CP process the first chosen half-edge
was blue (that is in $\mathcal{X}_n$) and we set $(X_{n+1,1},Y_{n+1,1}) =(X_{n,1},Y_{n,1} + d-2)$ otherwise.
Then we have that $(X_n,Y_n) = (X_{n,1},Y_{n,1})$ if for all $k
\leq n$, in the step ii) of the $k$-th round of CP we colored a new vertex,  (in other words coupled a half-edge in $\mathcal{Z}_n$).
 The conditional probability that this does not happen is
\[
\frac{\beta X_k(X_k-1) + \rho Y_k(Y_k-1) + (\beta + \rho)X_kY_k}{(\beta X_k +\rho
  Y_k) (X_k + Y_k + Z_k -1)} \leq 2\frac{X_k + Y_k}{X_k+Y_k+Z_k-1}\leq
\frac{4(d-2)}{M^{3/4}}.
\]
Here we used the fact that $X_n + Y_n$ can increase by at most $d-2$ and the sum $X_n +Y_n +Z_n$ can decrease by at most 2.
Now the probability that in each of the first $M^{1/4}$  rounds of CP we color a new vertex is bounded from below by
\[
1 - M^{1/4} \frac{4(d-2)}{M^{3/4}},
\]
which converges to $1$ as $M \to \infty$. 
Set $n(M)$ to be the integer part of $M^{1/4}$. In the case $\beta = \rho$, beta convergence of equal rate P\'{o}lya urns 
yields that in distribution
\[
\frac{Y_{n(M)}}{Y_{n(M)} +  X_{n(M)}} \to W.
\]
Stopping the CP process at $n(M)$, and starting it again with the initial $X_{n(M)}$ and $Y_{n(M)}$, apply the Markov property and Theorem \ref{thm:(1+o(1))_concentration_diff_rates_0} to get the claim. The proof for the different rates is analogous, if one uses non-balanced urn result in Theorem \ref{thm:svante_theorem} and Remark \ref{rem:dff_rates_and_urn_model} 
to get that for $\rho < \beta$
\[
\frac{\rho Y_{n(M)}}{(X_{n(M)}/(d-2))^{\rho/\beta}} \to \rho (d-2)\frac{U}{V^{\rho/\beta}}.
\]
In both the balanced ($\beta =\rho$) and the unbalanced case ($\beta \neq \rho$) the application of Theorem \ref{thm:(1+o(1))_concentration_diff_rates_0} is justified, as with probability converging to 1, we have that both $X_{n(M)}\to \infty$ and $Y_{n(M)}\to \infty$.
\end{proof}

Theorem \ref{thm:svante_theorem} can not be directly applied in the setting of Theorem \ref{thm:mixed_sizes}.
We will need the following lemma, which holds for all values of $\beta$ and $\rho$.

\begin{lemma}\label{lemma:small_and_large_when_to_stop}
Assume that $X_0=X_0(M)$ is constant and that $Y_0=Y_0(M)$ converges to $\infty$ as $M \to \infty$, so that $Y_0/M \to 0$. Then there exists a sequence $L_M \to \infty$  such that for the stopping time $\tau$ defined as 
the first
$k$ for which $X_k \geq X_0 L_M^{\beta/\rho}$ or $Y_k \geq Y_0 L_M$,
we have the convergence in distribution
 \[
 \frac{X_\tau}{(Y_\tau/Y_0)^{\beta/\rho}} \rightarrow \Gamma\Big(\frac{X_0}{d-2},\frac{1}{d-2}\Big),
\] as $M \to \infty$.
\end{lemma}

First we show how to use the above lemma to prove Theorem \ref{thm:mixed_sizes}. Yet again the proof of Theorem \ref{thm:boundary} iii) is analogous and will be omitted.

\begin{proof}[Proof of Theorem \ref{thm:mixed_sizes}]
The proof proceeds along the same steps as the proofs of Theorems \ref{thm:main_both_large} and \ref{thm:small_sizes}. 
In particular, we stop the process at the stopping time $\tau$, restart and use strong Markov property.
Simply observe that by Lemma \ref{lemma:small_and_large_when_to_stop} that both $\beta > \rho$ and $R_0N^{\rho/\beta}/N\to 0$ imply that $\frac{(Y_\tau/M)^\beta}{(X_\tau/M)^{\rho}} \to 0$ and that either $\beta \leq \rho$ or both $\beta > \rho$ and $(N/R_0)^{\beta/\rho}/N\to 0$ imply that $\frac{(X_\tau/M)^\rho}{(Y_\tau/M)^{\beta}} \to 0$. In the former case use \eqref{eq:limiting_concentration}, and in the latter case use it's analog  
for the blue process. 
\end{proof}

To prove Lemma \ref{lemma:small_and_large_when_to_stop} we need the
following corollary of Theorem \ref{thm:svante_theorem}. 
First say that a process $(S_n,Z_n)_n$ is a P\'{o}lya urn process with a replacement matrix $A$ and with respect to a filtration $(\mathcal{F}_n)$ if the conditional probability of $(S_{n+1},Z_{n+1})$ given $\mathcal{F}_n$ agrees with that of a P\'{o}lya urn process with replacement matrix $A$.

\begin{corollary}\label{cor:svante_theorem_applied}
Consider a P\'{o}lya urn process $(S_n,Z_n)_n$ with the replacement
matrix $\left(\begin{array}{ll} \alpha & 0 \\ 0 &
  \delta \end{array}\right)$, where $\alpha>0$, $\delta > 0$ and with respect to a filtration $(\mathcal{F}_n)$.  Let
$(L_M)$, $(L_{M,S})$ and $(L_{M,Z})$  be three sequences converging to $\infty$ as $M \to \infty$, let $S_0 = S_0(M)
> 0$ be fixed and let $Z_0 = Z_0(M) \geq L_M$.  Let  $\tau$ be a (almost surely finite) stopping time
such that $\mathbb{P}(S_\tau < S_0L_{M,S}, \ Z_\tau < Z_0 L_{M,Z}) = 0$. Then
\[
Z_0^{\alpha/\delta}\frac{S_\tau}{Z_\tau^{\alpha/\delta}} \rightarrow
\Gamma(S_0/\alpha,1/\alpha),
\]
in distribution, as $M \to \infty$. 
\end{corollary}

\begin{proof}
Define $U$ and $V$ as $\Gamma(S_0/\alpha,1)$ and
$\Gamma(Z_0/\delta,1)$ distributed independent random
variables. Furthermore 
conditioned on $\mathcal{F}_\tau$ take
$U'$ and $V'$ to be $\Gamma(S_\tau/\alpha,1)$ and
$\Gamma(Z_\tau/\delta,1)$ distributed independent random
variables. Recall that the Gamma distribution $\Gamma(m,1)$ has mean
and variance equal to $m$, which by Chebyshev inequality implies that
for any $\varepsilon$ the probability that a $\Gamma(m,1)$ random
variable is in the interval $((1-\varepsilon)m, (1+\varepsilon)m)$ converges
to $1$, as $m \to \infty$. Since $Z_\tau \geq Z_0 \geq L_M$, this concentration
holds for both $V$ and $V'$ ($m= Z_0/\delta$ and $m = Z_\tau/\delta$,
respectively), implying there is a sequence $\rho_{M,\varepsilon}$
(depending only on $(L_M)$) such that $\lim_{M \to
  \infty}\rho_{M,\varepsilon} = 0$, and
\begin{equation}\label{eq:janson_applied_0}
\mathbb{P}((1-\varepsilon/2)Z_0/\delta \leq V \leq (1+\varepsilon/2)Z_0/\delta) \geq
1-\rho_{M,\varepsilon}
\end{equation}
and
\begin{equation}\label{eq:janson_applied_0.1}
 \mathbb{P}((1-\varepsilon/2)Z_\tau/\delta \leq V' \leq (1+\varepsilon/2)Z_\tau/\delta | \mathcal{F}_\tau)
 \geq 1-\rho_{M,\varepsilon}.
\end{equation}

Next we show that for any $x > 0$ and  $\varepsilon > 0$ there is a sequence
$(\lambda_{M,\varepsilon})_M$ converging to $0$ such that
\begin{equation}\label{eq:janson_applied_1}
\mathbb{P} \left( \left. \frac{U'}{V'^{\alpha/\delta}} \leq
(1+\varepsilon)\frac{x/\alpha}{(Z_0/\delta)^{\alpha/\delta}} \right|
\mathcal{F}_\tau  \right) \geq 1-
\lambda_{M,\varepsilon}, \ \text{ on the event } \left\{\frac{S_\tau}{Z_\tau^{\alpha/\delta}} \leq
\frac{x}{Z_0^{\alpha/\delta}}\right\}
\end{equation}
and
\begin{equation} \label{eq:janson_applied_1.1}
\mathbb{P}\left(\left.\frac{U'}{V'^{\alpha/\delta}} \geq
(1-\varepsilon)\frac{x/\alpha}{(Z_0/\delta)^{\alpha/\delta}} \right|
\mathcal{F}_\tau  \right) \geq 1-
\lambda_{M,\varepsilon}, \ \text{ on the event } \left\{\frac{S_\tau}{Z_\tau^{\alpha/\delta}} \geq
\frac{x}{Z_0^{\alpha/\delta}}\right\}.
\end{equation}
On the event $S_\tau \geq (S_0 L_{M,S}) \wedge L_{M,Z}^{\alpha/(2\delta)}$  \eqref{eq:janson_applied_1} and
\eqref{eq:janson_applied_1.1} follow from
\eqref{eq:janson_applied_0.1} and the analogous bound for $U'$.  If
$S_\tau < (S_0 L_{M,S}) \wedge L_{M,Z}^{\alpha/(2\delta)}$ then we have $Z_\tau \geq
Z_0L_{M,Z}$ and,  since $S_\tau Z_\tau^{-\alpha/\delta} \leq
Z_0^{-\alpha/\delta}L_{M,Z}^{-\alpha/(2\delta)}$, we only need to prove \eqref{eq:janson_applied_1}.  This then follows from
the fact that $\frac{U'}{V'^{\alpha/\delta}} >
(1+\varepsilon)\frac{x/\alpha}{(Z_0/\delta)^{\alpha/\delta}}$
 implies either 
\[V' \leq \frac{(1-\varepsilon)Z_\tau}{\delta},\  \text{ or } \ U' \geq
\frac{(1+\varepsilon)x}{\alpha}((1-\varepsilon)L_{M,Z})^{\alpha/\delta}.\]
 The first
event has probability bounded from above by $\rho_{M,\varepsilon}$ and
the second one, by Markov inequality, has probability at most
\[\frac{S_\tau}{(1+\varepsilon)x((1-\varepsilon)L_{M,Z})^{\alpha/\delta}} \leq \frac{1}{(1+\varepsilon)x((1-\varepsilon)L_{M,Z}^{1/2})^{\alpha/\delta}},\]
which converges to zero as $M \to \infty$. This proves \eqref{eq:janson_applied_1} and \eqref{eq:janson_applied_1.1}.

If $\alpha < \delta $ then, by Theorem \ref{thm:svante_theorem},
conditionally on $\mathcal{F}_\tau$ the random variable
$(n-\tau)^{-\alpha/\delta}S_n$ converges in distribution to $\alpha
U'V'^{-\alpha/\delta}$, as $n \to \infty$. In particular, by \eqref{eq:janson_applied_1} we have
\[
\liminf_{n \to \infty} \mathbb{P}\left(\left.\frac{S_n}{(n-\tau)^{\alpha/\delta}} \leq (1+\varepsilon)\frac{x}{(Z_0/\delta)^{\alpha/\delta}}\right|\mathcal{F}_\tau\right) \geq 1- \lambda_{M,\varepsilon}, \ \text{ on the event } \left\{\frac{S_\tau}{Z_\tau^{\alpha/\delta}} \leq
\frac{x}{Z_0^{\alpha/\delta}}\right\}.
\]
Integrating over this event and using Fatou's lemma one easily gets
\begin{align*}
\mathbb{P}\left(\frac{S_\tau}{Z_\tau^{\alpha/\delta}}  \leq
\frac{x}{Z_0^{\alpha/\delta}} \right) & \leq
(1-\lambda_{M,\varepsilon})^{-1}\liminf_{n \to \infty}\mathbb{P}\left(\frac{S_n}{(n-\tau)^{\alpha/\delta}}
\leq (1+\varepsilon)\frac{x}{(Z_0/\delta)^{\alpha/\delta}}\right) \\ &
  \leq
(1-\lambda_{M,\varepsilon})^{-1}\liminf_{n \to \infty}\mathbb{P}\left(\frac{S_n}{n^{\alpha/\delta}}
\leq (1+\varepsilon)\frac{x}{(Z_0/\delta)^{\alpha/\delta}}\right) 
\end{align*}
Again, by Theorem \ref{thm:svante_theorem}, the probability on the right hand side
 converges to 
\[\mathbb{P}\Big( \alpha UV^{-\alpha/\delta} \leq
(1+\varepsilon)x\big(Z_0/\delta\big)^{-\alpha/\delta}\Big) \leq
\mathbb{P}( \alpha U \leq (1+ \varepsilon)(1+\varepsilon/2)^{\alpha/\delta}x) +
\rho_{M,\varepsilon}.\] 
Taking $M \to \infty$ and then $\varepsilon$ to zero
\[
\limsup_{M \to
  \infty}\mathbb{P}\left(Z_0^{\alpha/\delta}\frac{S_\tau}{Z_\tau^{\alpha/\delta}}
\leq x\right) \leq \mathbb{P}( \alpha U \leq x),
\]
Using \eqref{eq:janson_applied_1.1} and the fact that $\tau$ is finite almost surely, we can obtain the
upper bound on
$\mathbb{P}\Big(Z_0^{\alpha/\delta}\frac{S_\tau}{Z_\tau^{\alpha/\delta}}
\geq x\Big)$ in the same way which proves the claim when $\alpha <
\delta$.

By using the same arguments as in the previous case, for $\delta <
\alpha $, conditionally on $\mathcal{F}_\tau$ the random variable
$n^{-\delta/\alpha}Z_n$ converges in distribution to $\delta
V'U'^{-\delta/\alpha}$ and thus for a fixed $M$ and $n$ large enough
we have
\[
\mathbb{P}\Big(\frac{S_\tau}{Z_\tau^{\alpha/\delta}} \leq
\frac{x}{Z_0^{\alpha/\delta}} \Big) \leq
(1-\lambda_{M,\varepsilon})^{-1}\liminf_{n \to \infty}\mathbb{P}\Big(n^{-\delta/\alpha}Z_n
\geq
\frac{Z_0}{((1+\varepsilon)x/\alpha)^{\delta/\alpha}}\Big).
\]
The probability on the right hand side above converges to
$\mathbb{P}\Big( \delta VU^{-\delta/\alpha} \geq
Z_0((1+\varepsilon)x/\alpha)^{-\delta/\alpha}\Big)$ which in turn
is bounded by $\mathbb{P}(\alpha U \leq (1+ \varepsilon)^{1+\alpha/\delta}x) +
\rho_{M,\varepsilon}$. Taking $M \to \infty$ and $\varepsilon \to 0$ proves
one bound in this case. As before the other bound is proven in the
same way.

The case $\alpha = \delta$ is handled in exactly the same way. First note that the proofs of the inequalities \eqref{eq:janson_applied_1} and \eqref{eq:janson_applied_1.1} are valid in this case as well. Use the Beta convergence in this case, and for that note that $S_\tau/Z_\tau \leq x/Z_0$
and $S_\tau/Z_\tau \geq x/Z_0$ are equivalent to $S_\tau/(S_\tau +
Z_\tau) \leq x/(x + Z_0)$ and $S_\tau/(S_\tau + Z_\tau) \geq x/(x +
Z_0)$ respectively, and that $U/(U+V)$ and $U'/(U'+V')$ are distributed as $\text{Beta}(S_0/\alpha,Z_0/\alpha)$ and $\text{Beta}(S_\tau/\alpha,Z_\tau/\alpha)$ respectively.
\end{proof}

Now we prove Lemma \ref{lemma:small_and_large_when_to_stop}.

\begin{proof}[Proof of Lemma \ref{lemma:small_and_large_when_to_stop}]

Denote by $\sigma_1$ the first time $k$ that $X_{k} < X_{k-1}$ and for $\varepsilon>0$ by $\sigma_{2,\varepsilon}$ the first $k$ such that there are at least $\varepsilon Y_0/d$ indices $\ell < k$ for which either $X_{\ell+1} \neq X_\ell$ or $Y_{\ell+1} < Y_\ell$. In the first step we claim that we can choose the sequence $L_M$ so that 
\begin{equation}\label{eq:choice for bound on stopping time}
\lim_{M \to \infty} \mathbb{P}(\sigma_1 \wedge \sigma_{2,\varepsilon} \leq \tau) = 0,
\end{equation}
holds for all $\varepsilon > 0$. 

Denote by $k_0$ the integer part of $2X_0L_M^{\beta/\rho}+2Y_0L_M$. By Theorem \ref{thm: urn applied to model} we have that $\mathbb{P}(\tau < k_0)$ converge to zero. Thus the claim will be proven if we find  a sequence $(L_M)$ such that $\mathbb{P}(\sigma_1 \leq \tau \wedge k_0)$ and $\mathbb{P}(\sigma_{2,\varepsilon} \leq \tau \wedge k_0)$ both converge to zero for any $\varepsilon > 0$.

First we show that we choose $L_M$ to grow slow enough so that $\mathbb{P}(\sigma_1 \leq \tau \wedge k_0)$ converges to zero.
Given $k<\tau \wedge k_0$ the conditional probability that
$X_{k+1} <X_k$ is equal to
\[
\frac{\beta X_k(X_k -1) + (\rho+\beta)X_kY_k}{(\beta X_k + \rho Y_k)(X_k + Y_k
  +Z_k -1)} \leq \frac{(1+\beta/\rho)X_k}{X_k + Y_k +Z_k -1} \leq
(1+\beta/\rho)\frac{X_0 L_M^{\beta/\rho}}{M-2k_0}.
\]
Therefore, the expected number of indices $k < \tau \wedge k_0$ such that $X_{k+1} <X_k$ is bounded from above by 
\[
k_0(1+\beta/\rho)\frac{X_0 L_M^{\beta/\rho}}{M-2k_0},
\]
(see Remark \ref{rem:nesting}) and it is a simple observation that we can choose $L_M$ to grow slow enough so that the above expression converges to $0$.

Define $\sigma_3$ to be the first time $k$ such that $Y_k \leq Y_0/2$ and observe that since until the time $\sigma_{2.\varepsilon}$ there are at most $\varepsilon Y_0 /d$ indices $k < \tau \wedge k_0$ such that $Y_{k+1} < Y_k$, so for $\varepsilon$ small enough this will force $\sigma_3 > \sigma_{2,\varepsilon}$.
Therefore, for $\varepsilon$ small the event $\{\sigma_{2,\varepsilon} \leq \tau \wedge k_0\}$ is the same as $\{\sigma_{2,\varepsilon} \leq \tau \wedge k_0 \wedge \sigma_3\}$.

Given $k<\tau \wedge k_0\wedge \sigma_3$ the conditional probability that $X_{k+1} \neq X_k$ or $Y_{k+1} < Y_k$ is 
\[
\frac{\beta X_k}{\beta X_k + \rho Y_k} + \frac{\rho Y_k}{\beta X_k +\rho Y_k} \frac{X_k+Y_k}{X_k+Y_k+Z_k-1} \leq \frac{(\beta+ \rho )X_k}{\beta X_k + \rho Y_k} + \frac{Y_k}{X_k+Y_k+Z_k-1}.
\]
The right hand side can then be bounded from above by
\[
(1+\beta/\rho)\frac{X_0L_M^{\beta/\rho}}{Y_0} + \frac{Y_0L_M}{M-2k_0-1},
\]
and so the expected number of indices $k$ with the above property is no more than
\[
(1+\beta/\rho)k_0\frac{X_0L_M^{\beta/\rho}}{Y_0} + k_0\frac{Y_0L_M}{M-2k_0-1}.
\]
By Markov inequality the probability that there are more than $\varepsilon Y_0/d$ of such indices is at most
\[
(1+\beta/\rho)\frac{dX_0(X_0+Y_0)L_M^{1+2\beta/\rho}}{\varepsilon Y_0^2} + \frac{d(X_0+Y_0)L_M^{2+\beta/\rho}}{\varepsilon(M-2L_M(X_0+Y_0)-1)},
\]
which converges to zero for any $\varepsilon > 0$ if $L_M$  is chosen to grow slow enough. This proves that claim that one can choose $L_M$ so that  \eqref{eq:choice for bound on stopping time} holds.

Now fix $\varepsilon > 0$.
Now consider two processes $(X_{n,1},Y_{n,1})$ and $(X_{n,2},Y_{n,2})$ such that $(X_{0,1},Y_{0,1}) = (X_0,Y_0)$ and $(X_{0,2},Y_{0,2}) = (X_0,(1-\varepsilon)Y_0)$, and $(X_{n+1,i},Y_{n+1,i}) = (X_{n,i}+d-2,Y_{n,i})$ with the probability $\beta X_{n,i}/(\beta
X_{n,i} + \rho Y_{n,i})$ and $(X_{n+1,i},Y_{n+1,i}) = (X_{n,i},Y_{n,i}+d-2)$ otherwise, for $i=1,2$. We will show that there is a coupling of $(X_{n,1},Y_{n,1})$ and $(X_n,Y_n)$ such that $X_{k,1} \leq X_k$ and $Y_k \leq Y_{k,1}$ for all $k < \sigma_1 \wedge \sigma_{2,\varepsilon}$ and a coupling of of $(X_{n,2},Y_{n,2})$ and $(X_n,Y_n)$ such that $X_{k} \leq X_{k,2}$ and $Y_{k,2} \leq Y_{k}$ also for all $k < \sigma_1 \wedge \sigma_{2,\varepsilon}$. 

To prove the existence of the couplings we proceed inductively. Clearly at $k=0$ the desired inequalities are satisfied. 
We will show that if the inequalities $X_{k,1} \leq X_k$ and  $Y_k \leq Y_{k,1}$ are satisfied and if $k+1 < \sigma_1 \wedge \sigma_{2,\varepsilon}$ then we can couple the next step of the processes so that $X_{k+1,1} \leq X_{k+1}$ and  $Y_{k+1} \leq Y_{k+1,1}$. Assuming the inequalities hold at time $k$ we can surely couple the processes so that $X_{k+1,1} = X_{k,1}+d-2$ implies that a blue half-edge is chosen in the first step of the transition from $(X_k,Y_k)$ to $(X_{k+1},Y_{k+1})$. Therefore, if $X_{k+1,1} = X_{k,1}+d-2$ and  $k+1 < \sigma_1$ we have both $X_{k+1} = X_k + d-2$ and $Y_{k+1} = Y_k$ and the inequalities $X_{k+1,1} \leq X_{k+1}$ and $Y_{k+1} \leq Y_{k+1,1}$ are obvious. Furthermore, if $X_{k+1,1} = X_{k,1}$ the inequalities $X_{k+1,1} \leq X_{k+1}$ and  $Y_{k+1} \leq Y_{k+1,1}$ hold since for $k+1 < \sigma_1$ we have $X_{k+1} \geq X_k$.

For the other coupling observe that the value of $Y_k-Y_{k,2}$ can decrease by at most $d$, and that the decrease happens only when either $X_{k+1} \neq X_k$ or $Y_{k+1} < Y_k$ (because otherwise $Y_{k+1} = Y_k +d-2$). Since, the number of such $k$'s up to time $\sigma_1\wedge \sigma_{2,\varepsilon}$ is no more than $\varepsilon Y_0/d$ and $Y_0-Y_{0,2} = \varepsilon Y_0$, any coupling will satisfy $Y_{k,2} \leq Y_k$ for all $k \leq \sigma_1 \wedge \sigma_{2,\varepsilon}$. Thus it suffices to construct the coupling so that $X_{k+1} = X_k + d-2$ implies $X_{k+1,2} = X_{k,2} + d-2$. This is shown using the same inductive argument as before. 

To end the proof define $\tau_1$ to be $\tau$ if $Y_\tau \geq Y_0 L_M^{1/2}$ and otherwise the smallest index $k\geq \tau$ such that $X_{k,1} \geq X_0 L_M^{\beta/(2\rho)}$. In the former case we have $Y_{\tau_1,1} \geq Y_{0,1} L_M^{1/2}$ and in the latter $X_{\tau_1,1} \geq X_{0,1} L_M^{\beta/(2\rho)}$. As $\tau_1$ is clearly a stopping time with respect to the filtration $\mathcal{F}_n$, by Corollary \ref{cor:svante_theorem_applied} we have that $\frac{X_{\tau_1,1}}{(Y_{\tau_1,1}/Y_{0,1})^{\beta/\rho}}$ converges to $(d-2)\Gamma(X_0/(d-2),1)$.
For a fixed $t$ and $M$ large enough the event,
\[
\left\{\frac{X_\tau}{(Y_\tau/Y_0)^{\beta/\rho}}\leq t\right\}
\]
implies that $Y_\tau \geq Y_0\sqrt{L_M}$ (otherwise we have $X_\tau \geq X_0L_M^{\beta/\rho}$ and $\frac{X_\tau}{(Y_\tau/Y_0)^{\beta/\rho}} \geq X_0 L_M^{\beta /(2\rho)}$), and thus $\tau_1 = \tau$. Combined with the inequalities from the coupling the above and the estimate \eqref{eq:choice for bound on stopping time}, we have 
\[
\mathbb{P}\Big(\frac{X_\tau}{(Y_\tau/Y_0)^{\beta/\rho}}\leq t\Big) \leq \mathbb{P}\Big(\frac{X_{\tau_1,1}}{(Y_{\tau_1,1}/Y_{0,1})^{\beta/\rho}}\leq t\Big) + \xi_M,
\]
where $\lim_M \xi_M=0$.
This gives the upper bound on the probabilities
\[
\limsup_{M \to \infty}\mathbb{P}\Big(\frac{X_\tau}{(Y_\tau/Y_0)^{\beta/\rho}}\leq t\Big) \leq \mathbb{P} \Big((d-2)\Gamma(X_0/(d-2),1) \leq t\Big).
\]
The lower bounds are obtained analogously, using the coupling with the process $(X_{n,2},Y_{n,2})$, defining the stopping time $\tau_2$ to be $\tau$ if $X_\tau \geq X_0 L_M^{\beta /(2\rho)}$ and otherwise the smallest $k \geq \tau$ such that $Y_{k,2} \geq Y_{k,0}\sqrt{L_M}$ and  observing that $\varepsilon > 0$ is arbitrary.

\end{proof}

\section{Proof of the main estimates}
This section contains the proofs of Theorems \ref{thm:(1+o(1))_concentration_diff_rates_0} and \ref{thm: 1+o(1) concentration for boundary}. As the proofs are quite technical, in order simplify formulas we will assume that $\rho = 1$. This is always possible, as the both the law of $(X_n,Y_n)$ and the formulas in these theorems are invariant under scaling of $(\beta,\rho)$, so we simply need to scale the rates from $(\beta,\rho)$ to $(\beta/\rho,1)$. The assumption $\rho =1$ will be used throughout the whole section.
Assuming $\rho=1$ formula \eqref{eq:1+o(1)_concentration_beta_no_1} reduces to
\begin{equation}\label{eq:1+o(1)_concentration_beta_no_1_one_parameter}
R_{\text{fin}} - R_0 = (1 \pm \varepsilon) \frac{M}{d}\int_0^1
\phi_{\beta}^{-1}
\left(MX_0^{\frac{1}{\beta-1}}Y_0^{\frac{\beta}{1-\beta}}
\Big(t^{1/d} -
t^{(d-1)/d}\Big)\right)dt, 
\end{equation}
and \eqref{eq:(1+o(1))_concentration_boundary_diff_rates} to
\begin{equation}\label{eq:(1+o(1))_concentration_boundary_diff_rates_one_parameter}
D_{\text{fin}} = (1\pm\varepsilon) \frac{M}{2}\int_0^1\Big(1-t^{d/2-1}\Big) \kappa_{\beta} \circ \phi^{-1}_{\beta} 
\left(MX_0^{\frac{1}{\beta-1}}Y_0^{\frac{\beta}{1-\beta}}\Big(t^{1/2} -
t^{(d-1)/2}\Big)\right)dt, \ \text{ for } \beta \neq \rho,
\end{equation}
where functions reduce to $\kappa_{\beta}(t) = \frac{(\beta + 1)t(1-t)}{\beta t + (1-t)}$ and 
\[
\phi_{\beta} (s)= \Big(\frac{1-s}{\beta s}\Big)^{\frac{1}{\beta-1}} +  \Big(\frac{1-s}{\beta s}\Big)^{\frac{\beta}{\beta-1}}.
\]

The proofs  are based on a martingale method. In short we will identify two observables in our model that will ``behave like martingales''. More precisely, we will be able to effectively bound the conditional first and second moments of the step sizes in each process. 

First we present a general lemma bounding the
conditional expectation and variance of the differences in a general
random process.
\begin{lemma}\label{lemma: estimates for martingales}
Let $(K_n)_{n \geq 0}$ be a positive process such that $K_0$ is a
constant, and $p_n$ and $r_n$ positive real numbers defined for $n
\geq 0$, such that
\[
|\mathbb{E}(K_{n+1}-K_n|\mathcal{F}_n)| \leq p_n K_n, \ \text{and}
\ \ \mathbb{E}((K_{n+1}-K_n)^2|\mathcal{F}_n) \leq r_n K_n.
\]
Consider the process $I_0=K_0$, $I_n = K_n -
\sum_{k=0}^{n-1}\mathbb{E}(K_{k+1}-K_{k}|\mathcal{F}_k)$. Then process
$I_n$ is a martingale and for every positive integer $n$ we have
\[
|K_n - I_n| \leq \sum_{k=0}^{n-1}p_kq_{k+1,n-1}I_k, \ \text{and}
\ \ \mathbb{E}((I_{n}-I_0)^2) \leq K_0\sum_{k=0}^{n-1}r_kq_{0,k-1},
\]
where $q_{\ell,k}=\prod_{i=\ell}^k (1+p_i)$ for $\ell \leq k$ and
$q_{k,k-1}=1$, for all $k \geq 0$.
\end{lemma}

\begin{proof}
It is trivial to check that the process $I_n$ is a martingale.
Furthermore it can be shown by induction that for every $k \leq n$
\begin{equation}\label{eq:property of coefficients}
q_{k,n}-1 = \sum_{\ell=k}^n p_\ell q_{\ell+1,n} = \sum_{\ell=k}^n
p_\ell q_{k,\ell-1}.
\end{equation}
Using the first inequality in the statement we have that
 \begin{equation}\label{eq: difference between A and B 0}
 |K_n-I_n| \leq \sum_{k=0}^{n-1}p_k K_k.
 \end{equation}
In particular we have
\begin{equation}\label{eq: difference between A and B}
K_n - I_n \leq \sum_{k=0}^{n-1}p_k K_k = \sum_{k=0}^{n-1}p_k(K_k-I_k)
+ \sum_{k=0}^{n-1}p_kI_k.
\end{equation}
Using \eqref{eq: difference between A and B} inductively we can show
that $K_n - I_n \leq \sum_{k=0}^{n-1}a_{n,k}I_k$ whenever the sequence
$(a_{n,k})_{0 \leq k < n}$ satisfies $a_{n,n-1}=p_{n-1}$ and $a_{n,k}
= \sum_{\ell=k+1}^{n-1}p_\ell a_{\ell,k}+p_k$. Using
\eqref{eq:property of coefficients} it is easy to check that
$a_{n,k}=p_kq_{k+1,n-1}$ satisfies these conditions. Thus we have
 \begin{equation}\label{eq: difference between A and B 2}
 K_n \leq I_n + \sum_{k=0}^{n-1}p_kq_{k+1,n-1}I_k.
 \end{equation}
Plugging this back into \eqref{eq: difference between A and B 0} and
using \eqref{eq:property of coefficients} we get
\[
|K_{n}- I_{n}| \leq p_{n-1}I_{n-1} + \sum_{k=0}^{n-2}
p_k\Big(1+\sum_{\ell=k+1}^{n-1}p_\ell q_{k+1,\ell-1}\Big)I_{k} =
\sum_{k=0}^{n-1}p_kq_{k+1,n-1}I_k,
\]
which proves the first claim.

Note that \eqref{eq: difference between A and B 2} and
\eqref{eq:property of coefficients} imply that
\begin{equation}\label{eq: expectation of A}
\mathbb{E}(K_n) \leq \Big(1 + \sum_{k=0}^{n-1}p_kq_{k+1,n-1}\Big)I_0 =
q_{0,n-1}K_0.
\end{equation}
Thus the condition in the statement implies that
\[
\mathbb{E}((K_{n+1} - K_{n})^2) \leq r_n\mathbb{E}(K_{n}) \leq r_n
q_{0,n-1}K_0.
\]
It is easy to check that $\mathbb{E}((I_{n+1} - I_{n})^2|
\mathcal{F}_n) \leq \mathbb{E}((K_{n+1} - K_{n})^2|\mathcal{F}_n)$
which then yields
\[
\mathbb{E}((I_{n}-I_0)^2) = \sum_{k=0}^{n-1} \mathbb{E}((I_{k+1} -
I_{k})^2) \leq K_0\sum_{k=0}^{n-1}r_kq_{0,k-1}.
\]
This concludes the proof.
\end{proof}

As we mentioned before, the key to the proof of Theorems \ref{thm:(1+o(1))_concentration_diff_rates_0} and \ref{thm: 1+o(1) concentration for boundary} is to identify two processes for which we can estimate the conditional first and second moments or their step sizes. Then martingale methods (including the above lemma) will enable us to bound the maximal displacements of the processes throughout the whole relevant time regime. These processes are
\begin{equation}
\label{definition_of_K_n}
K_n = \frac{X_n}{Y_n^{\beta} (1 - 2n/M)^{(1-\beta)/2}}~.
\end{equation}
and
\[
L_{n}= \frac{X_n+Y_n}{(M-2n)-Z_0(1-2n/M)^{d/2}},
\]
Note that once we managed to bound the values of processes we can ``solve for $X_n$ and $Y_n$'' to estimate their values. The processes $K_n$ and $L_n$ tell us how $X_n+Y_n$ and $X_nY_n^{-\beta}$ behave. It is actually quite natural to study these processes. The process $X_n+Y_n$ corresponds to the pure configuration model, that is to erasing the colors of half-edges, and is equivalent to an urn model (see Lemma \ref{lemma:urn_in_configuration}). The motivation for the considering the process $X_nY_n^{-\beta}$  can be given as follows.
After removing the interaction and self-interaction of the colored half-edges, the whole model reduces to an (unbalanced) urn model with a diagonal replacement matrix. Consider the Poissonized version of that model and define continuous time processes $X_t$ and $Y_t$ as the number of balls of each color. Then the process  $X_tY_t^{-\beta}$ is a continuous time martingale \cite{Janson06}. The factor $(1 - 2n/M)^{-(1-\beta)/2}$ thus accounts for the interaction and self-interaction of colors. From this discussion it follows for $\beta<1$, that the value of the process $X_nY_n^{-\beta}$ is smaller in our model than in the model with interactions and self-interactions removed. Thus, compared to the dynamics on disjoint $d$-regular trees, 
 interactions and self-interactions on the random $d$-regular graph give the process with a faster rate $X_n$  an additional ``boost'' relative to $Y_n$. 

Estimates for the process $L_n$ are given in Theorem \ref{thm: urn applied to model}.
Process $K_n$ will be estimated in this section. First we estimate its steps to be able to apply Lemma \ref{lemma: estimates for martingales}.

\begin{lemma}\label{lemma:estimates for diff rates process}
For the process $K_n$ as defined in \eqref{definition_of_K_n}, there
exists a constant $C=C(\beta,d)>0$, such that for
all integers $n$, on the event that $Y_n \geq 2d$ we have both
\begin{equation}\label{eq: estimates for diff rates process_1}
|\mathbb{E}(K_{n+1}-K_n | \mathcal{F}_n)| \leq
\frac{CK_n}{Y_n(X_n+Y_n)},
\end{equation}
and
\begin{equation}\label{eq: estimates for diff rates process_2}
\mathbb{E}((K_{n+1}-K_n)^2|\mathcal{F}_n) \leq
\frac{CK_n}{Y_n^{1+\beta}(1-2n/M)^{(1-\beta)/2}}.
\end{equation}
\end{lemma}

\begin{proof}
Throughout the proof we assume that $M-2n \geq Y_n \geq 2d$.  To prove
\eqref{eq: estimates for diff rates process_1} we calculate
\begin{multline}\label{eq:diff_rates_aux_1}
\Big(1-\frac{2n+2}{M}\Big)^{(1-\beta)/2}(\beta
X_n+Y_n)(M-2n-1)\mathbb{E}(K_{n+1} | \mathcal{F}_n) =
\\ \frac{X_n+d-2}{Y_n^\beta}\beta X_n(M-2n-X_n - Y_n) +
\frac{X_n}{(Y_n+d-2)^\beta}Y_n(M-2n-X_n-Y_n) \\ +
\frac{X_n-2}{Y_n^\beta} \beta X_n (X_n-1) +
\frac{X_n-1}{(Y_n-1)^\beta} (1+\beta) X_n Y_n +
\frac{X_n}{(Y_n-2)^\beta} Y_n (Y_n-1).
\end{multline}
It can be easily verified that
\begin{multline}\label{eq:diff_rates_aux_2}
\Big(1-\frac{2n+2}{M}\Big)^{(1-\beta)/2}(\beta X_n+Y_n)(M-2n-1)K_n\\ =
\frac{X_n}{Y_n^\beta}(\beta X_n+Y_n)\Big(1-
\frac{2}{M-2n}\Big)^{(1-\beta)/2}(M-2n-1)\\ =
\frac{X_n}{Y_n^\beta}(\beta X_n+Y_n)(M-2n- 2+\beta + O((M-2n)^{-1})),
\end{multline}
where the absolute value of the term $O((M-2n)^{-1})$ is bounded by a
 constant multiple of $(M-2n)^{-1}$. To prove \eqref{eq: estimates
  for diff rates process_1} it is enough to show that the absolute
value of the difference of the terms in \eqref{eq:diff_rates_aux_1}
and \eqref{eq:diff_rates_aux_2} is bounded by
\begin{equation}\label{eq:diff_rates_aux_3}
 \frac{CX_n(M-2n)}{Y_n^{\beta+1}}
\end{equation}
for some constant $C$. First note that, since $X_n + Y_n\leq M-2n$,
the expression 
\[
X_nY_n^{-\beta}(\beta X_n+Y_n)(M-2n)^{-1}
\]
is bounded by \eqref{eq:diff_rates_aux_3}, for some $C>0$. Thus we can
disregard the term $O((M-2n)^{-1})$ in \eqref{eq:diff_rates_aux_2}.

By Taylor expansion we know that for any compact interval containing
$1$ there is a constant $C_1$ such that for all $t$ in this interval
\[
\Big|1- \beta t - \frac{1}{(1+t)^\beta}\Big| \leq \frac{C_1t^2}{
  (1+t)^\beta}
\]
(actually by a slightly more careful argument one can argue that $C_1$
does not depend on the interval).  Now fix any $k\geq  -2$ and choose
$t=kY_n^{-1}$ and a constant $C_1$ to obtain
\[
\left|\frac{1}{Y_n^\beta}\Big(1- \frac{k\beta}{Y_n}\Big) -
\frac{1}{(Y_n+k)^\beta} \right| \leq \frac{C_1k^2}{(Y_n+k)^\beta
  Y_n^2}.
\]
For $k=d-2$ this in particular implies that
\[
\left|X_nY_n^{1-\beta}(M-2n-X_n-Y_n)\Big(1-
\frac{\beta(d-2)}{Y_n}\Big) -
\frac{X_nY_n(M-2n-X_n-Y_n)}{(Y_n+d-2)^\beta}\right|\] is at most
\[
\frac{C_2X_n(M-2n)}{Y_n^{\beta+1}},
\]
for a constant $C_2 = (d-2)^2C_1$. Therefore we can replace the term
$\frac{X_n}{(Y_n+d-2)^\beta}Y_n(M-2n-X_n-Y_n)$ on the right hand side
of \eqref{eq:diff_rates_aux_1} by $X_nY_n^{-\beta}(M-2n-X_n-Y_n)(Y_n-
\beta(d-2))$. Arguing similarly we see that we can replace the terms
$\frac{X_n-1}{(Y_n-1)^\beta} (1+\beta) X_n Y_n$ and
$\frac{X_n}{(Y_n-2)^\beta} Y_n (Y_n-1)$ on the right hand side of
\eqref{eq:diff_rates_aux_1} by $ \frac{X_n-1}{Y_n^\beta} (1+\beta) X_n
(Y_n + \beta)$ and $\frac{X_n}{Y_n^{\beta}} (Y_n-1)(Y_n +
2\beta)$ respectively. Therefore it is enough to prove
\begin{multline*}
(M-2n-X_n - Y_n) \Big(\frac{X_n+d-2}{Y_n^\beta}\beta X_n +
  \frac{X_n}{Y_n^\beta}(Y_n - \beta (d-2)) -
  \frac{X_n}{Y_n^\beta}(\beta X_n+Y_n)\Big)\\ +
  \frac{X_n-2}{Y_n^\beta} \beta X_n (X_n-1) + \frac{X_n-1}{Y_n^\beta}
  (1+\beta) X_n (Y_n + \beta) + \frac{X_n}{Y_n^{\beta}} (Y_n-1)(Y_n +
  2\beta) \\ - \frac{X_n}{Y_n^\beta}(\beta X_n+Y_n)(X_n + Y_n -
  2+\beta) \leq \frac{CX_n(M-2n)}{Y_n^{\beta+1}},
\end{multline*}
for a large enough constant $C$.
Expanding the expressions in the left hand side above we see that it
is equal to $\beta(\beta+1)X_nY_n^{-\beta}$. This proves the claim.

Now we prove \eqref{eq: estimates for diff rates process_2}. First
note that it is enough to prove that
\begin{equation}\label{eq: estimates for diff rates process_2_1}
\mathbb{E}((K_{n+1}-K_n)^2|\mathcal{F}_n) \leq \frac{CK_n^2}{X_nY_n}.
\end{equation}
Analyzing all the cases we see that the value of $|K_{n+1} - K_n|$ is
\[
\left|\frac{X_n+d-2}{Y_n^\beta(1-\frac{2n+2}{M})^{(1-\beta)/2}}
-\frac{X_n}{Y_n^\beta(1-2n/M)^{(1-\beta)/2}}\right| \leq C_1K_n
\Big(\frac{1}{X_n} + \frac{1}{M-2n}\Big),
\]
\[
\left|\frac{X_n}{(Y_n+d-2)^\beta(1-\frac{2n+2}{M})^{(1-\beta)/2}}
-\frac{X_n}{Y_n^\beta(1-2n/M)^{(1-\beta)/2}}\right| \leq C_2K_n
\Big(\frac{1}{Y_n} + \frac{1}{M-2n}\Big),
\]
\[
\left|\frac{X_n-2}{Y_n^\beta(1-\frac{2n+2}{M})^{(1-\beta)/2}}
-\frac{X_n}{Y_n^\beta(1-2n/M)^{(1-\beta)/2}}\right| \leq C_3K_n
\Big(\frac{1}{X_n} + \frac{1}{M-2n}\Big),
\]
\[
\left|\frac{X_n-1}{(Y_n-1)^\beta(1-\frac{2n+2}{M})^{(1-\beta)/2}}
-\frac{X_n}{Y_n^\beta(1-2n/M)^{(1-\beta)/2}}\right| \leq C_4K_n
\Big(\frac{1}{X_n} + \frac{1}{Y_n} + \frac{1}{M-2n}\Big),
\]
or
\[
\left|\frac{X_n}{(Y_n-2)^\beta(1-\frac{2n+2}{M})^{(1-\beta)/2}}
-\frac{X_n}{Y_n^\beta(1-2n/M)^{(1-\beta)/2}}\right| \leq C_5K_n
\Big(\frac{1}{Y_n} + \frac{1}{M-2n}\Big),
\]
with probabilities
\begin{multline*}
\frac{\beta X_n (M-2n-X_n-Y_n)}{(\beta X_n + Y_n)(M-2n-1)},
\ \frac{Y_n (M-2n-X_n-Y_n)}{(\beta X_n + Y_n)(M-2n-1)}, \ \frac{\beta
  X_n(X_n-1)}{(\beta X_n + Y_n)(M-2n-1)},
\\ \frac{(1+\beta)X_nY_n}{(\beta X_n + Y_n)(M-2n-1)},
\ \frac{Y_n(Y_n-1)}{(\beta X_n + Y_n)(M-2n-1)},
\end{multline*}
respectively.  Here $C_1$, $C_2$, $C_3$, $C_4$ and $C_5$ are constants
depending only on $\beta$ and $d$.  Therefore for a large constant
$C_0$ we have
\begin{align*}
\mathbb{E}((K_{n+1}-K_n)^2|\mathcal{F}_n)& \leq C_0 K_n^2
\Big(\frac{1}{(M-2n)^2} + \frac{1}{X_n^2}\frac{X_n}{\beta X_n + Y_n} +
\frac{1}{Y_n^2}\frac{Y_n}{\beta X_n + Y_n}\Big) \\ & \leq C_0 K_n^2
\Big(\frac{1}{(M-2n)^2} + \frac{X_n+Y_n}{X_nY_n(\beta X_n +
  Y_n)}\Big),
\end{align*}
which, together with the fact $X_nY_n \leq (M-2n)^2$, yields
\eqref{eq: estimates for diff rates process_2_1} .

\end{proof}



Unfortunately Lemma \ref{lemma: estimates for martingales} will not allow us to directly estimate the process $K_n$. A reason for this is that using bounds in Lemma \ref{lemma:estimates for diff rates process} depend on the values of $X_n$ and $Y_n$ themselves. Lemma \ref{lemma:step_for_different_rates_2} is a first attempt to fix this issue. 

We introduce a function which will appear quite often in the analysis in this section. Define
\begin{equation}\label{eq: function f}
f(t) = t^{1/2} -  \frac{Z_0}{M}t^{(d-1)/2}.
\end{equation}
Clearly $f$ is a positive concave function on $(0,1)$.  The following estimate will be used in the proof of Lemma \ref{lemma:step_for_different_rates_2}.

\begin{lemma}\label{lemma:estimate_on_the_integral}
Let $n$ be such that $M-2n \geq 1$, and $\gamma > 1$. Then there is a constant $C_0 = C_0(d,\gamma)$  such that 
\[
\sum_{k=0}^{n-1} \frac{1}{(M-2k)f(1-2k/M)^\gamma} \leq C_0 \left(\Big(\frac{M}{M-2n}\Big)^{\gamma/2} + \Big(\frac{M}{X_0 + Y_0}\Big)^{\gamma -1}\right).
\]
\end{lemma}

\begin{proof}
It is easy to check that the summand corresponding to $k=0$ is equal to $M^{\gamma-1}(X_0+Y_0)^{-\gamma}$, and therefore we get neglect this term in the sum.
Define
\[
g(t) = \frac{1}{tf(t)^\gamma} =  \frac{1}{t^{1+\gamma/2}\Big(1- \frac{Z_0}{M}t^{d/2-1}\Big)^{\gamma}}.
\]
One can calculate
\[
g'(x) = x^{-2-\gamma/2} \Big(1- \frac{Z_0}{M}x^{d/2-1}\Big)^{-\gamma-1}\left(\frac{Z_0}{M}x^{d/2-1} \Big(1 + \frac{\gamma (d-1)}{2} \Big) - 1 -\frac{\gamma}{2}\right),
\]
so $g(x)$ is either decreasing on $(0,1)$ or decreasing on an interval $(0,x_0)$ and increasing on $(x_0,1)$, for some $0 < x_0 < 1$. Therefore, we can bound the sum by the integral
\begin{equation}\label{eq: estimate on the integral}
\sum_{k=1}^{n-1} \frac{1}{(M-2k)f(1-2k/M)^\gamma}  = \frac{1}{2} \sum_{k=1}^{n-1}\frac{2}{M} g\Big(1-\frac{2k}{M}\Big) \leq \frac{1}{2}\int_{1-2n/M}^1 g(t)~dt.
\end{equation}
We split the integral into two parts. First we consider
\begin{align*}
\int_{(1-2n/M)\wedge 1/2}^{1/2} g(t)~dt &  = \int_{(1-2n/M)\wedge 1/2}^{1/2} \frac{1}{t^{1+\gamma/2}\Big(1- \frac{Z_0}{M}t^{d/2-1}\Big)^{\gamma}}~dt \\
& \leq \frac{1}{\Big(1-2^{1-d/2}\Big)} \int_{(1-2n/M)\wedge 1/2}^{1/2} t^{-1-\gamma/2}~dt \\
& \leq \frac{2}{\gamma\Big(1-2^{1-d/2}\Big)} \Big(\frac{M}{M-2n}\Big)^{\gamma/2},
\end{align*}
To analyze the other part we use a simple inequality $t^{d/2-1} \leq (t+1)/2$, which holds for all $0 \leq t \leq 1$ and $d \geq 3$. We get
\begin{align*}
\int_{1/2}^1 g(t)~dt  &   = \int_{1/2}^{1} \frac{1}{t^{1+\gamma/2}\Big(1- \frac{Z_0}{M}t^{d/2-1}\Big)^{\gamma}}~dt  \\
& \leq 2^{1+\gamma/2}\int_{1/2}^1 \frac{1}{\Big(1- \frac{Z_0}{2M}(1+t)\Big)^{\gamma}}~dt \\
& \leq \frac{2^{2+\gamma/2}M}{Z_0}\int_{1-Z_0/M}^1s^{-\gamma}~ds \\
& \leq \frac{2^{2+\gamma/2}}{\gamma-1}\frac{M}{Z_0}\left(\Big(\frac{M}{M-Z_0}\Big)^{\gamma-1} -1\right).
\end{align*}
Now we use the inequality $((1-t)^{-\alpha} -1)t^{-1} \leq (1\vee \alpha) (1-t)^{-\alpha}$, which holds for all $0 < t < 1$ and all $\alpha > 0$ (this inequality follows easily from the fact that $(1-t)^\alpha \geq 1-t$ for $\alpha \leq 1$ and $(1-t)^\alpha \geq 1-\alpha t$ for $\alpha > 1$). We apply this inequality for $t=Z_0/M$ and $\alpha = \gamma-1$. The above expression is then bounded by
\begin{align*}
\int_{1/2}^1 g(t)~dt & \leq \frac{2^{2+\gamma/2}(1 \vee (\gamma-1))}{\gamma-1}\Big(\frac{M}{M-Z_0}\Big)^{\gamma-1} \\
& = \frac{2^{2+\gamma/2}(1 \vee (\gamma-1))}{\gamma-1}\Big(\frac{M}{X_0 + Y_0}\Big)^{\gamma-1}
\end{align*}
Adding both parts to \eqref{eq: estimate on the integral} yields the claim.
\end{proof}

\begin{lemma}\label{lemma:step_for_different_rates_2}
Let $0 < \varepsilon \leq 1/2$. For a positive real number $c$ assume
that the condition
\begin{equation}\label{eq:step_for_different_rates_2_condition_definition}
\frac{K_k}{M^{(1-\beta)/2}} \leq c(M-2k)^{(1-\beta)/2}\Big(1-
\frac{Z_0}{M}(1-2k/M)^{d/2-1}\Big)^{1-\beta}
\end{equation}
is satisfied for $k=0$ and define the stopping time $\tau$ as the
smallest positive integer $k$ for which
\eqref{eq:step_for_different_rates_2_condition_definition} is not
satisfied.
Then there exists a sequence $(\delta_M)$ converging to $0$ and depending only on $\beta$ and $d$, 
and a constant $C = C(\beta, d, c)>0$, 
such that for any positive integer $n$
\begin{equation}\label{eq:step_for_different_rates_2}
\mathbb{P}(|K_{k \wedge \tau} -K_0| \geq \varepsilon K_0, \text{ for some
} 0 \leq k \leq n) \leq
\frac{C}{\varepsilon^2}\Big(\frac{M^{(1-\beta)/2}}{(M-2n)^{(1+\beta)/2}K_0}
+\frac{1}{X_0}\Big) + \delta_M,
\end{equation}
whenever
\begin{equation}\label{eq:step_for_different_rates_2_condition}
\varepsilon \geq C\Big(\frac{1}{M-2n} + \frac{1}{X_0+Y_0}\Big).
\end{equation}
\end{lemma}

\begin{proof}
We begin by showing that for any $C_0>0$ we can choose $C$ so that
\eqref{eq:step_for_different_rates_2_condition} implies
\begin{equation}\label{eq:diff_rates_choise_of_constant}
\Big(1-\frac{2k}{M}\Big)\Big(1-\frac{Z_0}{M}\Big(1-\frac{2k}{M}\Big)^{d/2-1}\Big)
\geq \frac{C_0}{M\varepsilon},
\end{equation}
for all $0 \leq k \leq n$.  Since the function $\phi(t)=
t-t^{d/2}Z_0/M$ is concave on $[0,1]$ the minimum of the left hand
side in \eqref{eq:diff_rates_choise_of_constant} is either $\phi(1) =
(M-Z_0)/M = (X_0+Y_0)/M$ or
\[
\phi(1-2n/M)\geq \left(1-\frac{2n}{M} -\Big(1-\frac{2n}{M}\Big)^{d/2}\vee \Big(1-\frac{2n}{M}\Big)\frac{X_0+Y_0}{M}\right). 
\]
For $C=C_0$, the value $\phi(1)$ is clearly  bounded from below by the right hand
side of \eqref{eq:diff_rates_choise_of_constant}. To estimate the value of $\phi(1-2n/M)$ from below, use the first term on the right hand side when $1-2n/M<1/2$, and otherwise use the second term.

Now define $\sigma$ as the first time $k$ that
\[
X_k + Y_k \leq \frac{M-2k}{2}\Big(1 -
\frac{Z_0}{M}(1-2k/M)^{d/2-1}\Big) \text{ or } Y_k \leq 2d,
\]
and define the process $K'_k = K_{k \wedge \tau \wedge \sigma}$. Since
$\sigma$ and $\tau$ are stopping times with respect to the filtration
$\mathcal{F}_k$, the process $K_k'$ is adapted to this filtration.

Next we show that there is a positive constant $c_1$ such that for all
$k < \sigma \wedge \tau$ we have
\begin{equation}\label{eq:step_for_different_rates_2_condition_lower_boud_for_Y}
Y_k \geq c_1(M-2k)\Big(1 - \frac{Z_0}{M}(1-2k/M)^{d/2-1}\Big).
\end{equation}
Assume, for the sake of contradiction, that for some $k < \sigma
\wedge \tau$ we have
\[
Y_k < c_1(M-2k)\Big(1 - \frac{Z_0}{M}(1-2k/M)^{d/2-1}\Big).
\]
Then since $k < \tau$ we have
\begin{align*}
X_k & \leq Y_k^\beta (M-2k)^{(1-\beta)/2} c(M-2k)^{(1-\beta)/2}\Big(1-
\frac{Z_0}{M}(1-2k/M)^{d/2-1}\Big)^{1-\beta} \\ & < c c_1^\beta
(M-2k)\Big(1 - \frac{Z_0}{M}(1-2k/M)^{d/2-1}\Big).
\end{align*}
Since $k<\sigma$ we have $X_k + Y_k > \frac{M-2k}{2}\Big(1 -
\frac{Z_0}{M}(1-2k/M)^{d/2-1}\Big)$ which implies $c c_1^\beta + c_1
\geq 1/2$. When $c_1$ is small enough we obtain a contradiction and
prove
\eqref{eq:step_for_different_rates_2_condition_lower_boud_for_Y}.
Lemma \ref{lemma:estimates for diff rates process} now implies that
for all $0 \leq k \leq n$
\begin{equation}\label{eq:first_moment_estimates_for_different_rates}
|\mathbb{E}(K_{k+1}'-K_k'|\mathcal{F}_k)| \leq
\frac{C_1K_k'}{(M-2k)^2\Big(1- \frac{Z_0}{M}(1-2k/M)^{d/2-1}\Big)^2} = \frac{C_1K_k'}{M(M-2k)f(1-2k/M)^2},
\end{equation}
and
\begin{align}\label{eq:first_moment_estimates_for_different_rates_second_moment}
\mathbb{E}((K_{k+1}'-K_k')^2|\mathcal{F}_k) & \leq
\frac{C_1M^{(1-\beta)/2}K_k'}{(M-2k)^{(3+\beta)/2}\Big(1-
  \frac{Z_0}{M}(1-2k/M)^{d/2-1}\Big)^{1+\beta}}\nonumber \\ & = \frac{C_1K_k'M^{-\beta}}{(M-2k)f(1-2k/M)^{1+\beta}},
\end{align}
for some constant $C_1=C_1(\beta,d,c)$.
Define
\[
p_k = \frac{C_1}{M(M-2k)f(1-2k/M)}.
\]
By Lemma \ref{lemma:estimate_on_the_integral} for $\gamma=2$
\begin{align}\label{eq:diff_rates_help_1}
\sum_{k=0}^{n-1}p_k  = \sum_{k=0}^{n-1} \frac{C_1}{M(M-2k)f(1-2k/M)^2} \leq C_0C_1 \left(\frac{1}{M-2n}+ \frac{1}{X_0 + Y_0}\right)
\end{align}
Combining
this with 
\eqref{eq:step_for_different_rates_2_condition} yields
$\sum_{k=0}^{n-1} p_k \leq \varepsilon/3$, for a large enough constant
$C$. Defining $q_{k,l} = \prod_{i=k}^\ell (1+p_i)$ as in Lemma
\ref{lemma: estimates for martingales} we have for all $1 \leq k \leq
\ell \leq n-1$
\begin{equation}\label{eq:diff_rates_sum_of_ps}
q_{k,\ell} \leq e^{\sum_{k=0}^{n-1} p_k} \leq e^{\varepsilon/3} \leq
\tfrac{3}{2}.
\end{equation}

Define the martingale $I_0=K_0'$, $I_k = K_k' -
\sum_{\ell=0}^{k-1}\mathbb{E}(K_{\ell+1}'-K_{\ell}'|\mathcal{F}_\ell)$
as in Lemma \ref{lemma: estimates for martingales}, which together
with \eqref{eq:diff_rates_sum_of_ps} implies
\begin{equation}\label{eq:diff_rates_difference_between_K_and_I}
|K_k' - I_k| \leq \frac{3}{2}\sum_{\ell=0}^{k-1}p_\ell I_\ell.
\end{equation}

Next estimate the second moment of jumps. Define 
\[
r_k = \frac{C_1M^{-\beta}}{(M-2k)f(1-2k/M)^{1+\beta}},
\]
so that by Lemma \ref{lemma:estimate_on_the_integral} for $\gamma=1+\beta$ we have
\[
\sum_{k=0}^{n-1} r_k \leq C_0C_1M^{-\beta} \left(\Big(\frac{M}{M-2n}\Big)^{(1+\beta)/2} + \Big(\frac{M}{X_0+Y_0}\Big)^{\beta}\right).
\]
Lemma \ref{lemma: estimates for martingales}
now yields
\begin{align*}
\mathbb{E}((I_n-I_0)^2) \leq \frac{3C_0C_1 I_0}{2M^\beta}\left(\Big(\frac{M}{M-2n}\Big)^{(1+\beta)/2} + \Big(\frac{M}{X_0+Y_0}\Big)^{\beta}\right)
\end{align*}
Combined with Doob's maximal
inequality, this implies
\begin{align*}
\mathbb{P}(|I_k -I_0| \geq \frac{\varepsilon}{3} I_0, \text{ for some } 0
\leq k \leq n) & \leq \frac{27 C_0 C_1 Y_0^\beta}{M^\beta X_0\varepsilon^2} \left(\Big(\frac{M}{M-2n}\Big)^{(1+\beta)/2} + \Big(\frac{M}{X_0+Y_0}\Big)^{\beta}\right)\\
& \leq
\frac{C_2}{\varepsilon^2}\Big(\frac{M^{(1-\beta)/2}}{(M-2n)^{(1+\beta)/2}K_0}
+ \frac{1}{X_0}\Big),
\end{align*}
for some constant $C_2=C_2(\beta,d,c)$.
If $|I_k -I_0| \leq \frac{\varepsilon}{3} I_0$, for all $0 \leq k \leq
n$, then \eqref{eq:diff_rates_difference_between_K_and_I} and the
inequality $\sum_{k=0}^{n-1}p_k \leq \varepsilon/3$ imply
\[
|K_k'-K_0'| \leq |K_k'-I_k| + |I_k - I_0| \leq
\frac{3}{2}\Big(1+\frac{\varepsilon}{3}\Big)I_0\frac{\varepsilon}{3} +
\frac{\varepsilon}{3}I_0 \leq \varepsilon K_0.
\]
Thus we have
\[
\mathbb{P}(|K_k' -K_0'| \geq \varepsilon K_0, \text{ for some } 0 \leq k
\leq n) \leq
\frac{C_2}{\varepsilon^2}\Big(\frac{M^{(1-\beta)/2}}{(M-2n)^{(1+\beta)/2}K_0}
+ \frac{1}{X_0}\Big).
\]

Define $1- \delta_M$ to be the probability that
\begin{equation}\label{eq: an event to estimate the sum of colors}
X_k + Y_k \geq \frac{1}{2}\Big(M-2k-Z_0(1-2k/M)^{d/2}\Big)
\end{equation}
holds for all $0 \leq k \leq n$. By Theorem \ref{thm: urn applied to model} we have that
$\lim_{M \to \infty}\delta_M=0$. Since $K_k' = K_{k \wedge \tau}$ for
$k \leq \sigma \wedge n$ it is enough to show that
$\mathbb{P}(\sigma<\tau\wedge n) \leq \delta_M$. To this end simply
observe that on the event in \eqref{eq: an event to estimate the sum of colors}, inequality
\eqref{eq:step_for_different_rates_2_condition_lower_boud_for_Y},
$\sigma < \tau \wedge n$ and the fact that $Y_\sigma \leq 2d$ imply
\[
c_1(M-2k)\Big(1- \frac{Z_0}{M}(1-2k/M)^{d/2-1}\Big) \leq Y_k \leq
2d+2,
\]
for $k=\sigma-1$. However, by \eqref{eq:diff_rates_choise_of_constant}
and the fact that $\varepsilon \leq 1/2$, this is impossible for $C$
large enough in \eqref{eq:step_for_different_rates_2_condition}
(recall that the value of $c_1$ depended only on $c$ and $\beta$).
\end{proof}

\begin{remark}\label{rem:doob_vs_freedman}
A more careful analysis of the process $K_n$ would allow one to replace Doob's maximal inequality with Freedman's inequality (see \cite{Freedman75}), and obtain exponential bound in the statement of Lemma \ref{lemma:step_for_different_rates_2}. 
However the bound above suffices to our purposes and,
to avoid even more tedious analysis, we use Doob's maximal inequality. 
\end{remark}

It is perhaps inconvenient to apply Lemma \ref{lemma:step_for_different_rates_2} as the assumption \eqref{eq:step_for_different_rates_2_condition_definition} 
already involves an estimate on $K_k$ one would have to check. Fortunately, the following easy corollary handles this problem.

\begin{corollary}\label{cor:step_for_different_rates_3}
Let $0 < \varepsilon \leq 1/2$. Assume that for a positive real number
$c$ and an integer $n$ the inequality
\[
\frac{K_0}{M^{1-\beta}} \leq c f\Big(1-\frac{2k}{M}\Big)^{1-\beta}
\]
is satisfied for $0 \leq k \leq n$.  Then there exists a sequence
$(\delta_M)$ converging to $0$ and depending only on $\beta$ and $d$, and a
constant $C = C(\beta,d,c)>0$, such that
\begin{equation}\label{eq:step_for_different_rates_3_probability}
\mathbb{P}(|K_{k} -K_0| \geq \varepsilon K_0, \text{ for some } 0 \leq k
\leq n) \leq
\frac{C}{\varepsilon^2}\Big(\frac{M^{(1-\beta)/2}}{(M-2n)^{(1+\beta)/2}K_0}
+\frac{1}{X_0}\Big) + \delta_M,
\end{equation}
whenever
\[
\varepsilon \geq C\Big(\frac{1}{M-2n} + \frac{1}{X_0+Y_0}\Big).
\]
\end{corollary}

\begin{proof}
The assumption in the statement simply reads
\[
\frac{K_0}{M^{(1-\beta)/2}} \leq c(M-2k)^{(1-\beta)/2}\Big(1-
\frac{Z_0}{M}(1-2k/M)^{d/2-1}\Big)^{1-\beta}.
\]
Define stopping time $\tau$ as the smallest integer $k$ such that
\[
\frac{K_k}{M^{(1-\beta)/2}} > 2c(M-2k)^{(1-\beta)/2}\Big(1-
\frac{Z_0}{M}(1-2k/M)^{d/2-1}\Big)^{1-\beta}.
\]
Applying Lemma \ref{lemma:step_for_different_rates_2} we conclude that
the event
\begin{equation}\label{eq:diff_rates_first_part}
|K_{k \wedge \tau} -K_0| \leq \varepsilon K_0, \text{ for all } 0 \leq
k \leq n,
\end{equation}
has probability of at least
\[
1-
\frac{C}{\varepsilon^2}\Big(\frac{M^{(1-\beta)/2}}{(M-2n)^{(1+\beta)/2}K_0}
+\frac{1}{X_0}\Big) - \delta_M,
\]
for an appropriately chosen constant $C$.  Since $\tau < n$ implies
$K_\tau > 2K_0$, on the event in \eqref{eq:diff_rates_first_part} we
have that $\tau \geq n$. Thus in the event in
\eqref{eq:diff_rates_first_part} we can replace $K_{k \wedge \tau}$ by
$K_k$ which completes the proof.
\end{proof}

As we said, the assumption in the previous corollary is more convenient than \eqref{eq:step_for_different_rates_2_condition_definition}. However, the problem is that it might fail to hold throughout the time regime. To overcome this problem, we switch the roles of processes $X$ and $Y$ and the value of the parameter $\beta$ to $1/\beta$ when this condition fails to hold. This yields the following lemma.


\begin{lemma}\label{lemma:diff_rates_process_analysis}
Let $\beta$ be any positive real number. Let $(L_M)$ and $(L_M')$ be two sequences of
positive numbers converging to $\infty$ and such that $L_M \geq L_M'$ and $\lim_M L_M
M^{-\gamma} = 0$, for any $\gamma >0$.  Assume that $\min(X_0, Y_0)\geq L_M$.
Define $n_0$ as
the largest integer such that
\begin{equation}\label{eq:n0 definition}
M-2n_0 \geq L_M' \Big(\frac{M^{(1-\beta)/2}}{K_0} \vee
\frac{K_0}{M^{(1-\beta)/2}}\Big)^{2/(1+\beta)}.
\end{equation}
Then for any $\varepsilon > 0$ there is a sequence of numbers $\eta_M$ 
converging to zero such that
\[
\mathbb{P}(|K_n - K_0| \leq \varepsilon K_0, \ \text{ for all } 0 \leq n \leq n_0) \geq 1-\eta_M.
\]
\end{lemma}

\begin{remark}\label{rem:lower_bound_for_the_smaller_start}
Note that the slower $L_M'$ grows compared to $L_M$ the stronger the result is. In fact the sequence $\eta_M$ will depend on $L_M'$, $\beta$, $d$ and $\varepsilon$ (but not explicitly on $X_0$ and $Y_0$). We need this flexibility in the applications of Lemma \ref{lemma:diff_rates_process_analysis}.
Observe that  $X_0 \geq
L_M$ and $Y_0 \geq L_M$
imply that
\[
L_M M^{-\beta}\leq K_0 \leq M L_M^{-\beta}
\]
and therefore
\[
1-2n_0/M \leq L_M' \Big(L_M^{-2/(1+\beta)} \vee L_M^{-2\beta/(1+\beta)}\Big) + 2/M.
\]
In particular, for any $\gamma > 0$ there is $\delta>0$ such that when $L_M' L_M^{-\delta}$ converges to zero,  $L_M' (1-2n_0/M)^{\gamma}$ converges to 0 as well.
\end{remark}

First we prove a few technical details. We start by recalling Theorem \ref{thm: urn applied to model} we know that
for any $\varepsilon > 0$ there is a sequence $(\delta_M)$ converging
to $0$ such that with probability of at least $1- \delta_M$ we have
that for every $0 \leq n \leq \tfrac{M}2-1$
\begin{equation}\label{eq:diff_rates_basic_assumption}
1-\varepsilon \leq \frac{X_n+Y_n}{M-2n - Z_0 (1-2n/M)^{d/2}} =
\frac{X_n+Y_n}{f(1-2n/M)\sqrt{M(M-2n)}} \leq 1+\varepsilon~.
\end{equation}

The following two simple claims will be helpful when switching the roles of the processes $X$ and $Y$.

\begin{lemma}\label{lemma:switching_help}
There exists a sequence $\delta_M$ depending only on $\varepsilon$, $d$ and $\beta$ and converging to zero such that, with probability of at least $1-\delta_M$, for any non-negative integers $n$ and $k$ such that $M-2n-2k>1$ we have
\begin{equation}\label{eq:diff_rates_main_result_help}
(1-\varepsilon) \leq
  \frac{1-\frac{Z_n}{M-2n}\Big(1-\frac{2k}{(M-2n)}\Big)^{d/2-1}}{1-
    \frac{Z_0}{M}\Big(1-\frac{2(n+k)}{M}\Big)^{d/2-1}} \leq
  (1+\varepsilon)
\end{equation}
\end{lemma}

\begin{proof}
Using that $X_n +Y_n = M-2n-Z_n$ for all $n$ (which holds with probability converging to 1), the inequalities 
 \eqref{eq:diff_rates_basic_assumption} can be rewritten as
\[
(1+\varepsilon)\frac{Z_0}{M}\Big(1-\frac{2n}{M}\Big)^{d/2-1} -\varepsilon
\leq \frac{Z_n}{M-2n} \leq \varepsilon + (1-
\varepsilon)\frac{Z_0}{M}\Big(1-\frac{2n}{M}\Big)^{d/2-1}.
\]
It is not hard to check that this in turn implies \eqref{eq:diff_rates_main_result_help}.
\end{proof}

\begin{lemma}\label{lemma:switching_help_2}
Suppose the assumptions of Lemma \ref{lemma:diff_rates_process_analysis} hold. Assume that $\beta \neq 1$ and let $c>1$. Then there is a sequence $\delta_M$, depending on $\beta$ and $d$ and converging to zero, and a positive constant $c'$ depending on $\beta$, $c$ and $d$  such that with the probability of at least $1-\delta_M$, for all $0 \leq k \leq n_0$
\begin{equation}\label{eq:diff_rates_break_values_balance_property}
\frac{1}{c}f(1-2k/M) \leq \frac{K_k^{1/(1-\beta)}}{M} \leq cf(1-2k/M)
\Rightarrow X_k \wedge Y_k \geq c'L_M'.
\end{equation}
\end{lemma}

\begin{proof}
We first show that if the assumptions hold then for a constant $c''=c''(\beta,c,d)$
\[
\frac{1}{c}f(1-2k/M) \leq \frac{K_k^{1/(1-\beta)}}{M} \leq cf(1-2k/M)
\Rightarrow X_k \wedge Y_k \geq c''(X_k+Y_k).
\]
If $X_k < c''(X_k + Y_k)$ then $Y_k > (1-c'')(X_k + Y_k)$ and
\[
\frac{K_k^{1/(1-\beta)}}{M} =
\left(\frac{X_k}{Y_k^\beta}\right)^{1/(1-\beta)} \frac{1}{\sqrt{M(M-2k)}}
\lessgtr \frac{c''^{1/(1-\beta)}}{(1-c'')^{\beta/(1-\beta)}} \frac{X_k +
  Y_k}{\sqrt{M(M-2k)}},
\]
where the inequality in $\lessgtr$ is $<$ for $\beta < 1$ and $>$ for
$\beta > 1$. Using \eqref{eq:diff_rates_basic_assumption} to bound the
term $(X_k + Y_k)(M(M-2k))^{-1/2}$ we obtain a contradiction with the
left hand side of \eqref{eq:diff_rates_break_values_balance_property}
for $c''$ such that $c''(1-c'')^{-\beta} \leq
(c/(1-\varepsilon))^{-|1-\beta|}$, which yields $X_k \geq c'' (X_k
+Y_k)$. In the same way one can show that $Y_k \geq c''(X_k + Y_k)$ for
an appropriately chosen $c''$.

To finish the proof we show that or every $0
\leq n \leq n_0$
\[
X_n + Y_n \geq L_M'/3.
\]
To check this, by
\eqref{eq:diff_rates_basic_assumption} it is enough to check that
$\phi(t) \geq \frac{2L_M'}{3M}$ for $1-2n_0/M \leq t \leq 1$, where
$\phi(t)= t-t^{d/2}Z_0/M$. By the concavity of $\phi$  and the fact that $M - 2n_0 \geq L_M'$ it is enough to
check the lower bound for $t= L_M'/M$ and $t=1$ for which the claim is
obvious.  
\end{proof}

\begin{lemma}\label{lemma:starting_help}
If $X_0 \leq Y_0$, then $f(1) \geq K_0^{1/(1-\beta)}M^{-1}$ for $\beta < 1$ and $f(1) \leq 2K_0^{1/(1-\beta)}M^{-1}$ for $\beta>1$. 
\end{lemma}

\begin{proof}
To prove the statement for $\beta<1$ simply observe that it is
equivalent to $X_0 Y_0^{-\beta} \leq (X_0 + Y_0)^{1-\beta}$ and to
\[
\frac{X_0}{X_0 + Y_0} \leq \Big(\frac{Y_0}{X_0+Y_0}\Big)^\beta,
\]
which, because of $X_0 \leq Y_0$ surely holds for $\beta <
1$. The statement for $\beta>1$ is similarly
equivalent to
\[
\frac{2X_0}{X_0 + Y_0} \leq
\Big(\frac{2Y_0}{X_0+Y_0}\Big)^\beta,
\]
which again holds, since the left hand side is smaller than $1$ and the right hand side is larger than $1$.
\end{proof}

\begin{proof}[Proof of Lemma \ref{lemma:diff_rates_process_analysis}]
First observe that without loss of generality we can assume that $X_0 \leq Y_0$.
Otherwise, we simply replace the roles of processes $X$ and $Y$ and setting
$X$ to have rate $1$ and $Y$ to have the rate $1/\beta$. This causes the process $K_n$ to
become $K_n^{-1/\beta}$ while the value of $n_0$  remains unchanged. 

Throughout the proof we assume that $M$ is sufficiently large for the
estimates to hold. We can assume $\varepsilon < 1/2$ and $\varepsilon \geq C'/L_M'$
for any constant $C'$ (at different
stages in the proof we choose convenient values for $C'$).  
In particular, as $M-2n_0 \geq L_M'$ and $L_M ' \leq L_M$ this assumption implies
\begin{equation}\label{eq:diff_rates_assumption_on_epsilon}
\varepsilon \geq \frac{C'}{2}\Big(\frac{1}{M-2n_0} + \frac{1}{X_0+Y_0}\Big),
\end{equation}
which will enable us to apply Corollary \ref{cor:step_for_different_rates_3}.



First we present the bound for the simplest
case when $\beta = 1$. Because $X_0 \leq Y_0$ we have $K_0 \leq 1$ and
in this case $n_0$ is the largest integer with the property that
$M-2n_0 \geq L_M'/K_0$.  By Corollary
\ref{cor:step_for_different_rates_3} applied with $C_0 = 1$ we have
that with probability at least
\[
1 - \frac{C}{\varepsilon^2}\Big(\frac{1}{(M-2n_0)K_0} +
\frac{1}{X_0}\Big) - \delta_M \geq 1 - \frac{2C}{\varepsilon^2 L_M'} -
\delta_M,
\]
we have
\[
|K_{k} - K_0| \leq \varepsilon K_0, \text{ for all } 0 \leq k \leq n_0,
\]
which proves the claim when $\beta =1$.

Next we analyze the  case of $\beta
\neq 1$. 
Fix numbers $0 < c_2 < 1$ and $c_1 > 2$. 
Note that $f$ from \eqref{eq: function f} is a concave nonnegative function on $[0,1]$ and
$f(0)=0$.  By Lemma \ref{lemma:starting_help} if $\beta < 1$ there
is a unique point $0 < t_2 < 1$ such that $f(t_2) = c_2
K_0^{1/(1-\beta)}/M$.  If $\beta > 1$ then in the case when
$\max_{[0,1]}f > c_1 K_0^{1/(1-\beta)}/M$ denote by $t_2$ the
smallest element in $f^{-1}(c_2 K_0^{1/(1-\beta)}/M)$ and by $t_1$ the
largest element in $f^{-1}(c_1 K_0^{1/(1-\beta)}/M)$.  Define
$\overline{n}_1$ as the largest integer such that $M-2\overline{n}_1
\geq Mt_1$, and $\overline{n}_2$ the largest integer such that
$M-2\overline{n}_2 \geq Mt_2$.  Furthermore, define $n_1 =
\overline{n}_1 \wedge n_0$ and $n_2=\overline{n}_2 \wedge n_0$.

We also need the following inequality 
\begin{equation}\label{eq:diff_rates_break_values}
c_2 K_0^{1/(1-\beta)}/M \leq f(1-2n_i/M) \leq c_1 K_0^{1/(1-\beta)}/M,
\end{equation}
whenever $n_i < n_0$, for $i=1,2$. 
When $\beta < 1$ the lower bound $f(1-2n_2/M) \geq  c_2 K_0^{1/(1-\beta)}/M$ follows from concavity of $f$, Lemma \ref{lemma:starting_help} and the fact that 
$f(1-2n_2/M) \geq f(t_2)\wedge f(1)$.
For $\beta > 1$ the condition $X_0 \leq Y_0$ implies that $K_0^{1/(1-\beta)} \geq Y_0$. Since $f'(t) \geq -(d-1)/2$ for any $t$ we have $M(f(1-2n_i/M) - f(t_i))$ is bounded from below, and since $K_0^{1/(1-\beta)} \to \infty$ the lower bound follows.
For the upper bounds in \eqref{eq:diff_rates_break_values} it suffices to prove that 
$f'(t) \leq K_0^{1/(1-\beta)}/2$ whenever
\[
t \geq \frac{M-2n_0}{M}, \text{ and } c_2 \frac{K_0^{1/(1-\beta)}}{M}
\leq f(t) \leq c_1 \frac{K_0^{1/(1-\beta)}}{M}.
\]
To end this use a simple inequality $f'(t)\leq \frac{(d-1)f(t)}{2t}$ and then use it to show that for all $t$ as above we have $t \geq L_M'/M$ and
\[
f'(t) \leq \frac{c_1(d-1)K_0^{1/(1-\beta)}/M}{2L_M'/M} = \frac{c_1(d-1)}{2L_M'}K_0^{1/(1-\beta)}.
\]
This finishes the proof of \eqref{eq:diff_rates_break_values}.


We separate the analysis into three cases:
\begin{itemize}
\item[a)] $\beta > 1$ and $\max_{[0,1]}f \leq c_1
  K_0^{1/(1-\beta)}/M$,
\item[b)] $\beta > 1$ and $\max_{[0,1]}f > c_1 K_0^{1/(1-\beta)}/M$,
\item[c)] $\beta < 1$.
\end{itemize}

To summarize, in case a) we have
\begin{equation}\label{eq:diff_rates_function_behavior_a}
f(1-2k/M) \leq c_1 K_0^{1/(1-\beta)}/M, \ \text{ for } 0 \leq k \leq
M/2-1,
\end{equation}
in case b)
\begin{equation}\label{eq:diff_rates_function_behavior_b}
f(1-2k/M) \left\{\begin{array}{ll} \leq c_1 K_0^{1/(1-\beta)}/M, &
\text{ for } 0 \leq k \leq n_1,\\ \geq c_2 K_0^{1/(1-\beta)}/M, &
\text{ for } n_1 \leq k \leq n_2, \text{ if } n_1 < n_0,\\ \leq c_1
K_0^{1/(1-\beta)}/M, & \text{ for } n_2 \leq k \leq M/2-1, \text{ if }
n_2 < n_0,
\end{array}\right.
\end{equation}
and in the case c)
\begin{equation}\label{eq:diff_rates_function_behavior_c}
f(1-2k/M) \left\{\begin{array}{ll} \geq c_2 K_0^{1/(1-\beta)}/M, &
\text{ for } 0 \leq k \leq n_2, \\ \leq c_1 K_0^{1/(1-\beta)}/M, &
\text{ for } n_2 \leq k \leq M/2-1, \text{ if } n_2 < n_0.
\end{array}\right.
\end{equation}

For the case a) note that \eqref{eq:diff_rates_function_behavior_a}
can be rewritten as
\begin{equation}\label{eq:diff_rates_case_1_1}
\frac{K_0}{M^{1-\beta}} \leq c_1^{\beta
  -1}f(1-2k/M)^{1-\beta}.
\end{equation}
By
\eqref{eq:diff_rates_assumption_on_epsilon} we can apply Corollary
\ref{cor:step_for_different_rates_3} to get that the event that $|K_k
- K_0| \leq \varepsilon K_0$ for $0 \leq k \leq n_0$ has probability
at least
\[
1-
\frac{C}{\varepsilon^2}\Big(\frac{M^{(1-\beta)/2}}{(M-2n_0)^{(1+\beta)/2}K_0}
+\frac{1}{X_0}\Big) - \delta_M \geq 1-
\frac{C}{\varepsilon^2}\Big(\frac{1}{L_M'^{(1+\beta)/2}}+
\frac{1}{L_M}\Big) - \delta_M.
\]
The inequality above follows from the definition of $n_0$. This
suffices for the case a).

Next we assume that we are under the assumptions of case b).
From the first inequality in \eqref{eq:diff_rates_function_behavior_b}
we obtain that \eqref{eq:diff_rates_case_1_1} holds for $0 \leq k \leq
n_1$.
Because $M-2n_1 \geq M-2n_0$ we can apply Corollary
\ref{cor:step_for_different_rates_3} like in the case a) and conclude
that the event that $|K_{k} -K_0| \leq \frac{\varepsilon}{3} K_0$ holds
for all $0 \leq k \leq n_1$, has probability of at least
\[
1-
\frac{9C}{\varepsilon^2}\Big(\frac{M^{(1-\beta)/2}}{(M-2n_1)^{(1+\beta)/2}K_0}
+\frac{1}{X_0}\Big) - \delta_M \geq 1-
\frac{9C}{\varepsilon^2}\Big(\frac{1}{L_M'^{(1+\beta)/2}}+
\frac{1}{L_M'}\Big) - \delta_M.
\]
Now if $n_1=n_0$ we are done with the analysis in the case b).

Otherwise assume that $|K_{k} -K_0| \leq \frac{\varepsilon}{3} K_0$ holds
for all $0 \leq k \leq n_1$ indeed, and note that
\eqref{eq:diff_rates_break_values} implies that
\begin{equation}\label{eq:diff_rates_case_b_first_balance}
\frac{(1+\varepsilon)^{1/(1-\beta)}}{c_1}f(1-2n_1/M) \leq
\frac{K_{n_1}^{1/(1-\beta)}}{M} \leq
\frac{(1-\varepsilon)^{1/(1-\beta)}}{c_2} f(1-2n_1/M),
\end{equation}
which then by \eqref{eq:diff_rates_break_values_balance_property} 
implies that both $X_{n_1}$
and $Y_{n_1}$ are at least $c'L_M'$ for some constant $c'$.

Define $M' = M-2n_1$, $X_k' = X_{n_1+k}$, $Y_k' = Y_{n_1+k}$, $Z_k' =
Z_{n_1+k}$ and
\[
K_k' = \frac{Y_k'}{X_k'^{1/\beta}(1-2k/M')^{(1-1/\beta)/2}}.
\]
It is easy to check that in fact
\begin{equation}\label{eq:diff_rates_process_connection}
K_k' = K_{n_1+k}^{-1/\beta}\Big(\frac{M'}{M}\Big)^{(1-1/\beta)/2}.
\end{equation}
Similarly to \eqref{eq:diff_rates_case_1_1}, the second inequality in
\eqref{eq:diff_rates_function_behavior_b} implies that for $n_1 \leq k
\leq n_2$
\[
\frac{K_0}{M^{(1-\beta)/2}} \geq c_2^{\beta
  -1}(M-2k)^{(1-\beta)/2}\Big(1-
\frac{Z_0}{M}(1-2k/M)^{d/2-1}\Big)^{1-\beta}~.
\]
Combined with the inequality $K_{n_1} \geq (1-\varepsilon)K_0$,
\[
K_{n_1} \geq (1-\varepsilon)c_2^{\beta
  -1}M^{(1-\beta)/2}(M-2(n_1+k))^{(1-\beta)/2}\Big(1-
\frac{Z_0}{M}(1-2(n_1+k)/M)^{d/2-1}\Big)^{1-\beta},
\]
for $0 < k \leq n_2- n_1$.  Raising the above inequality to the power
of $-1/\beta$ and using \eqref{eq:diff_rates_main_result_help} and
\eqref{eq:diff_rates_process_connection} we obtain
\begin{equation}\label{eq:diff_rates_jump_condition}
\frac{K_0'}{M'^{(1-1/\beta)/2}} \leq
\frac{(1+\varepsilon)^{1-1/\beta}}{(1-\varepsilon)^{1/\beta}c_2^{1-1/\beta}}(M'-2k)^{(1-1/\beta)/2}\Big(1-\frac{Z_0'}{M'}(1-2k/M')^{d/2-1}\Big)^{1-1/\beta},
\end{equation}
for all $0 \leq k \leq n_2-n_1$.
Since $X_0' + Y_0' > 2c'L_M'$ and $M'-2(n_2-n_1) \geq M-2n_0 $ and \eqref{eq:diff_rates_assumption_on_epsilon}
we can apply Corollary \ref{cor:step_for_different_rates_3} to conclude that the
event $|K_{k}' -K_0'| \leq \frac{\varepsilon}{2^{\beta+3}} K_0'$, for all
$0 \leq k \leq n_2-n_1$, is of probability at least
\begin{align}\label{eq:diff_rates_second_part}
&1-
  \frac{4^{\beta+3}C}{\varepsilon^2}\Big(\frac{M'^{(1-1/\beta)/2}}{(M'-2(n_2-n_1))^{(1+1/\beta)/2}K_0'}
  +\frac{1}{Y_0'}\Big) - \delta_{M'} \nonumber \\ & \geq 1-
  \frac{4^{\beta+3}C}{\varepsilon^2}\Big(\frac{(1+\varepsilon)^{1/\beta}
    K_0^{1/\beta}M^{(1-1/\beta)/2}}{(M-2n_2)^{\frac{1+\beta}{2\beta}}}
  +\frac{1}{Y_{n_1}}\Big) - \delta_{L_M'} \nonumber \\ & \geq 1-
  \frac{4^{\beta+3}C}{\varepsilon^2}\Big((1+\varepsilon)^{1/\beta}
  L_M'^{-\frac{1+\beta}{2\beta}} +(c'L_M')^{-1}\Big) - \delta_{L_M'},
\end{align}
where we used the fact that $n_2 \leq n_0$, the definition of $n_0$
and the lower bound $Y_{n_1} \geq c'L_M'$.  Then this event can be
rewritten as
\[
\left|\Big(\frac{K_{n_1+k}}{K_{n_1}}\Big)^{-1/\beta} - 1\right| \leq
\frac{\varepsilon}{2^{\beta+3}}, \text{ for all } 0 \leq k \leq n_2-n_1,
\]
which, using the fact that $\varepsilon \leq 1/2$ easily implies that
\[
\left|K_{n_1+k} - K_{n_1}\right| \leq \frac{\varepsilon}{4} K_{n_1},
\text{ for all } 0 \leq k \leq n_2-n_1,
\]
and
\begin{equation}\label{eq:diff_rates_secon_part_main}
\left|K_{k} - K_{0}\right| \leq \left|K_{k} - K_{n_1}\right| +
\left|K_{n_1} - K_{0}\right| \leq
\frac{\varepsilon}{4}\Big(1+\frac{\varepsilon}{3}\Big)K_0 +
\frac{\varepsilon}{3}K_0 \leq \frac{2\varepsilon}{3} K_0, \text{ for all }
n_1 \leq k \leq n_2.
\end{equation}
If $n_2 = n_0$ we are done. 

Otherwise, assume that the event in
\eqref{eq:diff_rates_secon_part_main} holds and observe that
\eqref{eq:diff_rates_case_b_first_balance} holds when $n_1$ is
replaced by $n_2$.
Thus again we have that $X_{n_2} \geq c' L_M'$, and $Y_{n_2} \geq c'L_M'$.

Define $M'' = M-2n_2$, $X_k'' = X_{n_2+k}$, $Y_k'' = Y_{n_2+k}$,
$Z_k'' = Z_{n_2+k}$ and
\[
K_k'' = \frac{X_k''}{Y_k''^{\beta}(1-2k/M'')^{(1-\beta)/2}} =
K_{n_2+k}\Big(\frac{M''}{M}\Big)^{(1-\beta)/2}.
\]
Following the argument that lead to
\eqref{eq:diff_rates_jump_condition}, and using the third inequality
in \eqref{eq:diff_rates_function_behavior_b} we can deduce that
\[
\frac{K_0''}{M''^{(1-\beta)/2}} \leq
(1+\varepsilon)(1-\varepsilon)^{1-\beta}
c_1^{\beta-1}(M''-2k)^{(1-\beta)/2}\Big(1-\frac{Z_0''}{M''}(1-2k/M'')^{d/2-1}\Big)^{1-\beta}.
\]
Since $X_0'' + Y_0'' > 2c'L_M'$ and $M''-2(n_0-n_2) = M-2n_0$ and \eqref{eq:diff_rates_assumption_on_epsilon} we can apply Corollary \ref{cor:step_for_different_rates_3} to conclude  that with
probability at least
\begin{align*}
&1-
  \frac{16C}{\varepsilon^2}\Big(\frac{M''^{(1-\beta)/2}}{(M''-2(n_0-n_2))^{(1+\beta)/2}K_0''}
  +\frac{1}{X_0''}\Big) - \delta_{M''}\nonumber\\ &\geq 1-
  \frac{16C}{\varepsilon^2}\Big(\frac{M^{(1-\beta)/2}}{(1-2\varepsilon/3)(M-2n_0)^{\frac{1+\beta}{2}}K_{0}}
  +\frac{1}{X_{n_1}}\Big) - \delta_{L_M} \nonumber \\ & \geq 1-
  \frac{16C}{\varepsilon^2}\Big((1-2\varepsilon/3)^{-1} L_M'^{-(1+\beta)/2}
  +(c'L_M')^{-1}\Big) - \delta_{L_M}.
\end{align*}
the event
\[
|K_{k}'' -K_0''| \leq \frac{\varepsilon}{4} K_0'', \text{ for all } 0
\leq k \leq n_0-n_2
\]
occurs.  After a glance at the definition of $K_k''$ we proceed as in
the previous step and finish the analysis  in the
case b).

The case c) is handled in the same way. The first inequality in
\eqref{eq:diff_rates_function_behavior_c}, inequality
\eqref{eq:diff_rates_assumption_on_epsilon}
and Corollary \ref{cor:step_for_different_rates_3} imply
that the event $|K_k - K_0| \leq \varepsilon K_0/3$, for $0 \leq k
\leq n_2$, has probability at least
\[
1-
\frac{9C}{\varepsilon^2}\Big(\frac{M^{(1-\beta)/2}}{(M-2n_2)^{(1+\beta)/2}K_0}
+\frac{1}{X_0}\Big) - \delta_M \geq 1-
\frac{9C}{\varepsilon^2}\Big(\frac{1}{L_M'^{(1+\beta)/2}}+
\frac{1}{L_M}\Big) - \delta_M.
\]
This finishes the proof if $n_2 = n_0$. Otherwise, observe that
\eqref{eq:diff_rates_case_b_first_balance} holds and thus we have $Y_{n_2}
\geq c' L_M'$. Then define $X_k'=X_{n_2+k}$, $Y_k'=Y_{n_2+k}$,
$Z_k'=Z_{n_2+k}$, $M'=M-2n_2$ and
\[
K_k' = \frac{Y_k'}{X_k'^{1/\beta}(1-2k/M')^{(1-1/\beta)/2}} =
K_k^{-1/\beta}\Big(\frac{M}{M'}\Big)^{\frac{1-\beta}{2\beta}}.
\]
The second inequality in \eqref{eq:diff_rates_function_behavior_c} and
\eqref{eq:diff_rates_basic_assumption} now imply
\[
\frac{K_0'}{M'^{(1-1/\beta)/2}} \leq
\frac{c_1^{(1-\beta)/\beta}(1+\varepsilon)^{1-1/\beta}}{(1-\varepsilon)^{1/\beta}}(M'-2k)^{(1-1/\beta)/2}\Big(1-\frac{Z_0'}{M'}(1-2k/M')^{d/2-1}\Big)^{1-1/\beta}.
\]
Now we can apply Corollary \ref{cor:step_for_different_rates_3} and
conclude that with probability at least
\begin{align*}
& 1-
  \frac{4^{\beta+1}C}{\varepsilon^2}\Big(\frac{M'^{(1-1/\beta)/2}}{(M-2n_0)^{(1+1/\beta)/2}K_0'}+\frac{1}{Y_0'}\Big)
  - \delta_{M'} \\ & \geq 1-
  \frac{4^{\beta+1}C}{\varepsilon^2}\Big(\frac{(1+\varepsilon)^{1/\beta}K_0^{1/\beta}}{(M-2n_0)^{\frac{1+\beta}{2\beta}}M^{\frac{1-\beta}{2\beta}}}+\frac{1}{c'L_M'}\Big)
  - \delta_{L_M'}\\ & \geq 1-
  \frac{4^{\beta+1}C}{\varepsilon^2}\Big(\frac{(1+\varepsilon)^{1/\beta}}{L_M'^{\frac{1+\beta}{2\beta}}}+\frac{1}{c'L_M'}\Big)
  - \delta_{L_M'},
\end{align*}
we have that $|K_k' - K_0'| \leq \varepsilon 2^{-\beta-1} K_0'$ for all
$0 \leq k \leq n_0 - n_2$. Using the analysis similar to the case b)
we see that this event implies $|K_{n_2+k}-K_0| \leq \varepsilon K_0$,
for $0 \leq k \leq n_0 - n_2$. This finishes the proof.
\end{proof}


The next lemma controls the size of processes for large times.

\begin{lemma}\label{lemma:late_times}
Suppose the conditions of Lemma \ref{lemma:diff_rates_process_analysis} hold. There is a $\delta > 0$ such that when $L_M' L_M^{-\delta} \to 0$ the following holds:
There is a sequence $(\eta_M)$ converging to zero 
such that with the probability at least $1-\eta_M$ we have
\begin{itemize}
\item[(i)] $X_{n+1} \leq X_n \leq 2L_M'^{\frac{1+\beta}{2}}$ for all $n \geq n_0$, in the case $K_0 \leq M^{(1-\beta)/2}$,
\item[(ii)] $Y_{n+1} \leq Y_n \leq 2L_M'^{\frac{1+\beta}{2\beta}}$ for all $n \geq n_0$, in the case $K_0 \geq M^{(1-\beta)/2}$.
\end{itemize}
\end{lemma}

\begin{proof}
When $K_0 \leq
M^{(1-\beta)/2}$, the inequality $Y_{n_0} \leq M-2n_{0}$ implies
\[
K_{n_0} \geq X_{n_0} M^{(1-\beta)/2}(M-2n_0)^{-(1+\beta)/2},
\]
which, by the definition of $n_0$ and the fact that $K_{n_0} \leq
(1+\varepsilon)K_0$, yields $X_{n_0} \leq 2L_M'^{(1+\beta)/2}$, for an appropriate value of $\varepsilon$.
When $K_0 \geq M^{(1-\beta)/2}$ then the inequality $X_{n_0} \leq
M-2n_{0}$ implies
\[
K_{n_0} \leq M^{(1-\beta)/2}(M-2n_0)^{(1+\beta)/2}Y_{n_0}^{-\beta},
\] 
which, by the definition of $n_0$ and the fact that $K_{n_0} \geq
(1-\varepsilon)K_0$, yields $Y_{n_0} \leq
2L_M'^{\frac{1+\beta}{2\beta}}$, for an appropriate value of $\varepsilon$.  If $K_0 \leq
M^{(1-\beta)/2}$ denote by $U$ the process $X$ and $\tau =
\frac{1+\beta}{2}$ and if $K_0 \geq M^{(1-\beta)/2}$ denote by $U$ the
process $Y$ and $\tau = \frac{1+\beta}{2\beta}$  (if $K_0 =
M^{(1-\beta)/2}$ do either). Furthermore denote by $n_1$ and $n_2$ the
largest integers such that
\[
M-2n_1 \geq M\Big(\frac{L_M'}{Z_0}\Big)^{2/d}, \text{ and } M-2n_2
\geq M\Big(\frac{1}{L_M'Z_0}\Big)^{2/d}.
\]



First denote the event $\mathbf{U}_3 = \{U_{k+1}
\leq U_k:  \text{for all }k \geq n_2 \}$. As the value of $U$ can grow only if the value of $Z$ decreases, the event $\mathbf{U}_3$ contains the event that $\{Z_k = 0: \text{for all } k \geq n_2\}$. By Theorem \ref{thm: urn applied to model} the probability of this event converges to 1. This handles the time regime $k \geq n_2$. In particular, the claim is proved if $n_0 \geq n_2$, so we assume $n_0 < n_2$.

To finish the proof denote the events
\[
\mathbf{U}_1 = \{U_{k+1} \leq U_k: \text{for all } n_0 \leq k \leq n_1\}, \mathbf{U}_2
= \{U_{k+1} \leq U_k: \text{for all } n_1 \leq k \leq n_2\}, 
\]
in the case $n_0 \leq n_1$, and just
\[
\mathbf{U}_2
= \{U_{k+1} \leq U_k: n_0 \leq k \leq n_2\}, 
\]
if $n_0 > n_1$.

By Theorem \ref{thm: urn applied to model}, with high probability the process $Z_k$ is bounded from above by a constant multiple of
$Z_0(1-2k/M)^{d/2}$, for all $k \leq n_1$, and $X_k+Y_k$ from below by a constant multiple of $M-2k - Z_0(1-2k/M)^{d/2}$. 
Bounding the probability that this fails by $\delta_M$, by \eqref{eq:concentration_prob_very_likely}, the probability of the event $\mathbf{U}_1$ can be estimated as
\begin{multline*}
\mathbb{P}(\mathbf{U}_1) \geq \mathbb{P}(U_{n_0} \leq 2L_M'^\tau) \prod_{k = n_0}^{n_1}\left(1 -
\frac{c_2L_M'^{\tau}Z_0(1-2k/M)^{d/2}}{(M-2k -
  Z_0(1-2k/M)^{d/2})(M-2k)}\right) - \delta_M \\ \geq
1-\mathbb{P}(U_{n_0} > 2L_M'^\tau) -\frac{c_2}{1-2^{-d/2+1}}\sum_{k =
  n_0}^{n_1}\frac{L_M'^{\tau}Z_0(1-2k/M)^{d/2}}{(M-2k)^2} - \delta_M,
\end{multline*}
for some constant $c_2$, where we used that fact that $Z_0 \leq M$ and $M-2k \leq M/2$, for $k
\geq n_0$ and $\delta$ from the statement small enough (see Remark \ref{rem:lower_bound_for_the_smaller_start}).
It suffices to prove that the above sum converges to $0$. To estimate
it calculate
\begin{align}\label{eq:diff_rates_tail_1}
\frac{L_M'^{\tau}Z_0}{M^2}\sum_{k =
  n_0}^{n_1}\Big(1-\frac{2k}{M}\Big)^{d/2-2} & \leq
\frac{L_M'^{\tau}Z_0}{M} \int_{1-2(n_1+1)/M}^{1-2(n_0-1)/M}t^{d/2-2}dt \nonumber \\
& \leq \frac{2}{d-2}L_M'^{\tau}\Big(\frac{M-2n_0+2}{M}\Big)^{d/2-1},
\end{align}
where the first inequality follows by monotonicity of the function $t
\mapsto t^{d/2-2}$.  
By Remark \ref{rem:lower_bound_for_the_smaller_start} the right hand side converges to zero, for $\delta$ small enough.

Next we bound the probability of $\mathbf{U}_2$. Observe that
\begin{equation}\label{eq:diff_rates_n''_estimate}
M-2n_2 \geq \frac{M}{Z_0^{2/d}L_M'^{2/d}} \geq
\frac{M^{1-2/d}}{L_M'^{2/d}},
\end{equation}
and that with
probability converging to $1$ we have $Z_n \leq 2L_M'$, for $n
\geq n_1$. Using \eqref{eq:concentration_prob_very_likely} as in the bounds for $\mathbb{P}(\mathbf{U}_1)$ it suffices to prove that the following sum converges to zero
\[
 \sum_{k = n_1}^{n_2} \frac{c_3L_M'^{\tau+1}}{(M-2k)^2} =
c_3L_M^{\tau+1}\sum_{k = n_1}^{n_2}\frac{1}{(M-2k)^2} \leq
\frac{2c_3L_M^{\tau+1}}{M-2n_2} \leq
\frac{2c_3L_M^{\tau+1+2/d}}{M^{1-2/d}}.
\]
The right hand side clearly converges to zero, which finishes the
analysis of the event $\mathbf{U}_2$.

\end{proof}


We are now ready to prove  Theorem
\ref{thm:(1+o(1))_concentration_diff_rates_0}.
\begin{proof}[Proof of Theorem \ref{thm:(1+o(1))_concentration_diff_rates_0}]
To have notation more compatible with the previous results in this section (which will be heavily used) we change the notation from $L_N$ in the statement to $L_M$. Also recall we scale the pair $(\beta,\rho)$ so that $\rho = 1$.
Throughout the proof we will assume that $n_0$ is defined through \eqref{eq:n0 definition}, for a sequence $(L_M')$ satisfying $L_M'L_M^{-\delta} \to 0$, for $\delta > 0$ small enough. In particular, we will need the assumptions in Lemma
\ref{lemma:late_times} to be satisfied.


As in Lemma \ref{lemma:late_times} define $n_1$ as the largest integer such that $M-2n_1 \geq
M(L_M'/Z_0)^{2/d}$, and apply
Theorem \ref{thm: urn applied to model} to conclude that for any $\varepsilon > 0$ with probability converging to 1 (as $M \to \infty$) we have
\[
Z_n = (1 \pm \varepsilon) Z_0\Big(1-\frac{2n}{M}\Big)^{d/2}, \text{ for } 0 \leq n \leq
n_1.
\]
Also define $m = n_0 \wedge n_1$. We will assume throughout that $1-2m/M \to 0$, which can accomplished by selecting $\delta > 0$ above small enough.

First we will prove all the results with $R_{\text{fin}}$ replaced by $R_{m}$. 
By Lemma \ref{lemma:diff_rates_process_analysis} and Theorem \ref{thm: urn applied to model}  for any $\varepsilon > 0$ with probability converging to 1 we have
\[
\frac{1-Y_n/(X_n+Y_n)}{\big(Y_n/(X_n+Y_n)\big)^\beta} = (1 \pm
\varepsilon)\frac{K_0}{M^{1-\beta}(1-2n/M)^{(1-\beta)/2}\Big(1-
  \frac{Z_0}{M}(1-2n/M)^{d/2-1}\Big)^{1-\beta}}
\]
for $0 \leq n \leq n_0$. For $\beta=1$ this yields that for any $\varepsilon >0$ with probability converging to 1 we have
\[
\frac{Y_n}{X_n+Y_n} = (1 \pm \varepsilon)\frac{1}{1+K_0} = (1 \pm \varepsilon)
\frac{Y_0}{X_0 + Y_0}, \text{ for } 0 \leq n \leq n_0.
\]
For $\beta \neq 1$, define the function $\varphi_\beta \colon (0,1) \to \mathbb{R}$ by
$\varphi_\beta(t) = (1-t)^{1/(\beta -1)}t^{-\beta/(\beta-1)}$. Since the derivative
of $t \mapsto (1-t)t^{-\beta}$ is bounded away from zero on $(0,1)$, the derivative of the inverse of this function is bounded from above, and so, for any $\varepsilon > 0$ with the probability converging to 1 we have
\begin{equation}\label{eq:formula for ratio}
\frac{Y_n}{X_n+Y_n} = (1\pm
\varepsilon)\varphi_\beta^{-1}\left(MK_0^{1/(\beta-1)}\Big((1-2n/M)^{1/2} -
 \frac{Z_0}{M}(1-2n/M)^{(d-1)/2}\Big)\right), \text{ for } 0 \leq n \leq n_0.
\end{equation}
Define the function $\xi_\beta(s) = \frac{\beta s}{\beta s+(1-s)}$ which has a derivative bounded away from zero on $(0,1)$. Since 
\[
\frac{Y_n}{\beta X_n+Y_n} =  \xi_\beta^{-1}\Big(\frac{Y_n}{X_n+Y_n}\Big),
\]
and $\phi_\beta = \varphi_\beta \circ \xi_\beta$
we have with probability converging to 1 (as $M \to \infty$)
\[
\frac{Y_n}{\beta X_n+Y_n} = (1\pm
\varepsilon)\phi_\beta^{-1}\left(MK_0^{1/(\beta-1)}\Big((1-2n/M)^{1/2} -
 \frac{Z_0}{M}(1-2n/M)^{(d-1)/2}\Big)\right).
\]

With probability converging to 1 (as $M \to \infty$), for all $n$ the conditional probability that in the $(n+1)$-st
step we add a new red vertex is equal to
\[
\frac{Y_n}{\beta X_n + Y_n}\frac{Z_n}{M-2n-1}.
\]
Then for $\beta=1$, with probability converging to 1 (as $M \to \infty$) 
\[
R_m - R_0 = (1\pm \varepsilon) \sum_{n=0}^{m} \frac{Y_0}{X_0+Y_0}
\frac{Z_0}{M}(1-2n/M)^{d/2-1},
\]
where we used \eqref{eq:concentration_prob_very_unlikely} and the fact that the right hand side above converges to $\infty$ (which follows easily from $(X_0 + Y_0)/M \to 0$ and $1-2m/M \to 0$).
By replacing the Riemann sum with the integral
\[
R_m - R_0 = (1\pm \varepsilon)
\frac{Y_0}{X_0+Y_0}\frac{Z_0}{2} \int_{1-2m/M}^1 t^{d/2-1}dt  =
(1\pm\varepsilon)\frac{Y_0}{X_0+Y_0}\frac{Z_0}{d} \Big(1-
(1-2m/M)^{d/2}\Big).
\]
Since $1-2m/M \to 0$ and $Z_0/M \to 1$,
this proves \eqref{eq:1+o(1)_concentration_beta_equal_1} (with $R_m$ instead of $R_{\text{fin}}$).

For $\beta \neq 1$ proceed analogously.  
With probability converging to 1 (as $M \to \infty$)
\begin{align*}
R_m - R_0 & = (1\pm \varepsilon) \sum_{n=0}^{m}
\phi_\beta^{-1}\left(MK_0^{1/(\beta-1)}\Big((1-\tfrac{2n}M)^{1/2} -
\tfrac{Z_0}{M}(1-\tfrac{2n}M)^{(d-1)/2}\Big)\right)
\tfrac{Z_0}{M}(1-\tfrac{2n}M)^{d/2-1} \\ &
= (1\pm \varepsilon) \frac{Z_0}{d}\sum_{n=0}^{m}
\phi_\beta^{-1}\left(MK_0^{1/(\beta-1)}\Big((1-\tfrac{2n}M)^{1/2} -
\tfrac{Z_0}{M}(1-\tfrac{2n}M)^{(d-1)/2}\Big)\right)
\tfrac{d}{M}(1-\tfrac{2n}M)^{d/2-1},
\end{align*}
where, like in the case $\beta=1$, we used \eqref{eq:concentration_prob_very_unlikely} and now the estimates in Lemma \ref{lemma:inverse_estimates} as well (to justify that the right hand side converges to $\infty$).
The above sum on can be replaced as a Riemann sum of the function 
\[h_1(t) = \phi_\beta^{-1}\left(MK_0^{1/(\beta-1)}\Big(t^{1/d} -
  \frac{Z_0}{M}t^{(d-1)/d}\Big)\right),\]
over the interval $[(1-2m/M)^{d/2},1]$ with the subdivision at $(1-2n/M)^{d/2}$, $n=0,1, \dots , m$. To justify this we only need to replace $(1-2(n-1)/M)^{d/2} - (1-2n/M)^{d/2}$ by $dM^{-1}(1-2n/M)^{d/2-1}$, which is possible since the quotient of these two terms converges to 1 as $M \to \infty$ uniformly for all $n = 0,1, \dots , m$ (this follows by Taylor expansion and $M-2m \to \infty$). To justify replacing this Riemann sum with the corresponding integral 
\[
 \int_{(1-2m/M)^{d/2}}^1 h_1(t)~dt =  \int_{(1-2m/M)^{d/2}}^1
  \phi_\beta^{-1}\left(MK_0^{1/(\beta-1)}\Big(t^{1/d} -
  \frac{Z_0}{M}t^{(d-1)/d}\Big)\right)dt,
\]
we only need to show that the number of local extrema of $h_1(t)$ in $[0,1]$ is bounded uniformly in $M$ and that 
\[\frac{\max_{0 \leq t \leq 1} h_1(t)}{M\int_{(1-2m/M)^{d/2}}^1 h_1(t)~dt} \to 0.\] 
This is because one can bound the integral of $h_1(t)$ over each interval on which $h_1(t)$ is monotone from above and below by adding or removing a term in the Riemann sum. 
The first condition follows from the fact that $t^{1/d} -\frac{Z_0}{M}t^{(d-1)/d}$ is concave on $[0,1]$ and $\phi_\beta^{-1}$ is monotone. To check the second condition observe that since $Z_0/M \to 1$ the function $t^{1/d} -\frac{Z_0}{M}t^{(d-1)/d}$ is uniformly bounded away from zero on the interval $[1/4,3/4]$. Since $1-2m/M \to 0$, Lemma \ref{lemma:inverse_estimates} implies that the integral $\int_{(1-2m/M)^{d/2}}^1 h_1(t)~dt$ is bounded form below by a constant (uniform in $M$) times $(M/K_0)^{1/\beta}M^{-1} \wedge 1$. Now the condition follows simply from $\max_{0 \leq t \leq 1}h_1(t) \leq 1$.

So far we have shown that with probability converging to 1 
\[
R_m - R_0  = (1\pm \varepsilon)\frac{Z_0}{d} \int_{(1-2m/M)^{d/2}}^1
  \phi_\beta^{-1}\left(MK_0^{1/(\beta-1)}\Big(t^{1/d} -
  \frac{Z_0}{M}t^{(d-1)/d}\Big)\right)dt.
\]
Now we change the lower bound in the integral to zero, to match the one in 
\eqref{eq:1+o(1)_concentration_beta_no_1_one_parameter}. By Lemma \ref{lemma:inverse_estimates} it suffices to
prove that
\begin{equation}\label{eq:1+o(1)_concentration_fixing_lower_bound}
\int_0^{(1-2m/M)^{d/2}} \frac{M^{1/\beta
    -1}}{K_0^{1/\beta}}\Big(t^{1/d} -
\frac{Z_0}{M}t^{(d-1)/d}\Big)^{1/\beta -1} \wedge 1 \ dt \leq \delta_M
\Big(\frac{M^{1/\beta -1}}{K_0^{1/\beta}} \wedge 1\Big),
\end{equation}
for a sequence $(\delta_M)$, converging to $0$ and depending only on
$\beta$, $d$ and $L_M$. When $\frac{M^{1/\beta -1}}{K_0^{1/\beta}}
\geq 1$ then \eqref{eq:1+o(1)_concentration_fixing_lower_bound} holds
as long as we take $\delta_M \geq (1-2m/M)^{d/2}$. When
$\frac{M^{1/\beta -1}}{K_0^{1/\beta}} < 1$ then
\eqref{eq:1+o(1)_concentration_fixing_lower_bound} holds for
\[
\delta_M \geq \int_0^{(1-2m/M)^{d/2}} h_2(s)^{1/\beta -1} ds,
\] 
where $h_2(s) = s^{1/d}$, for $\beta < 1$ and $h_2(s) = s^{1/d} -
s^{(d-1)/d}$, for $\beta > 1$ (in either case $h_2(s)^{1/\beta -1}$ is
integrable in the neighborhood of $0$).

So far we have proved that when $\beta \neq 1$
\[
R_m - R_0 = (1\pm \varepsilon)\frac{Z_0}{d}\int_0^1 \phi_\beta^{-1}\Big(MK_0^{1/(\beta-1)}(t^{1/d}-\frac{Z_0}{M}t^{1-1/d})\Big)~dt,
\]
with probability converging to 1, as  $M \to \infty$.
We can immediately replace the factor $Z_0/d$ with $M/d$. Removing the factor $Z_0/M$ inside $\phi_\beta^{-1}$ appearing in the integral requires checking that $s_M/t_M \to 1$ implies $\phi_\beta^{-1}(s_M)/\phi_\beta^{-1}(t_M) \to 1$, which is the content of Lemma \ref{lemma:unif_continuity_of_log}. 

It is now easy to check that for all values of $\beta$ (using Lemma \ref{lemma:inverse_estimates} for $\beta \neq 1$) we have that for some constant $c$, with probability converging to 1
\begin{equation}\label{eq:estimates on Rm}
R_m- R_0 \geq c (M \wedge  M^{1/\beta}K_0^{-1/\beta}).
\end{equation}

To conclude the proof we need to replace $R_m$ by $R_{\text{fin}}$. 
For $m=n_0$ and $K_0 \geq M^{(1-\beta)/2}$  it follows from Lemma \ref{lemma:late_times} ii) that actually $R_{\text{fin}} = R_m$ with probability converging to 1, so there is nothing to prove in this case.
To handle other cases recall that the value of $R$ can increase (always by $1$) only if a new uncolored half-edge is matched, that is the value of $Z$ decreases (always by $d$). Thus
  $R_{\text{fin}}-R_m \leq Z_{m}$. 
Thus it suffices to show that with probability converging to $1$ we have $Z_m/(R_{m}-R_0)$ converges to zero in probability.  
For $m=n_0$ and $K_0 \leq M^{(1-\beta)/2}$, by $Z_{n_0} \leq M-2n_0$ and by \eqref{eq:estimates on Rm} it suffices to prove that 
$(M-2n_0)/M$ and $(M-2n_0)(K_0/M)^{1/\beta}$ both converge to zero. Recalling the definition of $n_0$ from \eqref{eq:n0 definition} and replacing $M-2n_0$ by $L_M'M^{\frac{1-\beta}{1+\beta}}K_0^{-\frac{2}{1+\beta}}$, the above reduces to showing that 
\[
L_M'(M^\beta K_0)^{-\frac{2}{1+\beta}} \ \text{ and } \ L_M' (K_0^{\beta-1}M^{1+\beta^2})^{-\frac{1}{\beta(1+\beta)}}
\]
both converge to zero. 
For $m=n_1$ by Theorem \ref{thm: urn applied to model} we have $Z_{n_1} \leq 2L_M'$, so by \eqref{eq:estimates on Rm} it suffices to show that
\[
L_M'/M \ \text{ and } \ L_M'(K_0/M)^{1/\beta}
\]
both converge to zero.
All of this follows from the bounds $L_M M^{-\beta} \leq K_0 \leq ML_M^{-\beta}$ for an appropriate exponent $\kappa > 0$.
\end{proof}

Now we prove Theorem \ref{thm: 1+o(1) concentration for boundary}.


\begin{proof}
The proof will follow the proof of Theorem \ref{thm:(1+o(1))_concentration_diff_rates_0}. However in this proof we do not use the $(1\pm o(1))$ concentration result for $Z_n$ (just for $X_n+Y_n$), so instead of $m = n_0 \wedge n_1$,  we can work with $n_0$, where $n_0$ is as in \eqref{eq:n0 definition}. Again we first prove the statement with $D_{\text{fin}}$ replaced by $D_{n_0}$. At the $n$-th step the conditional probability of connecting a blue to a red vertex is equal to 
\[
\frac{(1+\beta)X_nY_n}{(\beta X_n+Y_n)(M-2n-1)} = \frac{X_n+Y_n}{M-2n-1} \overline{\kappa}_\beta\Big(\frac{Y_n}{X_n+Y_n}\Big),
\]
where $\overline{\kappa}_\beta(t) = \frac{(1+\beta)t(1-t)}{t + \beta(1-t) }$.
With probability converging to 1 as $M \to \infty$ we have
\[
\frac{Y_n}{X_n+Y_n} = (1\pm
\varepsilon) \frac{Y_0}{X_0+Y_0}
, \ \text{ for } \beta = 1
\]
and by \eqref{eq:formula for ratio}
\[
\frac{Y_n}{X_n+Y_n} = (1\pm
\varepsilon)\varphi_\beta^{-1}\left(MK_0^{1/(\beta-1)}\Big((1-2n/M)^{1/2} -
 \frac{Z_0}{M}(1-2n/M)^{(d-1)/2}\Big)\right), \ \text{ for } \beta \neq 1
\]
for all $0 \leq n \leq n_0$, and 
\[
X_n+Y_n = (1\pm \varepsilon) \big(M-2n -Z_0(1-2n/M)^{d/2}\big)
\]
as well.
Observe that the derivative of $\overline{\kappa}_\beta$ is bounded from above on $[0,1]$.
By Remark \ref{rem:nesting} for $\beta =1$ we have 
\[
D_{n_0} - D_0 = (1\pm \varepsilon) \frac{2X_0Y_0}{(X_0+Y_0)^2}\sum_{n=0}^{n_0} \big(1 -\frac{Z_0}{M}(1-2n/M)^{d/2-1}\big).
\]
We can then replace the sum by 
\[
n_0 - \frac{Z_0}{2} \int_{1-2n_0/M}^1 t^{d/2-1}dt = n_0 - \frac{Z_0}{d} + \frac{Z_0}{d}\Big(1- \frac{2n_0}{M}\Big)^{d/2}.
\]
Since $2n_0/M \to 1$ and $Z_0 / M \to 1$ this implies \eqref{eq:(1+o(1))_concentration_boundary_eq_rates}.


For the case $\beta \neq 1$, one can replace the sum approximating $D_{n_0}$ by the corresponding integral
\begin{equation}\label{eq:boundary_preliminary}
D_{n_0} - D_0 = (1\pm \varepsilon)\int_{1-2n_0/M}^1\big(\frac{M}{2} - \frac{Z_0}{2}t^{d/2-1}\big) \overline{\kappa}_\beta \circ \varphi_\beta^{-1}\Big(MK_0^{1/(\beta-1)}\big(t^{1/2}- \frac{Z_0}{M}t^{(d-1)/2}\big)\Big)dt.
\end{equation}
This is justified by splitting the sum into two parts, one containing the factor $M/2$ and the other containing the factor $\frac{Z_0}{2}t^{d/2-1}$, and for each of these two sums applying and argument, analogous to the one in the proof of Theorem \ref{thm:(1+o(1))_concentration_diff_rates_0}.
As established in the proof of Theorem \ref{thm:(1+o(1))_concentration_diff_rates_0} $\phi_\beta = \varphi_\beta \circ \xi_\beta$ and since $\kappa_\beta  = \overline{\kappa}_\beta \circ \xi_\beta$, we get $\overline{\kappa}_\beta \circ \varphi_\beta^{-1} = \kappa_\beta \circ \phi_\beta^{-1}$, so we can make this replacement in \eqref{eq:boundary_preliminary}. Furthermore, replacing $Z_0/2$ with $M/2$ is trivial, while to replace $Z_0/M$ with 1, again use Lemma \ref{lemma:unif_continuity_of_log}.
To fix the lower bound $1-2n_0/M$ follow steps in the proof of  \eqref{eq:1+o(1)_concentration_fixing_lower_bound}.

We are only left to prove that $D_{n_0}$ can be replaced by $D_{\text{fin}}$. For this observe that similarly as in \eqref{eq:estimates on Rm} we have $D_{n_0}- D_0 \geq c (M \wedge  M^{1/\beta}K_0^{-1/\beta})$. On the other hand $D$ can increase only when the value of $X$ and $Y$ both decreases, so it suffices to show that either both $X_{n_0}/M$ and $X_{n_0}(K_0/M)^{1/\beta}$, or both $Y_{n_0}/M$ and $Y_{n_0}(K_0/M)^{1/\beta}$  converge to zero in probability. This follows from Lemma \ref{lemma:late_times} and again the bound $K_0 \leq ML_M^{-\beta}$.
\end{proof}


\section{Dynamics on the torus}
As mentioned in the introduction, the behavior of the competing
infection process is extremely different when the underlying graph is
a $d$ dimensional torus, and not a random $d$-regular graph. The
difference is stated in Theorem \ref{TorusThm}, which is proved in this
section.  The proof relies on the following shape theorem from
\cite{CoxDurrett81}, due to Cox and Durrett. 

In the following theorem we consider the continuous time, rate 1, first passage percolation process on $\mathbb{Z}^2$ started from the origin. Let $\mathcal{S}_t$ be the set that the process occupies at time $t$, thickened by $1/2$, that is $\mathcal{S}_t$ is the union of closed hipercubes of side length $1$ centered at the points explored by the first passage percolation process at time $t$.
The result was originally proven for more general distributions for edge weights, and holds in higher dimensions as well.
\begin{theorem}[Cox, Durrett]
\label{CoxDurrett}
There exists a non-trivial, convex set $A \subset \mathbb{R}^2$ which is symmetric around the origin, and such that for
any $\delta > 0$
\[
\lim_{t \to \infty}\pr \big( (1 - \delta)tA \subset \mathcal{S}_t \subset (1 + \delta) tA \big)
\rightarrow 1.
\]
\end{theorem}
This theorem was generalized for every $d \geq 3$, see for example
\cite{Kesten86}. The theorem can be understood in the following way. For $x \in \mathbb{R}^d$, define $d(x) = \min \{  t |
x \in tA \}$. Since $A$ is convex and symmetric, it is an easy exercise to show that $x \mapsto d(x)$ is a norm on $\mathbb{R}^d$, and thus $d(x,y) = d(y-x)$ is a metric on $\mathbb{R}^d$. Then Theorem \ref{CoxDurrett} says that with probability converging to $1$ as $t\to \infty$, the ball in the random first-passage percolation metric of radius $t$ contains the $d$-metric ball of radius $(1-\delta)t$ and is contained in the  $d$-metric ball of radius $(1+\delta)t$. Furthermore, observe that changing the rate of the first-passage percolation process to $\beta$ simply corresponds to scaling of the set $A$ by a factor of $\beta$.

Assume that in Theorem \ref{TorusThm} we start both processes simultaneously from two uniformly chosen vertices $\mathcal{B}_0 = \{x\}$ and $\mathcal{R}_0=\{y\}$ (that is $B_0 = R_0 =1$). 
Then by Theorem \ref{CoxDurrett} it is easy to see that for any $t_0>0$ and $\varepsilon > 0$, with probability converging to $1$ as $n\to \infty$, every vertex $v \in \torus{N}d$ such that $d(x,v) < (t-\varepsilon)\beta n$ and $d(y,v) > (t+\varepsilon)\rho n$, for some $t> t_0$ satisfies $v \in \mathcal{B}_{\text{fin}}$, and every vertex $v \in \torus{N}d$ such that $d(x,v)> (t+\varepsilon)\beta n$ and $d(y,v) < (t-\varepsilon)\rho n$, for some $t>t_0$ satisfies $v \in \mathcal{R}_{\text{fin}}$. Also observe that for any  $\delta > 0$ we can find $r>0$ such that the probability that $d(x,y) < rn$ is less than $\delta$. 
It follows that by scaling the torus by the factor $1/n$ (to a unit torus), and sending $n \to \infty$, Theorem \ref{CoxDurrett} and the discussion in the previous paragraph imply that the pair of sets $(\mathcal{B}_{\text{fin}}/n, \mathcal{R}_{\text{fin}}/n)$ converge in Hausdorff metric to the Voronoi partition of the continuous unit torus in the metric $d$. More precisely, the limiting set for $\mathcal{B}_{\text{fin}}/n$ is a set of points $v$ on the unit torus $\mathbb{R}^d/[0,1]^d$ for which $d(x,v)/\beta < d(y,v)/\rho$, where $x$ and $y$ are two points chosen uniformly and independently on the torus. This in particular yields Theorem \ref{TorusThm} in this special case, and the proof of the general case presented below is essentially the same.

\begin{proof}[Proof of Theorem \ref{TorusThm}]
Fix $\varepsilon > 0$. Assume first that  $(\mathcal{B}_0,\mathcal{R}_0)$ are chosen uniformly at random of size $(B_0,R_0)$. Then there exists $\delta'>0$ such that, with probability at  least $1-\varepsilon /2$, the Euclidean distance between any pair of points in $\mathcal{B}_0 \cup \mathcal{R}_0$  is at least $\delta' n$. For $\delta''>0$ and every $x \in \mathcal{B}_0$ define the ball $B_x = \{v \in \torus{N}d: d(x,v) < n\beta\delta''\}$, and for every  $y \in \mathcal{R}_0$ define
$R_y = \{v \in \torus{N}d: d(y,v) < n\delta''\rho\}$. It is not hard to see that one can choose $\delta''$ small enough so that all the sets $B_x$ for $x \in \mathcal{B}_0$ and $R_y$, for $y \in \mathcal{R}_0$ are disjoint with probability at least $1-\varepsilon /2$. Conditioned on this event, Theorem \ref{CoxDurrett} yields that for $n$ large enough, with probability at least $1-\varepsilon/2$ for $t = 3\delta'' n/4$, the set $\mathcal{B}_t$ contains all the balls $\frac{1}{2}B_x = \{v \in \torus{N}d: d(x,v) < n\beta\delta''/2\}$, $x \in \mathcal{B}_0$ and is contained in $\cup_{x \in \mathcal{B}_0}B_x$ (and the analogous claim holds for $\mathcal{R}_t$). As all sets $B_x$ and $R_y$ have size linear in $N$, the claim follows.

If on the other hand we select $(\mathcal{B}_0,\mathcal{R}_0)$ uniformly of size $(B_0,R_0)$ with $\mathcal{B}_0$ center of size $k_0$, then simply apply the above argument with sets $B_x = \{v \in \torus{N}d: d(x,v) < n\beta\delta''\}$ defined for 
$x \in \mathcal{B}_0^0$. 



\end{proof}

\end{document}